\documentclass[final,onefignum,onetabnum]{siamart190516}

\usepackage{braket,amsfonts}
\usepackage{url,titletoc,enumitem}

\usepackage{array}
\usepackage{bm}
\usepackage{mathrsfs}

\usepackage{tikz}
\usetikzlibrary{intersections,shapes.arrows}

\usepackage[caption=false]{subfig}
\usepackage{pgfplots}

\newsiamthm{claim}{Claim}
\newsiamremark{remark}{Remark}
\newsiamremark{hypothesis}{Hypothesis}
\crefname{hypothesis}{Hypothesis}{Hypotheses}

\usepackage{algorithmic}

\usepackage{graphicx,epstopdf}

\Crefname{ALC@unique}{Line}{Lines}

\usepackage{amsopn}

\usepackage{xspace}
\usepackage{bold-extra}
\usepackage[most]{tcolorbox}

\colorlet{texcscolor}{blue!50!black}
\colorlet{texemcolor}{red!70!black}
\colorlet{texpreamble}{red!70!black}
\colorlet{codebackground}{black!25!white!25}

\lstdefinestyle{siamlatex}{%
  style=tcblatex,
  texcsstyle=*\color{texcscolor},
  texcsstyle=[2]\color{texemcolor},
  keywordstyle=[2]\color{texemcolor},
  moretexcs={cref,Cref,maketitle,mathcal,text,headers,email,url},
}

\tcbset{%
  colframe=black!75!white!75,
  coltitle=white,
  colback=codebackground, 
  colbacklower=white, 
  fonttitle=\bfseries,
  arc=0pt,outer arc=0pt,
  top=1pt,bottom=1pt,left=1mm,right=1mm,middle=1mm,boxsep=1mm,
  leftrule=0.3mm,rightrule=0.3mm,toprule=0.3mm,bottomrule=0.3mm,
  listing options={style=siamlatex}
}

\newtcblisting[use counter=example]{example}[2][]{%
  title={Example~\thetcbcounter: #2},#1}

\newtcbinputlisting[use counter=example]{\examplefile}[3][]{%
  title={Example~\thetcbcounter: #2},listing file={#3},#1}

\DeclareTotalTCBox{\code}{ v O{} }
{ 
  fontupper=\ttfamily\color{black},
  nobeforeafter,
  tcbox raise base,
  colback=codebackground,colframe=white,
  top=0pt,bottom=0pt,left=0mm,right=0mm,
  leftrule=0pt,rightrule=0pt,toprule=0mm,bottomrule=0mm,
  boxsep=0.5mm,
  #2}{#1}

\patchcmd\newpage{\vfil}{}{}{}
\usepackage{extramath}

\def\beq{\begin{equation}}
\def\eeq{\end{equation}}

\def\bgam{ {\boldsymbol{\gamma}} }
\def\bbgam{ {\bar{\boldsymbol{\gamma}}} }

\def\am{\mathrm{am}}
\def\gm{\mathrm{gm}}

\numberwithin{equation}{section}

\def\spann{\mathrm{span}}

\newtheorem{exa}{Example}[section]

\def\beq{\begin{equation}}
\def\eeq{\end{equation}}

\def\epsavg{ {{\epsilon_{\mathrm{avg}}}}}

\def\cN{ {{\mathcal N}}}

\def\bbN{ {\mathbb N}}

\def\mathand{{\quad \quad \mathrm{and} \quad \quad}}

\def\R{ {\mathbb{R}} }
\def\K{ {\mathbb{K}} }
\def\C{ {\mathbb{C}} }

\def\cpd{{\cp{\bA,\bB,\bC}}}

\def\cE{ {\mathcal E} }

\def\scrL{{{\mathscr L}}}
\def\scrQ{{{\mathscr Q}}}
\def\scrZ{{{\mathscr Z}}}
\def\scrW{{{\mathscr W}}}

\def\cM{{{\mathcal M}}}
\def\cN{ {\mathcal N} }

\def\cS{{\mathcal S} }

\def\hcS{{\hat{\cS}}}

\def\cT{{\mathcal T}}

\def\cW{{\mathcal W}}

\def\bA{ {\mathbf{A}} }
\def\bB{ {\mathbf{B}} }
\def\bC{ {\mathbf{C}} }

\def\tbA{ {\tilde{\bA}} }
\def\tbB{ {\tilde{\bB}} }
\def\tbC{ {\tilde{\bC}} }

\def\bD{ {\mathbf{D}} }
\def\bE{ {\mathbf{E}} }

\def\bG{ {\mathbf{M}} }
\def\bH{ {\mathbf{N}} }
\def\bI{ {\mathbf{I}} }

\def\bL{ {\mathbf{L}} }
\def\bM{ {\mathbf{M}} }
\def\bN{ {\mathbf{N}} }
\def\bQ{ {\mathbf{Q}} }
\def\bP{ {\mathbf{P}} }

\def\hR{ {\hat{R}}  }
\def\bS{ {\mathbf{S}} }

\def\bT{ {\mathbf{T}} }

\def\bU{ {\mathbf{U}} }
\def\bV{ {\mathbf{V}} }

\def\hbV{ {\hat{\mathbf{V}}} }
\def\bW{ {\mathbf{W}} }
\def\bX{ {\mathbf{X}} }
\def\bY{ {\mathbf{Y}} }
\def\bZ{ {\mathbf{Z}} }

\def\ba{ {\vec{a}} }
\def\bb{ {\vec{b}} }
\def\bc{ {\vec{c}} }
\def\tba{ {\tilde{\ba}} }
\def\tbb{ {\tilde{\bb}} }
\def\tbc{ {\tilde{\bc}} }
\def\tbcjp{ {\tbc_j^\perp} }
\def\bq{ {{\vec{q}}} }
\def\bu{ {\vec{u}} }
\def\bv{ {\vec{v}} }
\def\bw{ {\vec{w}} }
\def\bx{ {\vec{x}} }
\def\by{ {\vec{y}} }
\def\bz{ {\vec{z}} }

\def\lam{ {\lambda} }

\def\blam{ {\vec{\lambda}} }
\def\bLam{ {\vec{\Lambda}} }

\def\bomega{ { \vec{\omega}} }

\newcommand{\tnorm}[1]{{\|{#1}\|_{2}} }
\newcommand{\spnorm}[1]{{\|{#1}\|_{\mathrm{sp}}} }

\newcommand{\NepsL}[1]{ {{\mathfrak{N}_\epsilon (\scrL_{{#1}})}} }
\newcommand{\NepsLal}[1]{ {{ \mathfrak{N}_{\alpha \epsilon} (\scrL_{{#1}})}} }
\newcommand{\NepsLzero}[1]{ {{\mathfrak{N}_{0} (\scrL_{{#1}})}} }
\newcommand{\NepsLone}[1]{ {{ \mathfrak{N}_{1 \epsilon} (\scrL_{{#1}})}} }

\def\b0{ {\vec{0}} }

\def\tcT{{\tilde{\cT}}}
\def\hcT{{\hat{\cT}}}
\def\hcW{{\hat{\cW}}}

\def\scriptsize{}
\def\rank{\mbox{rank}}

\def\cS{{\mathcal S}}

\newcommand{\df}[1]{{\bf{#1}}{\index{#1}}}

\def\mat{\partial^{\mathrm{mat}}}

\def\ran{\mathrm{ran \ }}

\def\sv{{\mathrm{sv}}}

\def\md{{\mathrm{md}}}

\def\be{ {\mathbf{e}} }

\begin{tcbverbatimwrite}{tmp_\jobname_header.tex}
\title{Guarantees for existence of a best canonical polyadic approximation of a noisy low-rank tensor\thanks{\funding{Research supported by: (1) Flemish Government: This work was supported by the Fonds de la Recherche Scientifique--FNRS and the Fonds Wetenschappelijk Onderzoek--Vlaanderen under EOS Project no 30468160 (SeLMA); (2) KU Leuven Internal Funds C16/15/059 and ID-N project no 3E190402. (3) This research received funding from the Flemish Government under the “Onderzoeksprogramma Artificiële Intelligentie (AI) Vlaanderen” programme.  (4) Work supported by Leuven Institute for Artificial Intelligence (Leuven.ai).}
}}

\author{Eric Evert\footnotemark[2] 
\and Lieven De Lathauwer\thanks{KU Leuven, Dept. of Electrical Engineering ESAT/STADIUS, Kasteelpark
Arenberg 10, bus 2446, B-3001 Leuven, Belgium and KU Leuven – Kulak, Group Science, Engineering and Technology, E.
Sabbelaan 53, B-8500 Kortrijk, Belgium
\newline \hbox{\quad} (Eric.Evert, Lieven.DeLathauwer)@kuleuven.be.}}

\headers{Existence of best low-rank CP approximations}{Eric Evert and Lieven De Lathauwer}
\end{tcbverbatimwrite}
\input{tmp_\jobname_header.tex}

\begin{document}
\maketitle

\begin{tcbverbatimwrite}{tmp_\jobname_abstract.tex}
\begin{abstract}
The canonical polyadic decomposition (CPD) of a low rank tensor plays a major role in data analysis and signal processing by allowing for unique recovery of underlying factors. However, it is well known that the low rank CPD approximation problem is ill-posed. That is, a tensor may fail to have a best rank $R$ CPD approximation when $R>1$. 

  This article gives deterministic bounds for the existence of best low rank tensor approximations over  $\K=\R$ or $\K=\C$.  More precisely, given a tensor $\cT \in \K^{I \times I \times I}$ of rank $R \leq I$, we compute the radius of a Frobenius norm ball centered at $\cT$ in which best $\K$-rank $R$ approximations are guaranteed to exist. In addition we show that every $\K$-rank $R$ tensor inside of this ball has a unique canonical polyadic decomposition. This neighborhood may be interpreted as a neighborhood of ``mathematical truth" in with CPD approximation and computation is well-posed.
  
  In pursuit of these bounds, we describe low rank tensor decomposition as a ``joint generalized eigenvalue" problem. Using this framework, we show that, under mild assumptions, a low rank tensor which has rank strictly greater than border rank is defective in the sense of algebraic and geometric multiplicities for joint generalized eigenvalues. Bounds for existence of best low rank approximations are then obtained by establishing perturbation theoretic results for the joint generalized eigenvalue problem. In this way we establish a connection between existence of best low rank approximations and the tensor spectral norm. In addition we solve a ``tensor Procrustes problem" which examines orthogonal compressions for pairs of tensors.
  
  The main results of the article are illustrated by a variety of numerical experiments.

\end{abstract}


\end{tcbverbatimwrite}
\input{tmp_\jobname_abstract.tex}


\section{Introduction}

Tensors, i.e. multiindexed arrays, are natural generalizations of matrices and are commonplace  in fields such as machine learning and signal processing where data sets often have inherent higher-order structure \cite{Setal17,Cetal15}. Tensor decompositions are powerful tools which can be used either to recover source signals responsible for measured data or to compress data \cite{CJ10}. 
	
One of the most popular decompositions for tensors is the canonical polyadic decomposition (CPD) which expresses a given tensor as a sum of {\it rank one tensors}. A key property of this decomposition is that, under mild conditions, the CPD of a low rank tensor is essentially unique \cite{DD13,K77,SB00}. This essential uniqueness makes CPD an indispensable tool in applications such as indpendent compenent analysis where one wants to recover source signals from measured data \cite{KB09,DDV00a}. In the case that the data has a natural tensor structure, CPD can allow for unique recovery of the source signals. 

In applications one typically works with a measurement of a {\it low rank tensor}, that is, a tensor which is representable by a small number of source signals. This measurement is  always corrupted by measurement error, thus the measurement itself does not truly have low rank. As a consequence, in order to obtain the signals which represent the tensor, a practitioner must compute a low rank approximation to the measured tensor. 

Although computation of a best low rank approximation is a natural problem, it is not necessarily a well-posed problem. In particular, an arbitrary tensor may not have a best rank (less than or equal to) $R$ approximation when $R>1$ \cite{KHL89,SL08}. Furthermore, even if a best rank $\R$ approximation exists, it can fail to be unique. In fact, when considering CPD over $\R$, nonexistence of a best low rank approximation occurs with positive probability \cite{S06,TKD88}. In contrast, it has recently been shown that  a unique best rank $R$ approximation exists with probability one when working over $\C$ \cite{QML19}. See \cite{FO14,FS15} for further discussion of uniqueness of best approximations.\footnote{Note that uniqueness of a best rank $R$ approximation is a separate issue from uniqueness of the CPD of that approximation. }

The issue of potential nonexistence of best low rank approximations is not one which should be taken lightly. Should a tensor fail to have a best low rank approximation, one could of course choose to simply take a near optimal low rank approximation. However, when a best low rank approximation fails to exist, near optimal low rank approximations exhibit a phenomena known as {\it diverging components} \cite{KDS08,S08}. Roughly, the optimal way to represent a tensor which does not have a best low rank approximation is to express the tensor as destructive interference between signals which, in the limit, have infinite magnitude. Stated in an alternative fashion, computation of low rank approximations is unstable in a neighborhood around a tensor which does not have a best low rank approximation \cite{V17}. 

Although computation of a best rank $R$ approximation can suffer from issues of nonexistence, tensor approximations typically work surprisingly well in practice. The reason for this is closely linked to the low rank structure of tensors which naturally arise in applications. In this article we explain this connection and quantify how much error can be present in the measurement of a low rank tensor before one encounters issues of nonexistence. Additionally, we further develop the perspective that low rank tensor approximation is in fact well-posed in a neighborhood of ``mathematical truth" for tensors occurring in practical applications. In particular, we provide deterministic bounds for the existence of best low rank approximations for tensors of order three.

More precisely, given a low rank tensor of order three, we compute a deterministic lower bound for the radius of a Frobenius norm ball centered at the tensor in which best low rank approximations are guaranteed to exist and in which every rank $R$ tensor has a unique CPD. While \cite[Theorem 5.1]{AGHKT14} gives a probabilistic lower bound for the radius of such a neighborhood which may be applied to a restricted class of rank $R$ symmetric tensors, the computation of a deterministic lower bound which may be applied to a large collection of rank $R$ tensors is novel.

 To obtain these bounds, we develop perturbation theoretic bounds for a ``joint generalized eigenvalue decomposition" which, with mild assumptions, is equivalent to the CPD of a low rank tensor. In particular, we show that low border rank tensors which have rank strictly greater than border rank can be identified with matrix tuples which are defective in the sense of algebraic and geometric multiplicities for the joint generalized eigenvalue problem. Given a low rank tensor, we then compute lower bounds on the distance to a tensor which is defective in this sense.

\subsection{Brief statement of results}

Given a tensor $\cT \in \K^{R \times R \times K}$, we call a nonzero vector $\bx \in \K^R$ a \df{joint generalized eigenvector} (JGE vector) of $\cT$ if there exists a $\blam \in \K^K$ and a nonzero vector $\by \in \K^R$ such that
\beq
\label{eq:JGEVdefIntro}
\bT(:,:,\ell) \bx = \lam_\ell \by
\eeq
for all $\ell = 1, \dots, K$, and we call $\spann(\blam)$ the \df{joint generalized eigenvalue} of $\cT$ corresponding to $\bx$. In Section \ref{sec:JGEmultdefs} we formally define appropriate notions of characteristic polynomials and algebraic multiplicity for JGE vectors. For now we roughly define the \df{algebraic multiplicity} of $\spann(\blam)$ to be the degree to which $\blam$ ``divides" the characteristic polynomial of $\cT$, and we define the \df{geometric multiplicity} of $\spann(\blam)$ to be the number of linearly independent JGE vectors corresponding to $\spann(\blam)$. 
 
  \begin{theorem}
\label{theorem:CPDiffJGEVDSimple}
Let $\cT \in \K^{R \times R \times K}$ and assume there is some vector $\bv \in \K^K$ such that $\cT \cdot_3 \bv$ is invertible. Then $\cT$ has $\K$-rank $R$ if and only if $\cT$ has a basis of JGE eigenvectors.

Furthermore, if $\cS \in \K^{R \times R \times K}$ is a border $\K$-rank $R$ tensor which has $\K$-rank strictly greater than $R$ and there is a vector $\bu \in \K^K$ such that $\cS \cdot_3 \bu$ is invertible, then $\cS$ has a joint generalized eigenvalue with algebraic multiplicity strictly greater than geometric multiplicity. 
\end{theorem}

 Theorem \ref{theorem:CPDiffJGEVDSimple} is a simplified statement of the upcoming Theorem \ref{theorem:rankvsmult}. We remark that Theorem \ref{theorem:CPDiffJGEVDSimple} can be viewed as a reformulation of algebraic-geometric tensor results in the language of the joint generalized eigenvalue problem, e.g. see \cite{Stra83,LM17}. While this result is known in a different language, the reformulation in terms of joint generalized eigenvalues illustrates the connection between perturbation theory for joint generalized eigenvalues and existence of best rank $R$ tensor approximations. Using this connection we give computable deterministic guarantees for the existence of best low rank tensor approximations.
 
\begin{theorem}
\label{theorem:MultiplePencilBoundSimple}
Let $\cT \in \K^{R \times R \times K}$ be a tensor of $\K$-rank $R$ and assume there is some vector $\bv \in \K^K$ such that $\cT \cdot_3 \bv$ is invertible. For each $k=1,\dots, \lfloor K/2 \rfloor$, let $[\![\bA_k,\bB_k,\bC_k]\!]$ be a CPD of the matrix subpencil $(\cT(:,:,2k-1),\cT(:,:,2k))$ of $\cT$ where the columns of $\bC_k$ have unit Euclidean norm. Define 
\[
\epsilon_k := \frac{\sigma_{\min} (\bA_k) \sigma_{\min} (\bB_k) \min_{i \neq j} \chi (\bC_k(:,i), \bC_k(:,j))}{2} \qquad \mathrm{for \ } k=1,\dots,\lfloor K/2 \rfloor
\]
and set $\epsilon = \|(\epsilon_1,\dots,\epsilon_{\lfloor K/2 \rfloor})\|_2$. If $\cW \in \K^{R \times R \times K}$ satisfies $\|\cW-\cT\|_{\mathrm{F}} < \epsilon/2$ then $\cW$ has a best $\K$-rank $R$ approximation. Furthermore, any best $\K$-rank $R$ approximation of $\cW$ has a unique CPD, and, with probability one, the best $\K$-rank $R$ approximation of $\cW$ is unique. Here $\chi(\cdot,\cdot)$ denotes the chordal distance between unit vectors and $\sigma_{\min} (\cdot)$ denotes the smallest (possibly zero) singular value of a matrix.
\end{theorem}

The full statement of this result, Theorem \ref{theorem:MultiplePencilBound}, allows for more general collections of subpencils to be considered than those considered above. Additionally, a slight modification of this result may be applied to a measured tensor $\cM'$ in an attempt to guarantee that $\cM'$ has a best rank $R$ approximation, see Theorem \ref{theorem:MLSVDExistenceBound}.

\subsection{An overview of our approach}

The primary goal in this article is to give deterministic bounds for the existence of best low rank tensor approximations. A simple yet important idea for our approach is the following: If $\hcT'$ is a best border rank $R$ approximation of $\cM' = \cT' + \cN' \in \K^{I_1 \times I_2 \times I_3}$ where $\cT'$ has rank $R$, then the Frobenius distance from $\cT'$ to $\hcT'$ cannot exceed $2 \|\cN'\|_\mathrm{F}$. Therefore, to show that every tensor in a ball of radius $\epsilon/2$ around $\cT'$ has a best rank $R$ approximation, it is sufficient to show that every border rank $R$ tensor in a $\epsilon$ radius ball around $\cT'$ has rank $R$. 

We temporarily restrict our discussion to the case of $R \times R \times K$ tensors which are ``slice mix invertible". In this setting, using the language of joint generalized eigenvalues, Theorem \ref{theorem:rankvsmult} shows that slice mix invertible tensors of size $R \times R \times K$ with border rank $R$ and rank strictly greater than $R$ have a joint generalized eigenvalue with algebraic multiplicity strictly greater than geometric multiplicity. We call such tensors defective. These defective tensors can be seen as a natural extension to the tensor setting of matrices which have a nontrivial Jordan block.

From this point our goal is to guarantee that a best border rank $R$ approximation is nondefective, hence has rank equal to border rank. Mirroring the matrix case, a natural sufficient condition for a tensor to be nondefective is that the tensor has distinct joint generalized eigenvalues. This motivates the study and development of perturbation theoretic bounds for the joint generalized eigenvalue problem, see Theorems \ref{theorem:TensMDBound} and \ref{theorem:TensBauerFike}. 

Theorem \ref{theorem:TensMDBound} in particular may be used to guarantee that a perturbation of a matrix pencil is nondefective. By applying this result to a collection of subpencils of the signal tensor, we obtain a full tensor based bound which may be used to guarantee the existence of a best rank $R$ approximation of the measured tensor, see Theorems \ref{theorem:MultiplePencilBoundSimple}, \ref{theorem:MultiplePencilBound}, and \ref{theorem:MLSVDExistenceBound}.

The more general setting where $\cM' = \cT' +\cN'$ has size $I_1 \times I_2 \times I_3$ and $\cT'$ has rank $R \leq \min \{I_1,I_2\}$ is treated by a reduction to the $R \times R \times K$ setting. As above let $\hcT'$ be a best border rank $R$ approximation of $\cM'=\cT'+\cN'$. Numerically, we wish to work with orthogonal compressions of $\cT'$ and $\hcT'$, as this puts us in the $R \times R \times K$ setting discussed above. However, there can be adverse effects if one independently computes the compressions of $\cT'$ and $\hcT'$. 

Theorem \ref{theorem:TensorProcrustes} discusses how to orthogonally compress a \textit{pair}\footnote{Computing an orthogonal compression or a low multilinear rank approximation of a \textit{single} tensor is a standard preprocessing in many applications, e.g., see \cite{Setal17}. } of tensors and shows that there exist orthogonal compressions $\cT$ and $\hcT$ of $\cT'$ and $\hcT'$, respectively, so that
\beq
\label{eq:ProcrustesEqIntro}
\|\cT-\hcT\|_\mathrm{F} \leq \|\cT'-\hcT'\|_\mathrm{F}.
\eeq
From this point we may use our results in the $R \times R \times K$ setting to guarantee that $\hcT$ has rank $R$, hence $\hcT'$ has rank $R$.

\subsection{Connections to existing literature}

We now briefly discuss the relationship between our results and existing literature.

\subsubsection{Constrained existence results}

While our results  are novel in providing computable deterministic bounds for the existence of best unconstrained low rank tensor approximations, there are several known existence results in the setting of \textit{constrained} tensor decompositions. For example, \cite[Theorem 25]{LC14} shows that if one places constraints on the coherence of the factor matrices for the approximation of $\cM'$, then the constrained best low rank approximation problem always has a minimizer. As a special case, this result implies that best low rank approximations exist if one restricts to considering factor matrices which have orthogonal columns. Also see \cite{Hars1984d,Setal12} and references therein.

 Another commonly considered constraint is nonnegativity. \cite[Theorem 6.1]{LC09}, see also \cite{QCL16}, shows that for nonnegative tensors, best low nonnegative rank approximations always exist. That is, a best low rank approximation exists if one constrains their factors to be nonnegative.

As previously mentioned, it is in general known that if $\cM'$ does not have a best rank $R$ approximation then, near optimal approximations exhibit diverging components. Thus if one places constraints on the factors which prevent diverging components from occurring, then a best approximation with respect to those conditions will always exist. While in some applications it may be natural to consider such constraints, it is important to separately consider the \textit{unconstrained} problem, as constraints will not always be applicable. Additionally, as discussed in \cite{LC14}, imposing constraints to ``fix" ill-posedness may be misleading, as the problem at hand may in fact be ill-posed. For example, CPD may not be a proper data model for the problem, or the rank in consideration may be too low.

\subsubsection{Existence in tubular neighborhoods}

In the language of algebraic geometry, it is well known that given a smooth point on an algebraic variety, there is an open neighborhood around that point in which unique best approximations to the variety exist, e.g. see \cite{FS13,Aba80,Sot16,Lan12}. Translating to this terminology, the ``simple" (see Section \ref{sec:JGEdefs}) rank $R$ tensors we consider are smooth points on the $R$th secant variety of the Segre variety, thus there is necessarily some neighborhood around a simple rank $R$ tensor in which best rank $R$ approximations exist. This leads to a so called \textit{tubular neighborhood} around the set of simple rank $R$ tensors in which best rank $R$ approximations are guaranteed to exist.

 Though it was already known that there is  \textit{some neighborhood} around a given simple rank $R$ tensor in which best rank $R$ approximations are guaranteed to exist, until now there has been no known method for computing a lower bound for the \textit{radius} of such a neighborhood. That is, until now there has been no known method for providing a \textit{deterministic bound on noise levels} which may be used to guarantee that a measured tensor has a best rank $R$ approximation. 

\subsubsection{The joint eigenvalue problem}

A problem which is closely related to our joint generalized eigenvalue is the \textit{joint eigenvalue problem} where one studies joint eigenvalues of a tuple of commuting matrices. In particular, by performing a modal multiplication in the first or second mode, a rank $R$ slice mix invertible tensor $\cT \in \K^{R \times R \times K}$ can be transformed to a tensor whose frontal slices form a collection of commuting matrices. 

There are indeed results in literature for the joint eigenvalue problem which may be used to show that if $(\bT_1,\dots,\bT_k) \in (\K^{R \times R})^K \cong \K^{R \times R \times K}$ is a simultaneously diagonalizable collection of matrices whose span contains a rank $R$ matrix, then there exists \textit{some neighborhood} around this tuple in which \textit{other tuples of commuting matrices} are simultaneously diagonalizable, e.g., see \cite[Theorem 3.5]{KP02}; however, the results in this direction that we are aware of do not provide a method for computing a lower bound for the \textit{radius} of such a neighborhood. Furthermore, though there is much literature, e.g. see \cite{KSS12}, for the joint eigenvalue problem, there are two significant shortcomings in attempting to directly apply this literature to our setting. 

First, we essentially wish to determine if some tensor $\hcT \in \K^{R \times R \times K}$ is diagonalizable in the sense of the joint generalized eigenvalue problem based on its distance to a slice mix invertible tensor $\cT$ which itself is diagonalizable in the sense of the joint generalized eigenvalue problem.  Here $\hcT$ is thought of as a best low rank approximation to some noisy observed tensor $\cT+\cN$ for which $\cT$ is the signal portion. However, \textit{a transformation which makes the frontal slices of $\cT$ commute may not make the frontal slices of $\hcT$ commute and there may exist no invertible transformation which simultaneously converts both the frontal slices of $\cT$ and the frontal slices of $\hcT$ to tuples of commuting matrices}. Thus, many tensors which can be treated with our results are missed by converting to the commuting case.

Second, treating a generalized eigenvalue problem\footnote{If $K=2$, then our notions of joint generalized eigenvectors and joint generalized eigenvalues coincide exactly with the (classical) notions of generalized eigenvalues and generalized eigenvectors as discussed in \cite{SS90}.} by transforming it to a classical eigenvalue problem can  have adverse effects on stability. This is illustrated by the discussion in \cite[Section 7.7.1]{GV96}. These difficulties motivate studying perturbation theory for the joint generalized eigenvalue problem in its own right.

\subsubsection{Perturbation theory for the generalized eigenvalue problem}

One of our main results, Theorem \ref{theorem:TensMDBound}, gives a lower bound on the radius of a neighborhood around a given diagonalizable matrix pencil in which all matrix pencils are diagonalizable. This result extends the ``spectral variation" bound of \cite[Chapter VI, Theorem 2.7]{SS90} to the much stronger notion of a ``matching distance" bound. Additionally, our extension is given in the tensor spectral norm. The proof of this extension may be unsurprising to experts in perturbation theory for classical generalized eigenvalues; however, the authors are unaware of an equivalent  extension in the literature. 

In general the literature on perturbation theory for generalized eigenvalues and generalized eigenvectors is rich, for example see \cite{GV96,Naka11,Cang03}. We note that other bounds presented in the literature often use information from both the original matrix pencil and its perturbation. While such bounds can provide helpful insight into how generalized eigenvalues are affected by a specific perturbation, they are not well-suited to our needs.

\subsection{Organization}

Section \ref{sec:notationdefinitions} formally introduces our notation and definitions and gives formal statements of our main results. In Section \ref{sec:JGEbasics} we discuss basic theory for the joint generalized eigenvalue problem. Section \ref{sec:perturbation} discusses perturbation theoretic results for the joint generalized eigenvalue problem, while Section \ref{sec:Bounds} gives deterministic bounds for existence of a best low rank tensor approximation using results from Sections \ref{sec:JGEbasics} and \ref{sec:perturbation}. Numerical experiments illustrating our results are given in Section \ref{sec:figures}. 

\section{Notation, Definitions, and Main Results}
\label{sec:notationdefinitions}

We now give our basic definitions and notations as well as formal statements of the main results of the article. Let $\K$ denote $\R$ or $\C$. We denote scalars, vectors, matrices, and tensors with entries in $\K$ by lower case $(m)$, bold lower case $(\vec{m})$, bold upper case $(\bM)$, and calligraphic script $(\cM)$, respectively. Additionally, script font $(\mathscr{M})$ is used to denote a subspace of $\K^I$. Given a matrix $\bM \in \K^{I_1 \times I_2}$, let $\bM^\mathrm{T},\bM^\mathrm{H},\bM^\dagger$ denote the transpose, Hermitian, and Moore-Penrose pseudo inverse of $\bM$, respectively.  If $\bM$ is invertible, let $\bM^{-1}$ and $\bM^{-\mathrm{T}}$ denote the inverse and inverse transpose of $\bM$. Additionally let $\|\bM\|_2$ and $\|\bM\|_\mathrm{F}$ denote the spectral and Frobenius norms of $\bM$. 

Let $\langle \cdot, \cdot \rangle$ denote the usual inner product on $\K^R$ which is conjugate linear in the second term. We make frequent use of the linear form
\[
\langle \blam, \bbgam \rangle = \lambda_1 \gamma_1 + \lambda_2 \gamma_2 +\cdots + \lambda_K \gamma_K.
\]
Here $\blam,\bgam \in \K^K$ and $\bbgam$ denotes the complex conjugate of $\bgam$. Note that $\langle \cdot,\cdot \rangle$ denotes the usual inner product on $\K^K$ which is conjugate linear in the second term.\footnote{The linear form we are interested in is linear in $\bgam$ while the inner product $\langle \cdot,\cdot \rangle$ is conjugate linear in the second argument. Hence, the complex conjugate of $\bgam$ appears in the expression $\langle \blam, \bbgam \rangle$. We use the notation $\langle \cdot,\cdot \rangle$ as we often consider some geometric relationship between $\blam$ and $\bgam$.}

A \df{tensor} is a multiindexed array $\cM'$ with entries in $\K$. In this article our primary interest lies with tensors of order $3$. That is, we study tensors $\cM' \in \K^{I_1 \times I_2 \times I_3}$.  Often we will restrict our study to low rank (pencil-like) tensors having size $R \times R \times K$ and rank $R$. To the reader keep track of context, we label tensors of size $I_1 \times I_2 \times I_3$ with the symbol $'$, while tensors of size $R \times R \times K$ not labelled in this way. So for example, $\cM' \in \K^{I_1 \times I_2 \times I_3}$ while $\cM \in \K^{R \times R \times K}$.

Define the \df{outer product} of vectors $\ba \in \K^{I_1}, \bb \in \K^{I_2}$ and  $\bc \in \K^{I_3}$, denoted $\ba \otimes \bb \otimes \bc$, to be the $I_1 \times I_2 \times I_3$ tensor whose $i,j,k$th entry is given by
\[
(\ba \otimes \bb \otimes \bc)_{ijk}= a_i b_j c_k.
\]
A tensor which can be expressed as the outer product of nonzero vectors is called a \df{rank one tensor}.

Every tensor may be expressed as a sum of rank one tensors. That is, given a tensor $\cM' \in \K^{I_1 \times I_2 \times I_3}$ there exists a collection of vector tuples  $\{(\ba_r, \bb_r, \bc_r) \}_{r=1}^R \subset \K^{I_1} \times \K^{I_2} \times \K^{I_3}$ such that
\beq
\cM' = \sum_{r=1}^R \ba_r  \otimes \bb_r \otimes \bc_r.
\label{eq:CPDdef}
\eeq
In the case that $R$ is as small as possible, the tensor $\cM'$ is said to have $\K$-\df{rank} $R$, and the decomposition in equation \eqref{eq:CPDdef} is called a \df{canonical polyadic decomposition} (CPD) of $\cM'$. We let $\rank_\K (\cM')$ denote the rank of $\cM'$ over $\K$. A fundamental question in the study of tensors and in many applications is to determine the CPD of a given tensor over either $\R$ or $\C$. We emphasize that the field in which the entries of the tensor lie may differ from the field over which the decomposition is taken. In particular, if $ \cM'$ is a tensor with real entries, then both real and complex rank are interesting to consider, and these ranks are not necessarily equal, e.g. see \cite{BK99,B13,K89}.

Let $\bA \in \K^{I_1 \times R}$ be the matrix whose $r$th column is the vector $\ba_r$, and similarly define $\bB \in \K^{I_2 \times R}$ and $\bC \in \K^{I_3 \times R}$. Write 
\[
\cM'= \cpd
\]
as a compact notation for the CPD of $\cM'$. The matrices $\bA,\bB,$ and $\bC$ are called the \df{factor matrices} of $\cM'$. As discussed in \cite{Setal17} one has 
\[
\cM'(:,:,k) = \bA D_k (\bC) \bB^\mathrm{T} \qquad \mathrm{for \ all \ } k=1,\dots, I_3.
\]
Here $D_k (\bC) \in \K^{R \times R}$ is the diagonal matrix whose diagonal entries are given by the $k$th row of $\bC$. 

Say a CPD $\cM'=\cpd$ is \df{unique} if for any other CPD $\cM'=[\![\tilde\bA,\tilde\bB,\tilde\bC]\!]$, there exist a permutation matrix $\Pi$ and invertible diagonal matrices $\bD_\bA, \bD_\bB,\bD_\bC$ such that 
\[
\bA \bD_\bA \Pi = \tilde\bA \mathand  \bB \bD_\bB \Pi = \tilde\bB \mathand  \bC \bD_\bC \Pi = \tilde\bC.
\]
That is, a CPD is unique if the only indeterminacies in the CPD are permutation and scaling/counter scaling of the columns of the factor matrices.

Central to the potential nonexistence of a best low rank tensor approximation is the fact that the set of tensors of rank less than or equal to $R$ is, in general, not a closed set \cite{KHL89,SL08}. For example, there are tensors having rank $3$ which are a limit of rank $2$ tensors. Of course, such a tensor cannot have a best rank $2$ approximation. Say a $\cM'$ tensor has \df{border $\K$-rank $R$} if $\cM'$ is a limit of $\K$-rank $R$ tensors and there is no integer $L<R$ such that $\cM'$ is a limit of $\K$-rank $L$ tensors. That is, $R$ should be the smallest integer such that $\cM'$ is a limit of $\K$-rank $R$ tensors. We let $\K$-$\underline{\rank} (\cM')$ denote the border $\K$-rank of $\cM'$. 

We will make frequent use of the fact that the set of border $\K$-rank $R$ tensors is a closed set. This in particular implies that a tensor always has a (not necessarily unique) best border $\K$-rank $R$ approximation.

\subsection{Approximately low rank tensors}
\label{sec:Compress}
 In this article we are mainly interested in tensors which have (approximately) low rank in the following sense. Stated briefly, for tensors $\cM' \in \K^{I_1 \times I_2 \times I_3}$ occuring in applications, one often has $\cM' = \cT'+\cN'$ where $\cT'$ is a rank $R$ tensor with rank equal to multilinear rank and $\cN'$ is noise. That is, $\cM'$ is a noisy measurement of low rank signal tensor $\cT'$. Our main result, Theorem \ref{theorem:MultiplePencilBound}, uses information from $\cT'$ to compute an upper bound $\epsilon$ on the Frobenius norm of the noise $\|\cN'\|_\mathrm{F}$ such that $\cM'$ is guaranteed to have a best rank $R$ approximation if $\|\cN'\|_\mathrm{F} < \epsilon$. A detailed explanation of our perspective follows.

Given a tensor $\cM' \in \K^{I_1 \times I_2 \times I_3}$, define
\beq
\label{eq:MLrankDef}
\begin{array}{ccc}
R_1 (\cM'):= \dim \spann \{\cM'(:,k,\ell)\}_{\forall k,\ell}, \\
R_2 (\cM'):= \dim \spann \{\cM'(j,:,\ell)\}_{\forall j,\ell}, \\
R_3 (\cM'):= \dim \spann \{\cM'(j,k,:)\}_{\forall j,k}. \\
\end{array}
\eeq
The integer triple $(R_1 (\cM'), R_2 (\cM'), R_3 (\cM'))$ is called the \df{multilinear rank} of $\cM'$. In the case where $R_i (\cM') = I_i$ for each $i=1,2,3$ we say $\cM'$ has \df{full multilinear rank}. The vectors in equation \eqref{eq:MLrankDef} used to define each $R_i (\cM')$ are called the \df{mode-$i$ fibers} of $\cM'$, and we let $\cM'_{(i)}$ denote a matrix whose columns are equal to the mode-$i$ fibers of $\cM'$. The matrix $\cM'_{(i)}$ is called a \df{mode-$i$ unfolding} of $\cM'$.\footnote{There is freedom in the ordering of the columns of $\cM'_{(i)}$, so in this sense we have not well-defined $\cM'_{(i)}$. However, the properties of mode-$i$ unfoldings that we are most interested in are invariant under permutations of columns, so this ambiguity makes little difference for our purpose so long as one is consistent in the ordering of columns.}

It is not difficult to show that one always has 
\beq
\label{eq:RankVsMLRank}
\rank_\K (\cM') \geq \max \{R_1 (\cM'), R_2 (\cM'), R_3 (\cM')\}.
\eeq
For generic tensors this inequality is strict; however, for the signal portion of many tensors occurring in applications we have equality. In particular, in applications the tensor $\cM'$ is often a noisy measurement of some low rank signal tensor $\cT'$ which satisfies  $\rank_{\K} (\cT') \leq \min\{I_1,I_2,I_3\}$ and, up to a permutation of indices, 
\beq
\label{eq:LowRankNotion}
R=\rank_\K (\cT') = R_{1} (\cT') = R_{2} (\cT') \geq R_{3} (\cT')=K.
\eeq
That is, we think of $\cM'$ as $\cM'=\cT'+\cN'$ where $\cT'$ has $\K$-rank $R$ and multilinear rank $(R,R,K)$ and $\cN'$ is noise. Note that our results still hold if the multilinear rank of $\cT'$ in one mode is less than $R$, and allowing this increases flexibility, hence $R_3 (\cT')=K$ is allowed to be less than or equal to $R$.

For later use we note that the multilinear ranks of a tensor not only lower bound the rank of the tensor, but also the border rank of the tensor. That is, one has
\beq
\label{eq:BRankVsMLRank}
\K\text{-}\underline{\rank} (\cM') \geq \max \{R_1 (\cM'), R_2 (\cM'), R_3 (\cM')\}.
\eeq
This fact follows from lower semi-continuity of matrix rank, as tensors in a sufficiently small neighborhood around $\cM'$ must have multilinear rank greater than or equal to that of $\cM'$. It follows that the multilinear rank of $\cM'$ provides a lower bound for the rank of tensors in a neighborhood of $\cM'$.

\subsection{Orthogonal compressions and modal products}
In the case where $R_i(\cT') < I_i$ for some $i$ the tensor $\cT'$ maybe be orthogonally compressed to a tensor $\cT\in \K^{R_1 (\cT') \times R_2 (\cT') \times R_3 (\cT')}$. In particular, for $i=1,2,3$ let $\bV_i \in \K^{I_i \times R_i (\cT')}$ be a matrix whose columns form an orthonormal basis for the span of the mode-$i$ fibers of $\cT'$. Then the tensor 
\beq
\label{eq:orthcompr}
\cT' = \cT' \cdot_1 \bV_1^\mathrm{H} \cdot_2 \bV_2^\mathrm{H} \cdot_3 \bV_3^\mathrm{H} \in \K^{R_1 (\cT') \times R_2 (\cT') \times R_3 (\cT')}
\eeq
contains the same algebraic and geometric information as $\cT$, and we call this tensor an \df{orthogonal compression} of $\cT'$. One may recover $\cT'$ from $\cT$ using 
\beq
\label{eq:orthrecover}
\cT'=(\cT' \cdot_1 \bV_1^\mathrm{H} \cdot_2 \bV_2^\mathrm{H} \cdot_3 \bV_3^\mathrm{H}) \cdot_1 \bV_1 \cdot_2 \bV_2  \cdot_3 \bV_3.
\eeq

In equations \eqref{eq:orthcompr} and \eqref{eq:orthrecover} above, the product $\cT' \cdot_1 \bV_1^\mathrm{H}$ denotes the \df{mode-1 product} between the tensor $\cT'$ and the matrix $\bV_1^\mathrm{H}$. In general, the mode-1 product between a tensor $\cM' \in \K^{I_1 \times I_2 \times I_3}$ and a matrix $\bU \in \K^{J_1 \times I_1}$ is the tensor $\cM' \cdot_1 \bU \in \K^{J_1 \times I_2 \times I_3}$ with mode-$1$ fibers given by
\[
(\cM' \cdot_1 \bU)(:,k,\ell) = \bU(\cM' (:,k,\ell)) \qquad \mathrm{for \ all \ } k,\ell. 
\]
The mode-$2$ and mode-$3$ products $\cdot_2$ and $\cdot_3$ are defined analogously.

A common first step in CPD computation in applications is to compute\footnote{To compute an orthogonal compression or an (approximately) best low multilinear rank approximation of a tensor it is common to take the truncated core of a multilinear singular value decomposition of the tensor \cite{DDV00b}.}  the core of a multilinear $(R_1 (\cT'), R_2 (\cT'), R_3 (\cT'))$ rank approximation to $\cM'=\cT'+\cN' \in \K^{I_1 \times I_2 \times I_3}$, and it will be convenient for us to have notation for such tensors. Recalling that we typically assume $\cT'$ has multilinear rank $R \times R \times K$, we let $\cW'$ denote a tensor of size $I_1 \times I_2 \times I_3$ which has multilinear rank $(R,R,K)$ and we let $\cW \in \K^{R \times R \times K}$ be an orthogonal compression of $\cW'$. In particular, $\cW$ is thought of as the core of a low multilinear rank approximation to $\cM'$, while $\cW'$ itself is thought of as a low multilinear rank approximation to $\cM'$.

\subsubsection{Implicit subpencils of tensors}

We frequently consider subpencils of a given tensor. That is, we will consider pairs of matrices which are formed from the frontal slices of a given tensor. The subpencils considered in this article differ slightly from the standard terminology in that we allow subpencils which are ``revealed" by an orthogonal change of basis in the third mode of the tensor. We call such a subpencil an \df{implicit subpencil}.

Defined precisely, an implicit subpencil of $ \cS$ is an $R \times R \times 2$ tensor of the form $\cS = \cW \cdot_3 \bV \in \K^{R \times R \times 2}$ where $\bV \in \K^{2 \times K}$ has rank $2$ and has orthonormal rows. Since we always work with implicit subpencils, we drop the distinction ``implicit" in the remainder of the article.

\subsection{The joint generalized eigenvalue problem}
\label{sec:JGEdefs} 

 We now discuss the joint generalized eigenvalue problem in greater detail. Note that we will interchangeably view $\cT \in \K^{R \times R \times K}$ as either a tensor or a $K$-tuple of $R \times R$ matrices tuple via the map 
\[
\cT \to (\cT(:,:,1), \cT(:,:,2), \cdots ,\cT(:,:,K)). 
\]
For notational convenience we denote $\bT_k=\cT(:,:,k)$. 

Recall that given a tensor (or a matrix tuple) $\cT \in \K^{R \times R \times K}$, we call a nonzero vector $\bx \in \K^R$ a joint generalized eigenvector of $\cT$ if there exists a $\blam \in \K^K$ and a nonzero vector $\by \in \K^R$ such that
\beq
\label{eq:JGEVdef}
\bT_\ell \bx = \lam_\ell \by
\eeq
for all $\ell = 1, \dots, K$. Note that assuming $\by \neq \b0$ implies that $\lam_\ell=0$ if and only if $\bT_\ell \bx = \b0$.

There are natural scaling indeterminates in equation \eqref{eq:JGEVdef}. In particular, if the triple $(\bx,\blam,\by)$ is a solution to equation \eqref{eq:JGEVdef}, then so is $(\alpha \bx,\beta \blam, \gamma \by)$ for any nonzero scalars $\alpha,\beta,\gamma \in \K$ satisfying $\alpha  = \beta \gamma$. For this reason, if $(\blam,\bx)$ is a solution to \eqref{eq:JGEVdef}, then we say the one-dimensional subspace $\scrL:= \spann (\blam) \subset \K^K$ is a \df{joint generalized eigenvalue} (JGE value) of $\cT$ and that $(\scrL,\bx)$ is a \df{joint generalized eigenpair} of $\cT$. We typically use $\scrL \subset \K^K$ to denote a JGE value, while $\blam \in \K^K$ typically denotes a vector of unit norm such that $\spann(\blam)=\scrL$. Say a tensor has a \df{joint generalized eigenbasis} if its joint generalized eigenvectors form a basis for $\K^R$. We call the set of JGE values of a tensor the \df{(joint generalized) spectrum} of the tensor. An equivalent formulation of joint generalized eigenvectors is used in \cite{ZLP18} for blind source extraction.

We note that some authors have considered other notions of eigenvalues and eigenvectors for a tensor. These are not equivalent to and should not be confused with joint generalized eigenvalues. The eigenvalues and eigenvectors considered by other authors are intended to allow a spectral theory for general tensors. See \cite{S16,L05,QCC18}.

\subsection{Multiplicities of JGE values}
\label{sec:JGEmultdefs}
One of the main themes of this article is that low border rank tensors which have rank strictly greater than border rank can be identified with matrix tuples which are defective in the sense of algebraic and geometric multiplicities for joint generalized eigenvalues. These multiplicities we now define. To simplify the exposition, we only consider tensors of size $R \times R \times K$ when discussing multiplicities. While it is possible to define multiplicities of JGE values for general third order tensors, certain pathologies arise in the $I_1 \times I_2 \times I_3$ case which must be treated with care. For our purposes, it will not be necessary to consider multiplicities in the $I_1 \times I_2 \times I_3$ setting.

Given a tensor $\cT \in \K^{R \times R \times K}$, define the polynomial $p_\cT (\gamma)$ in the indeterminate $\bgam=(\gamma_1,\dots, \gamma_K)$ by
\[
p_\cT (\bgam):=\det \left( \sum_{k=1}^K \gamma_k \bT_k \right).
\]
We call $p_\cT (\bgam)$ the \df{characteristic polynomial} of $\cT$. 
Given a nonzero $\blam \in \K^K$, the \df{algebraic multiplicity} of $\scrL=\spann(\blam)$ as a JGE value of the tensor $\cT \in \K^{R \times R \times K}$, denoted $\am(\scrL)$, is the largest integer $m$ such that $\langle \blam, \bbgam \rangle ^m$ divides the characteristic polynomial of $\cT$. If $p_\cT$ factors as a product of $R$  (not necessarily distinct) linear terms, then we say that $\cT$ has $R$ JGE values counting algebraic multiplicity.

Define the \df{geometric multiplicity} of $\scrL$ as a JGE value of $\cT \in \K^{R \times R \times K}$, denoted $\gm(\scrL)$, to be the dimension of the span of the set of JGE vectors of $\cT$ which correspond to $\spann(\blam)$. That is,
\[
\gm(\scrL)= \dim \spann (\{\bx |\  (\scrL,\bx) \mathrm{\ is \ a \ JGE \ pair \ of \ } \cT\}).
\]
Similar to the cases of classical and generalized eigenvalues, the algebraic multiplicity of a JGE value is always greater than or equal to  the geometric multiplicity of said JGE value, see Proposition \ref{proposition:alggreatergeo}. However, we warn that it is possible for a JGE value to have nonzero algebraic multiplicity while having geometric multiplicity equal to zero, see Example \ref{exa:geomultzero}.

\begin{definition}
Say a tensor $\cT \in {R \times R \times K}$ is \df{defective} if it has a JGE value $\scrL$ with $\am(\scrL)>\gm(\scrL)$. 
\end{definition}

A degenerate case can occur in which $p_{\cT} (\bgam)$ is identically equal to zero. In this case, $\langle  \blam,\bbgam \rangle^m$ divides $p_{\cT} (\bgam)$ for all $\blam \in \K^K$ and for all integers $m$, hence we may reasonably say that $\spann (\blam)$ is a JGE value of $\cT$ with infinite algebraic multiplicity for all $\blam \in \K^K$. In this sense, if $p_{\cT} (\bgam) \equiv 0$ then the tensor $\cT$ is necessarily defective. 

It is straightforward to show that $p_{\cT} (\gamma)$ is not identically zero if and only if the span of the frontal slices of $\cT$ contains an invertible matrix. We call such a tensor a \df{slice mix invertible} tensor. See Example \ref{exa:pequiv0} and Lemma \ref{lemma:invertslice} for further discussion.

\begin{definition}
Say $\cT \in {R \times R \times K}$ is \df{simple} if $\cT$ is slice mix invertible and has $R$ distinct JGE values each having algebraic multiplicity equal to $1$. In other words, $\cT$ is simple if $p_{\cT}(\bgam)$ is not identically zero and factors as a product of $R$ distinct linear terms. 
\end{definition}

We are now in position to formally state the connection between the joint generalized eigenvalue decomposition and the canonical polyadic decomposition.

\begin{theorem}
\label{theorem:rankvsmult}
Let $\cT' \in \K^{I_1 \times I_2 \times I_3}$ and assume $\cT'$ has multilinear rank $(R,R,K)$. Let $\cT \in \K^{R \times R \times K}$ be an orthogonal compression of $\cT'$. Then we have the following.

\begin{enumerate}
\item \label{it:RankIFFBasis} $\cT'$ has $\K$-rank $R$ if and only if the JGE vectors of $\cT$ span $\K^R$.  Furthermore, if $\cT=\cpd$ has $\K$-rank $R$, then the JGE values of $\cT$ are equal to the spans of columns of the factor matrix $\bC$, and the columns of $\bB^{-\mathrm{T}}$ are JGE vectors of $\cT$.

\item \label{it:NonderogatoryImpliesRankR} If $\cT'$ has border $\K$-rank $R$ and $\cT$ is simple then $\cT'$ has $\K$-rank $R$.

\item \label{it:BRankLessRankImpliesDefective} If $\cT'$ has border $\K$-rank $R$ but $\K$-rank strictly greater than $R$ and if $\cT$ is slice mix invertible, then $\cT$ must have a JGE value $\scrL \subset \K^K$ which satisfies
\[
\am(\scrL) > \gm(\scrL) \geq 1.
\]
In this case, $\cT$ is defective in the sense of joint generalized eigenvalues. 
\end{enumerate} 
\end{theorem}

\begin{proof}
The proof of Theorem \ref{theorem:rankvsmult} is given in Section \ref{sec:rankvsmultproof}.
\end{proof}

Theorem \ref{theorem:rankvsmult} leads to a strategy for determining the radius of an open ball centered at a given rank $R$ tensor in which best $\K$-rank $R$ approximations are guaranteed to exist. In particular, it is sufficient to guarantee that nearby border $\K$-rank $R$ tensors are simple.

\begin{proposition}
\label{proposition:existenceCor}
Let $\cT' \in \K^{I_1 \times I_2 \times I_3}$ and assume $\cT'$ has $\K$-rank $R$ and multilinear rank $(R,R,K)$, and let $\cT \in \K^{R \times R \times K}$ be an orthogonal compression of $\cT'$. If every border $\K$-rank $R$ tensor in an open ball of Frobenius radius $\epsilon$ around $\cT$ is simple, then every tensor in an open ball of radius $\epsilon/2$ around $\cT'$ has a best $\K$-rank $R$ approximation. In addition, every $\K$-rank $R$ tensor inside the open ball of Frobenius radius $\epsilon$ centered at $\cT'$ has a unique CPD. Furthermore, the set of tensors inside this open ball for which the best rank $R$ approximation is not unique has Lebesgue measure zero.

In particular, with these assumptions the ball of radius $\epsilon$ around $\cT'$ does not contain a tensor with border $\K$-rank $R$ and $\K$-rank strictly greater than $R$.
\end{proposition}
\begin{proof} See Section \ref{sec:existenceCorProof}.
\end{proof}

 As previously mentioned, the existence of some neighborhood in which best rank $R$ approximations exist is known. The main motivation for Proposition \ref{proposition:existenceCor} is that a lower bound for the radius of a ball  centered at a simple rank $R$ tensor in which border rank $R$-tensors are simple can be computed using perturbation theoretic results for joint generalized eigenvalues, e.g., see Theorems \ref{theorem:MultiplePencilBound} and \ref{theorem:TensMDBound}. 

We also note that \cite{QML19, FO14} together show that the set of tensors which do not have a unique best border $\K$-rank $R$ approximation has measure $0$. The uniqueness of best $\K$-rank $R$ approximations discussed in Proposition \ref{proposition:existenceCor} follows quickly from these results after having shown existence of best $\K$-rank $R$ approximations.\footnote{In fact it is known that there is a neighborhood in which best rank $R$ approximations exist \textit{and are unique}. In regards to uniqueness, the conclusion of Proposition \ref{proposition:existenceCor} is slightly weaker; however, the fact that a lower bound for $\epsilon$ is computable is a significant advantage.}

\subsection{Deterministic bounds for existence of a best rank $R$ approximation}

We now explain how to compute the $\epsilon$ appearing in Proposition \ref{proposition:existenceCor}. For general $K$, Proposition \ref{proposition:DefectiveSubpencils} shows that if a tensor is defective, then any subpencil of the tensor must be defective. Hence by examining any given subpencil, one may use Theorem \ref{theorem:TensMDBound} to compute the radius of a ball around a given tensor in which best $\K$-rank $R$ approximations exist. A bound obtained from a single pencil may then be improved by combining information from a collection of non-overlapping subpencils of the tensor. In this way one arrives at the following multiple pencil based bound for the existence of a best $\K$-rank $R$ approximation.

\begin{theorem}
\label{theorem:MultiplePencilBound}
Let $\cT' \in \K^{I_1 \times I_2 \times I_3}$ be a tensor of $\K$-rank $R$ and multilinear rank $(R,R,K)$, and let $\cT \in \K^{R \times R \times K}$ be an orthogonal compression of $\cT'$. Let $\bU \in \K^{K \times K}$ be a unitary matrix and set $\cS = \cT \cdot_3 \bU$. For each $i=1,\dots, \lfloor K/2 \rfloor$, let $\epsilon_i \geq 0$ be a nonnegative real number such that every matrix pencil in an open Frobenius norm ball of radius $\epsilon_i$ around the pencil
\[
(\bS_{2i-1},\bS_{2i})
\]
is simple and set $\epsilon = ||(\epsilon_1, \dots, \epsilon_{ \lfloor K/2 \rfloor})||_2$. Then every tensor in an open ball of Frobenius radius $\epsilon/2$ centered at $\cT$ has a best $\K$-rank $R$ approximation and every rank $R$ tensor inside the open ball of Frobenius radius $\epsilon$ centered at $\cT$ has a unique CPD. Furthermore, the set of tensors inside this open ball for which the best rank $R$ approximation is not unique has Lebesgue measure zero.

In particular, each $\epsilon_i$ may be computed using Theorem \ref{theorem:TensMDBound}.
\end{theorem}
\begin{proof}
See Section \ref{sec:MultiplePencilBoundProof}.
\end{proof}

\begin{remark}
As an alternative to Theorem \ref{theorem:TensMDBound}, one could use \cite[Corollary 3]{DK87} to compute the  individual $\epsilon_i$ appearing in Theorem \ref{theorem:MultiplePencilBound}. We observe in experiments that the  $\epsilon_i$ computed by Theorem \ref{theorem:TensMDBound} are typically about the same or slightly larger than those computed using \cite[Corollary 3]{DK87}. Additionally, it is much more computationally intensive to obtain the $\epsilon_i$ using \cite[Corollary 3]{DK87}. However, for particularly ill-conditioned tensors, \cite[Corollary 3]{DK87} may outperform Theorem \ref{theorem:TensMDBound}.
\end{remark}

Though Theorem \ref{theorem:MultiplePencilBound} explains how to compute a lower bound for the radius of a neighborhood around a given rank $R$ tensor in which best rank $R$ approximations are guaranteed to exist, this result may be difficult to directly use in applications as one will not have access to the low rank signal tensor one is interested in. 

To remedy this difficulty, we present a variation on Theorem \ref{theorem:TensMDBound} which can be directly applied to (a low multilinear rank approximation of) a measured tensor. This variation, Theorem \ref{theorem:MLSVDExistenceBound}, only requires access to the measured tensor to use.

\begin{theorem}
\label{theorem:MLSVDExistenceBound}

Given $\cM' \in \K^{I_1 \times I_2 \times I_3}$ and an integer $R \leq \min\{I_1,I_2\}$, set $K=\min\{R_3(\cM'),R\}$. Let $\cW'$ be a multilinear rank $(R,R,K)$ approximation of $\cM'$ and let $\cW \in \K^{R \times R \times K}$ be an orthogonal compression of $\cW'$. Also let $\bU \in \K^{K \times K}$ be a unitary matrix and set $\cS = \cW \cdot_3 \bU$. For each $i=1,\dots, \lfloor K/2 \rfloor$, let $\epsilon_i \geq 0$ be a nonnegative real number such that every matrix pencil in an open Frobenius norm ball of radius $\epsilon_i$ around the pencil
\[
(\bS_{2i-1},\bS_{2i})
\]
is simple and set $\epsilon = ||(\epsilon_1, \dots, \epsilon_{ \lfloor K/2 \rfloor})||_2$. If there exists some $\K$-rank $R$ tensor $\tcT' \in \K^{I_1 \times I_2 \times I_3} $ such that 
\beq
\label{eq:MLSVDexistIneq}
\|\cM'-\tcT'\|_\mathrm{F} < \epsilon - \|\cM'-\cW'\|_\mathrm{F}
\eeq 
then $\cM'$ has a best $\K$-rank $R$ approximation and  any best $\K$-rank $R$ approximation of $\cM'$ has a unique CPD. Furthermore, the set of tensors $\cM'$ which satisfy these assumptions and do not have a unique best $\K$-rank $R$ approximation has Lebesgue measure zero.

In particular, each $\epsilon_i$ may be computed using Theorem \ref{theorem:TensMDBound}.
\end{theorem}
\begin{proof}
The proof of Theorem \ref{theorem:MLSVDExistenceBound} is nearly identical to the proof of Theorem \ref{theorem:MultiplePencilBound}. For completeness, details of the proof are given in the supplementary materials.
\end{proof}

Theorem \ref{theorem:MLSVDExistenceBound} is most naturally applicable to tensors which have low or approximately low multilinear rank. Indeed if this theorem is applied to a tensor which does not approximately have low multilinear rank, then the right hand side of equation \eqref{eq:MLSVDexistIneq} can be negative in which case no conclusion can be made.

However, as previously discussed, a common approach to computing a CPD in application is to first compute a core $\cW$ of a low multilinear rank approximation of $\cM'$, then to compute a low rank CPD approximation of $\cW$. In this approach, a critical question to answer is whether or not the core tensor $\cW$ has a best low rank approximation. To address this question, Theorem \ref{theorem:MLSVDExistenceBound} can be directly applied to $\cW$. Since $\cW$ already has low multilinear rank, if one uses Theorem \ref{theorem:MLSVDExistenceBound} in this way, then inequality \eqref{eq:MLSVDexistIneq} simplifies to checking if there is some $\K$-rank $R$ tensor $\tcT \in \K^{R \times R \times K}$ such that
\[
\|\cW-\tcT\|_\mathrm{F} < \epsilon
\]
where $\epsilon$ is computed as described in the theorem statement.
The tensor $\tcT$ used in this theorem can for instance be obtained by applying any standard CPD algorithm to $\tcT$. The resulting approximation error being bounded by $\epsilon$ is sufficient to make sure the approximation problem for $\cW$ is well-posed. The fact that $\tcT$ itself may not be optimal is not an issue. 

Numerical experiments which illustrate Theorem \ref{theorem:MLSVDExistenceBound} are given in Section \ref{sec:existenceImagesBigDim}.

\subsection{Deterministic bounds when $K=2$}
The following theorem explains how to compute the $\epsilon_i$ appearing in Theorems \ref{theorem:MultiplePencilBound} and \ref{theorem:MLSVDExistenceBound}. Before stating the result we give three brief definitions. Given two one-dimensional subspaces $\scrL_1,\scrL_2 \subseteq \K^K$, the \df{chordal metric}\footnote{The chordal metric is a popular choice of metric for studying generalized eigenvalues as it accounts for the natural scaling invariance in the generalized eigenvalue problem. See e.g. \cite[Chapter VI]{SS90}.} between $\scrL_1$ and $\scrL_2$, denoted $\chi(\scrL_1,\scrL_2)$, is equal to the sine of the angle between $\scrL_1$ and $\scrL_2$. 

Next, for $i=1,2$ let $\cT_i \in \K^{R \times R \times K}$ be a tensor which has $R$ JGE values counting algebraic multiplicity and let $\{\scrL_{i,r}\}_{r=1}^R$ be the spectrum of $\cT_i$. Define the \df{matching distance} between $\cT_1$ and $\cT_2$, denoted $\md[\cT_1,\cT_2]$, by 
\beq
\label{eq:MatchingDistanceDef}
\md[\cT_1,\cT_2]= \min_{\pi \in 
S_R} \max_{k=1, \dots, K} \chi(\scrL_{1,k}, \scrL_{2,\pi (k)}).
\eeq
Here $S_R$ is the group of permutations of $\{1, \dots, R\}$. We note that the matching distance is a metric between the spectrum $\cT_1$ and the spectrum of $\cT_2$.

Finally, for a tensor $\cM' \in K^{I_1 \times I_2 \times I_3}$, define the \df{spectral norm} of $\cM'$, denoted $\spnorm{\cM'}$  by
\[
\|\cM'\|_{\mathrm{sp}}:= \max_{\|\bx\|_2=\|\by\|_2=\|\bz\|_2=1} |\cM' \cdot_1 \bx \cdot_2 \by \cdot_3 \bz|
\]
where the vectors $\bx,\by,\bz$ are of compatible dimension. Equivalently, the spectral norm of $\cM'$ is the largest singular value of $\cM'$ which is in turn given by the Frobenius norm of a best rank $1$ approximation of $\cM'$. See \cite{FL18,FMPS13,FO14,QHX21} for more information about best rank $1$ approximations of tensors, tensor singular values, and the spectral norm of tensors.

\begin{theorem}
\label{theorem:TensMDBound}
Let $\cT$ and $\cW$ be $R \times R \times 2$ tensors with entries in $\K$. Assume that $\cT$ has $\K$-rank $R$ with CPD $\cpd$ where the columns of $\bC$ are normalized to have unit Euclidean norm. Let $\{\scrL_r\}_{r=1}^R$ be the spectrum of $\cT$. If 
\[
\spnorm{\cT-\cW} < \frac{\sigma_{\min} (\bA) \sigma_{\min} (\bB) \min_{i \neq j} \chi (\scrL_i, \scrL_j)}{2},
\]
then $\cW$ is a  simple  tensor of $\K$-rank $R$ and satisfies
\beq
\label{eq:MatchingDistanceBound}
\md[\cT,\cW]  \leq \frac{\spnorm{\cT-\cW}}{\sigma_{\min} (\bA) \sigma_{\min} (\bB)}.
\eeq
\end{theorem}

Since the Frobenius norm is an upper bound for the spectral norm of a tensor, Theorem \ref{theorem:TensMDBound} illustrates that if the pencil $(\bS_{2i-1},\bS_{2i})\cong \cS_i$ appearing in Theorem \ref{theorem:MultiplePencilBound} has CPD $\cS_i= [\![\bA_i,\bB_i,\bC_i]\!]$ where the columns of $\bC_i$ have unit Euclidean norm, then one may take
\[
\epsilon_i = \frac{\sigma_{\min} (\bA_i ) \sigma_{\min} (\bB_i ) \min_{k \neq j} \chi (\bC_i(:,k), \bC_i(:,j))}{2}.
\]
Here $\bC_i(:,k)$ is the $k$th column of $\bC_i$. Note that one can use a generalized eigenvalue decomposition to determine if an $R \times R \times 2$ tensor is slice mix invertible and has rank $R$ and, if so, compute a CPD of the tensor. This approach is often called Jennrich's algorithm \cite{Harsh70}.

In Theorem \ref{theorem:TensBauerFike} we give a perturbation theoretic bound for the \textit{joint} generalized eigenvalues of tensors $\cT$ and $\cW$ of size $R \times R \times K$. However, the measurement of variation between the joint generalized spectra of $\cT$ and $\cW$ used in Theorem \ref{theorem:TensBauerFike} is notably weaker than the matching distance which is used in Theorem \ref{theorem:TensMDBound}.

\begin{remark}
In the language of the classical generalized eigenvalue problem and matrix pencils, Theorem \ref{theorem:TensMDBound} gives a bound which guarantees that the matrix pencil $(\bW_1,\bW_2)$ is diagonalizable over $\K$. 
\end{remark}

\section{Basic results for the joint generalized eigenvalue problem}
\label{sec:JGEbasics}

We now develop several basic results for joint generalized eigenvalues. These results help reformulate CPD in terms of the joint generalized eigenvalue problem. This reformulation allows us to apply our upcoming perturbation theoretic results to the CPD setting.

\subsection{A simplifying assumption}

As a simplifying assumption throughout the section we restrict to the case where the tensor $\cT$ of interest has size $R \times R \times K$ for some $K \leq R$. In addition, we typically assume that $\cT$ is slice mix invertible. The assumption that $\cT$ is slice mix invertible is closely related to the tensor having multilinear rank $(R,R,K)$ and rank $R$, see Lemma \ref{lemma:invertslice}. 

As described in Section \ref{sec:Compress}, orthogonal compressions preserve geometry, so the assumption that $\cT$ has multilinear rank $R \times R \times K$ is not restrictive. The assumption that $\cT$ has a linear combination of frontal slices which is invertible holds for tensors with generic entries. 

We remark that, without assuming $\cT \in \K^{R \times R \times K}$ has rank $R$, the assumption that $\cT$ has multilinear rank $(R,R,K)$ is not sufficient for $\cT$ to be slice mix invertible. In fact, it is possible for a tensor to have multilinear rank $(R,R,K)$ and border $\K$-rank $R$ and not be slice mix invertible. This is illustrated by the following example.

\begin{exa}
\label{exa:pequiv0}
Let $\cT \in \R^{3 \times 3 \times 3}$ be the tensor with frontal slices
\[
\bT_1= \begin{pmatrix}
1 & 0 & 0 \\
0 & 0 & 0 \\
0 & 0 & 1
\end{pmatrix}
\quad
\bT_2= \begin{pmatrix}
0 & 2 & 0 \\
0 & 0 & -2 \\
0 & 0 & 0
\end{pmatrix}
\quad
\bT_3= \begin{pmatrix}
0 & 0 & -1 \\
0 & 0 & 0 \\
0 & 0 & 0
\end{pmatrix}.
\]
Then $\cT$ has multilinear rank $(3,3,3)$ and border $\R$-rank $3$, but $\cT$ is not slice mix invertible. That is, the span of the frontal slices of $\cT$ does not contain an invertible matrix. To see that $\cT$ has border $\R$-rank $3$, for each $n$ define the factor matrices
\[
\bA_n = \begin{pmatrix}
1 & 1 & 1 \\ 0 & \frac{1}{n} & \frac{1}{n^4} \\ 
0 & 0 & \frac{1}{n^2}
\end{pmatrix}
\quad
\bB_n = \begin{pmatrix}
1 & 0 & 0 \\ -\frac{1}{n^2} & n & 0 \\ -n^2 & -n^2 & n^2
\end{pmatrix}
\quad
\bC_n = \begin{pmatrix}
1 & \frac{1}{n^3} & 1 \\ \frac{4}{n} & \frac{2}{n} & \frac{6}{n} \\ 
\frac{1}{n^2} & 0 & 0
\end{pmatrix}
\]
and let $\cT^n$ be the $\K$-rank $3$ tensor $\cT^n = [\![\bA_n, \bB_n, \bC_n]\!]$. Then $\lim \cT^n = \cT$, hence the border $\K$--rank of $\cT$ is at most $3$. That $\cT$ has border $\K$--rank $3$ then follows from the fact that $\cT$ has multilinear rank $(3,3,3)$  together with equation \eqref{eq:BRankVsMLRank}. As a consequence of the forthcoming Lemma \ref{lemma:invertslice}, $\cT$ has $\K$-rank strictly greater than $3$.  
\end{exa}

\begin{lemma}
\label{lemma:invertslice}
Let $\cT \in \K^{R \times R \times K}$ be a tensor with $\K$-rank $R$ and full multilinear rank, and let $\cT=\cpd$ be a CPD of $\cT$. Then $\cT$ is slice mix invertible and the factor matrices $\bA,\bB \in \K^{R \times R}$ are invertible.
\end{lemma}
\begin{proof}
Assuming that $\cT$ has multilinear rank $(R,R,K)$ implies that each mode-$i$ unfolding of $\cT$ has full row rank. By choosing an appropriate ordering of columns, the mode-$1$ unfolding of $\cT$ is given by
\[
\bA (\bC \odot \bB)^\mathrm{T} \in \K^{R \times RK},
\]
where $\odot$ denotes the Khatri-Rao (column-wise Kronecker) product, e.g. see \cite{Cetal15}. $\cT$ having multilinear rank $(R,R,K)$ implies that the mode-$1$ unfolding has rank $R$ from which it follows that $\bA \in \K^{R \times R}$ has rank $R$ and is invertible. A similar argument shows that $\bB$ is invertible. 

We now show that $\cT$ is slice mix invertible. Using $\bT_i = \bA D_i (\bC) \bB^\mathrm{T}$ where $D_i (\bC)$ is the diagonal matrix whose entries are given by the $i$th row of $\bC$, to show that $\cT$ has a linear combination of frontal slices which is invertible, it is sufficient to show that there is a linear combination of rows of $\bC$ which has no entry equal to zero. To this end, note that since $\cT$ has $\K$-rank $R$, the factor matrix $\bC$ cannot have a column with all entries equal to zero. Therefore, for each $r=1,\dots,R$, the set of $\alpha \in \K^K$ such that
\[
\sum_{k=1}^K  \alpha_k \bC(r,k)=0
\]
is a nontrivial Zariski closed subset of $\K^K$. It follows that for $\alpha \in \K^K$ outside of a set of measure zero, one has
\[
\sum_{k=1}^K  \alpha_k \bC(r,k) \neq 0 \mathrm{\ for \ all \ } r=1,\dots,R
\]
from which the result follows.
\end{proof}

\subsection{Distinct JGE values correspond to linearly independent JGE vectors}

A perhaps unsurprising yet important fact is that a set of joint generalized eigenvectors corresponding to a set of distinct joint generalized eigenvalues is linearly independent.

\begin{lemma}
\label{lemma:DistinctENsImpliesEigenbasis}
Let $\cT \in \K^{R \times R \times K}$ and assume $\cT$ is slice mix invertible. Let $\{(\scrL_j,\bx_j)\}_{j=1}^\ell$ be a collection of JGE pairs of $\cT$ such that the $\scrL_j$ are all distinct. Then $\{\bx_j\}_{j=1}^\ell$ is a linearly independent set. As a consequence, if $\cT$ is simple and each of its JGE values has geometric multiplicity at least one, then $\cT$ has a JGE basis.
\end{lemma} 
\begin{proof}
 We treat the problem by reducing to the $K=2$ setting. To this end, for $j=1,\dots,\ell$ let $\blam_j \in \K^K$ be a vector such that $\spann(\blam_j)=\scrL_j$, and let $\bL \in \K^{K \times \ell}$ be the matrix whose $j$th column is equal to $\blam_j$. Since the $\scrL_j$ are distinct, the matrix $\bL$ has Kruskal rank at least $2$. That is, no pair of columns of $\bL$ is linearly dependent. It follows that there exists a matrix $\bV \in \K^{2 \times K}$ such that $\bV \bL$ has Kruskal rank $2$. 

Set $\cS = \cT  \cdot_3 \bV \in \K^{R \times R \times 2}$ and let $\bS_1$ and $\bS_2$ be the frontal slices of $\cS$. It is straightforward to show that each $\bx_j$ is a generalized eigenvector of the matrix pencil $(\bS_1, \bS_2)$ which corresponds to the generalized eigenvalue $\spann(\bV \blam_j)$. Furthermore, since the matrix $\bV \bL$ has Kruskal rank $2$, the generalized eigenvalues $\{\spann(\bV \blam_j)\}_{j=1}^\ell$ are all distinct. As a consequence of \cite[Chapter VI, Theorem 1.11]{SS90}, distinct generalized eigenvalues correspond to linearly independent eigenvectors, so it follows that $\{\bx_j\}_{j=1}^\ell$ is a linearly independent set.
\end{proof}

It is possible for a joint generalized eigenvalue of a slice mix invertible tensor to have nonzero algebraic multiplicity while having geometric multiplicity equal to zero, hence, the assumption that the joint generalized eigenvalues of $\cT$ have nonzero geometric multiplicity in the second part of of Lemma \ref{lemma:DistinctENsImpliesEigenbasis} is necessary.

\begin{exa}
\label{exa:geomultzero}
Let $\cT \in \R^{3 \times 3 \times 3}$ be the tensor with frontal slices
\[
\bT_1= \begin{pmatrix}
1 & 0 & 0 \\
1 & 1 & 0 \\
3 & 4 & 1
\end{pmatrix}
\quad
\bT_2= \begin{pmatrix}
1 & 0 & 0 \\
1 & 2 & 0 \\
4 & 5 & -1
\end{pmatrix}
\quad
\bT_3= \begin{pmatrix}
1 & 0 & 0 \\
5 & 3 & 0 \\
7 & 8 & 5
\end{pmatrix}.
\]
Then 
\[
p_\cT(\gamma)=\det \left(\sum_{k=1}^3 \gamma_k \bT_k \right)= (\gamma_1 + \gamma_2 + \gamma_3) (\gamma_1 + 2 \gamma_2 + 3 \gamma_3) (\gamma_1 - \gamma_2 + 5 \gamma_3)
\]
hence $\cT$ is slice mix invertible and has three distinct JGE values counting algebraic multiplicity, namely $\spann ((1,1,1))$, $\spann ((1,2,3)),$ and $\spann ((1,-1,5))$. However, the JGE values $\spann ((1,1,1))$ and $\spann ((1,2,3))$ do not have a corresponding JGE vector.
\end{exa}

In Example \ref{exa:geomultzero}, one could reasonably consider left joint generalized eigenvalues of $\cT$ in which case the JGE value $\spann ((1,1,1))$ in Example \ref{exa:geomultzero} corresponds to the left JGE vector $(1,0,0) \in \R^3$. However, the JGE values $\spann ((1,-1,5))$ and $\spann ((1,2,3))$ do not have left JGE vectors. The JGE value $\spann ((1,2,3))$ is of particular note since it neither has left nor right JGE vectors. 

Assuming that a tensor $\cT$ is slice mix invertible, it is not difficult to show that $\cT$ has a basis of right JGE vectors if and only if it has a basis of left JGE vectors. However, this fact and left JGE vectors in general have little impact on our results, so we do not discuss them further. 

\subsection{Algebraic and geometric multiplicities}

We now more carefully examine algebraic and geometric multiplicities for joint generalized eigenvalues. In particular we show that algebraic multiplicity is always greater than or equal to geometric multiplicity of a joint generalized eigenvalue. As an immediate consequence, if the tensor $\cT \in \K^{R \times R \times K}$ has a JGE pair $(\scrL,\bx)$, then $\am(\scrL) \geq 1$. 

Before stating our result we recall that if $\cT$ is not slice mix invertible, then $p_{\cT}$ is identically zero in which case $\langle \blam,\bgam \rangle^m$ divides $p_{\cT} (\bgam)$ for all $\blam \in \K^K$ and all $m \in \mathbb{N}$. In this case one can say that $\spann(\blam)$ has infinite algebraic multiplicity as a JGE value of $\cT$ for all $\blam \in \K^K$. For this reason, in the upcoming proposition it is sufficient to restrict to the case where $\cT$ is slice mix invertible.

\begin{proposition}
\label{proposition:alggreatergeo}
Let $\cT \in \K^{R \times R \times K}$ and assume $\cT$ is slice mix invertible. Then for any nonzero $\blam \in \K^K$ one has $\am(\spann(\blam)) \geq \gm(\spann(\blam))$. In particular, if $(\spann(\blam),\bx)$ is a JGE pair of $\cT$, then there is an integer $m>0$ such that the characteristic polynomial of $\cT$ factors as
\beq
\label{eq:gmamFactor}
p_\cT (\bgam) = \langle \blam, \overline{\bgam} \rangle^m g(\bgam)
\eeq
where $g$ is a nonzero polynomial in the indeterminate $\bgam$.

Furthermore, if $\cT$ has  border $\K$-rank $R$, then $p_\cT(\bgam)$ factors into a product of $R$ linear terms. That is, one has
\[
p_\cT (\bgam) = \Pi_{r=1}^R \langle \blam_r, \overline{\bgam} \rangle
\]
hence $\cT$ has $R$ JGE values counting algebraic multiplicity.
\end{proposition}
\begin{proof}

Suppose $\gm( \blam) = m$ and let $\bx_1, \dots, \bx_m$ be the corresponding JGE vectors. For $i=1, \dots, m$ let $\by_i$ be the nonzero vector corresponding to the eigenpair $(\spann(\blam),\bx_i)$, as in equation \eqref{eq:JGEVdef}. Note that the assumption that $\cT$ is slice mix invertible guarantees that $\{\by_1, \dots, \by_m\}$ is a linearly independent set. To show that $p_{\cT}$ factors as desired, we will use the $\bx_i$ and $\by_i$ to construct matrices $\bX$ and $\bQ$ which simultaneously block upper triangularize the frontal slices $\bT_k$ with $\lambda_k \bI_m$ on the first diagonal block.

To simultaneously block upper diagonalize the $\bT_k$, let $\bX_1 \in \R^{R \times m}$ be a matrix whose $i$th column is $\bx_i$, and let $\bX_2 \in \R^{R \times (R-m)}$ be any matrix such that the block matrix $\bX= (\bX_1 \ \bX_2)$ is invertible. Similarly let $\bY_1  \in \R^{R \times m}$ be the matrix whose $i$th column is $\by_i$, and let $\bQ_2 \in \R^{(R-m) \times R}$ be a matrix whose rows form a basis for the orthogonal complement of the range of $\bY_1$. Set 
\[
\bQ =\begin{pmatrix}
\bY_1^\dagger  \\
\bQ_2
\end{pmatrix} .
\]

With this set up, for each $k=1, \dots, K$ one has
\[
\bQ \bT_k \bX = \bQ (\lam_k \bY_1 \ \bT_k \bX_2) = \begin{pmatrix}
\lam_k  \bY_1^\dagger \bY_1 & \bY_1^\dagger \bT_k \bX_2 \\
\lambda_k \bQ_2 \bY_1 &  \bQ_2 \bT_k \bX_2 
\end{pmatrix} = \begin{pmatrix}
\lam_k  \bI_m & \bY_1^\dagger \bT_k \bX_2 \\
\mathbf{0} & \bQ_2 \bT_k \bX_2 
\end{pmatrix} .
\]
The matrices $\bX$ and $\bQ$ are both invertible by construction so using 
\[
\det\left(\bQ\right) \det \left( \sum_{k=1}^K \gamma_k \bT_k \right) \det(\bX)  =\det \left( \sum_{k=1}^K \gamma_k \bQ \bT_k \bX \right) = 
\det \left( \sum_{k=1}^K \gamma_k \begin{pmatrix} 
\lam_k  \bI_m & \bY_1^\dagger \bT_k \bX_2 \\
\mathbf{0} & \bY_2 \bT_k \bX_2 
\end{pmatrix} \right)
\]
shows that $p_\cT (\bgam)$ factors as in equation \eqref{eq:gmamFactor}. The fact that the polynomial $g(\bgam)$ is nonzero is a consequence of the assumption that $\cT$ is slice mix invertible, hence $p_\cT$ is not identically zero.

To complete the proof, first assume that $\cT$ has $\K$-rank $R$ and let $\cT=\cpd$ be the CPD of $\cT$. We have assumed $\cT$ is slice mix invertible, so Lemma \ref{lemma:invertslice} shows that $\bA$ and $\bB$ are invertible. Using $\bT_k = \bA D_k (\bC) \bB^\mathrm{T} $ shows that
\[
p_\cT (\bgam)= \det \left( \bA \sum_{k=1}^K \gamma_k D_k(\bC) \bB^\mathrm{T} \right) = \det (\bA)  \det \left(\sum_{k=1}^K \gamma_k D_k(\bC) \right) \det( \bB^\mathrm{T}).
\]
The final term in this equality easily factors as desired. 

 Now assume that $\cT$ has border $\K$-rank $R$. Then $\cT=\lim \cT^n$ for some sequence $\{\cT^n\}_{n=1}^\infty$ of $\K$-rank $R$ tensors. Since $\cT$ is slice mix invertible, we may assume that each $\cT^n$ is slice mix invertible. Thus, as we have shown above, for each $n$ we have 
\[
p_{\cT^n}  (\bgam) = \alpha_n \Pi_{r=1}^R \langle \blam_{r,n}, \overline{\bgam} \rangle.
\]
Here the constant $\alpha_n \in \K$ has been added so that we may scale to have $||\blam_{r,n}||_2=1$ for each $r=1,\dots, R$ and for all $n \in \bbN$. 

Since all the $\blam_{r,n}$ have norm $1$, by passing to a common subsequence if necessary, the sequence $\{\blam_{r,n}\}_{n=1}^\infty$ converges to some unit norm $\blam_r \in \K^K$ for each $r=1,\dots, R$. We next show that the constants $\alpha_n$ are bounded in norm. To this end note that there must exist some unit vector $\bgam_* \in \K^K$ such that
\[
 \Pi_{r=1}^R \langle  \blam_{r}, \overline{\bgam_*} \rangle \neq 0,
\]
as otherwise would imply that the union of the orthogonal complements of the $\blam_r$ is dense in $\K^K$. 

Since the $\cT^n$ converge to $\cT$ we must have $\lim_{n \to \infty} p_{\cT^n} (\gamma) = p_{\cT} (\gamma)$ from which it follows that 
\[
p_{\cT} (\bgam_*) = \lim_{n \to \infty} p_{\cT^n} (\bgam_*) = \lim_{n \to \infty} \alpha_n \Pi_{r=1}^R \langle  \blam_{r,n}, \overline{\bgam_*} \rangle.
\]
Using this together with 
\[
\lim_{n \to \infty} \Pi_{r=1}^R \langle \blam_{r,n}, \overline{\bgam_*} \rangle =  \Pi_{r=1}^R \langle \blam_{r}, \overline{\bgam_*} \rangle \neq 0
\]
shows that the sequence $\{\alpha_n\}_{n=1}^\infty$ converges to some $\alpha \in \K$. It follows that for all $\bgam \in \K^K$ we have 
\[
\lim_{n\to \infty} p_{\cT^n} (\bgam) = \lim_{n\to \infty} \alpha_n \Pi_{r=1}^R \langle \blam_{n,r}, \overline{\bgam} \rangle =  \alpha \Pi_{r=1}^R \langle \blam_{r}, \overline{\bgam} \rangle = p_{\cT} (\bgam).
\]
Furthermore, since $\cT$ is slice mix invertible, we must have $\alpha \neq 0$. We conclude that $p_{\cT} (\bgam)$ factors as desired. 
\end{proof}

\subsection{Multiplicities for border rank $R$ tensors}

We now examine geometric multiplicities of JGE values for slice mix invertible border $\K$-rank $R$ tensors. Namely we show that JGE values of a slice mix invertible tensor with border $\K$-rank $R$ have nonzero geometric multiplicity. This lemma plays an important role in the proof of Theorem \ref{theorem:rankvsmult} \eqref{it:NonderogatoryImpliesRankR}, as it allows one to conclude that a simple border $\K$-rank $R$ tensor  has a basis of JGE vectors. 

\begin{lemma}
\label{lemma:BorderRankGeo}
Let $\cT \in \K^{R \times R \times K}$ be a slice mix invertible tensor and assume $\cT$ has multilinear rank $(R,R,K)$ and border $\K$-rank $R$. If $\scrL \subset \R^K$ is a JGE value of $\cT$ with $\am(\scrL)=m \geq 1$, then $\gm(\scrL) \geq 1$. 
\end{lemma}

\begin{proof}
Let $\{\cT^n\}_{n=1}^\infty$ be a sequence of $\K$-rank $R$ tensors which converges to $\cT$ and let $\blam \in \K^K$ be a unit vector such that $\spann(\blam)=\scrL$.  Proposition \ref{proposition:alggreatergeo} shows that $p_\cT (\bgam)$ and all the $p_{\cT^n} (\bgam)$ factor as products of linear terms.  Note that $\lim \cT^n = \cT$ implies that $\lim p_{\cT^n} = p_\cT$, so the normalized linear factors of the $p_{\cT^n}$ converge to the linear factors of $p_\cT$. It follows that there exists a sequence of unit vectors $\{\blam_n\}_{n=1}^\infty \subset \K^K$ such that $\spann(\blam_n)$ is a JGE value of $\cT^n$ for each $n$ and such that 
\[
\lim_{n \to \infty} \blam_n = \blam.
\]

As a consequence of Theorem \ref{theorem:rankvsmult}, each $\spann(\blam_n)$ has geometric multiplicity at least $1$ as a JGE value of $\cT^n$. For each $n$, let $\bx_n \in \K^R$ be a unit norm JGE vector of $\cT^n$ corresponding to $\blam_n$, and let $\by_n \in \K^R$ be the nonzero vector such that the tuple $(\cT^n,\blam_n,\bx_n,\by_n)$ satisfies equation \eqref{eq:JGEVdef} for each $n$. 

By construction the  sequences $\{\bx_n\}$ and $\{\blam_n\}$ lie in a compact sets. Additionally, since the $\cT^n$ converge to $\cT$, the norms of the frontal slices of $\cT^n$ are uniformly bounded, so since the $\blam_n$ all have norm $1$, the sequence $\{\by_n\}$ must also lie in a compact set. Therefore, by passing to a common subsequence if necessary, the sequences $\{\bx_n\}$ and $\{\by_n\}$ have limits $\bx,\by \in \K^R$. For each $k=1,\dots K$ we then have
\beq
\label{eq:limJGEV}
\bT_k \bx = \lim_{n \to \infty} \bT_k^n \bx_n =  \lim_{n \to \infty} \lam_{n,k} \by^n = \lam_k \by
\eeq
where $\lam_{n,k}$ denotes the $k$th entry of $\blam_n$.

By assumption, $\cT$ is slice mix invertible, so the frontal slices of $\cT$ cannot have a common null space and the vector $\by$ cannot be equal to zero. We conclude that $(\scrL,\bx)$ is a joint generalized eigenpair of $\cT$ and that $\gm(\scrL) \geq 1$.
\end{proof}

\subsection{Proof of Theorem \ref{theorem:rankvsmult}}
\label{sec:rankvsmultproof}

We are now in position to give the proof of our first main result, Theorem \ref{theorem:rankvsmult}. 

\begin{proof}

We first prove Item \eqref{it:RankIFFBasis}. Note that an orthogonal compression $\cT$ of $\cT'$ has $\K$-rank $R$ if and only if $\cT'$ has $\K$-rank $R$, so throughout the proof we WLOG assume that $\cT'$ has size $R \times R \times K$ and full multilinear rank, hence $\cT=\cT'$. 

First assume that $\cT$ has $R$ JGE values in $\K^K$ counting geometric multiplicity. That is, assume that $\cT$ has a basis of JGE vectors. Then there exists an invertible matrix $\bX \in \K^{R \times R}$, a collection of diagonal matrices\footnote{The matrices $\{\bLam_k\}_{k=1}^K$ are closely related to the JGE values of $\cT$. In particular the subspace $\spann((\bLam_1 (r,r),\bLam_2 (r,r),\dots,\bLam_K (r,r)) \subset \K^K$  is a JGE value of $\cT$ for each $r=1,\dots,R$. Here $\bLam_k (r,r)$ denotes the $(r,r)$th entry of $\bLam_k$.} $\{\bLam_k\}_{k=1}^K$, and matrix $\bY \in \K^{R \times R}$ such that 
\[
\bT_k \bX = \bY \bLam_k \qquad \mathrm{for \ all \ } k=1, \dots, K. 
\]
It follows that 
\[
\bT_k = \bY \bLam_k \bX^{-1} \qquad \mathrm{for \ all \ } k=1, \dots, K. 
\]

Set $\bA=\bY$ and $\bB= \bX^{-\mathrm{T}}$ and let $\bC \in \K^{R \times K}$ be the matrix whose $k$th row is given by the diagonal entries of $\bLam_k$. Then $\cpd=\cT$ is an $R$-term CPD for $\cT$, hence $\cT$ has rank less than or equal to $R$. That $\cT$ has rank $R$ then follows from the assumption that $\cT$ has multilinear rank $(R,R,K)$ and the fact that the rank of a tensor dominates its multilinear rank. 

Now, similar to the proof of Proposition \ref{proposition:alggreatergeo}, assume that $\cT$ has rank $R$ with CPD $\cT= \cpd$. By assumption $\cT$ has full multilinear rank, so Lemma \ref{lemma:invertslice} shows that the factor matrices $\bA,\bB \in \K^{R \times R}$ are invertible. Using $\bT_k = \bA D_k (\bC) \bB^\mathrm{T} $, we then have
\[
\bT_k \bB^{-\mathrm{T}} = \bA D_k (\bC) \qquad \mathrm{for \ all \ } k=1, \dots, K.
\]

For $r=1,\dots, R$, let $\bx_r$ be the $r$th column of the matrix $\bB^{-\mathrm{T}}$ and let $\by_r$ be the $r$th column of $\bA$. Then $\bx_r,\by_r \in \K^K$ are nonzero vectors for all $r$ and for each $r=1,\dots,R$ and $k=1,\dots,K$ we have 
\[
\bT_k \bx_r = c_{kr} \by_r
\]
where $c_{kr}$ denotes the $(k,r)$ entry of $\bC$. Setting $\scrL_r = \spann \big(c_{1r}, \dots, c_{Kr}\big)$, it follows that $(\scrL_r, \bx_r)$ is a joint generalized eigenpair of $\cT$ for each $r=1, \dots, R$. We conclude that $\cT$ has a JGE basis, that the JGE values of $\cT$ are given by the spans of the columns of $\bC$, and that the JGE vectors of $\cT$ are given by the columns of $\bB^{-\mathrm{T}}$, as claimed.

Item \eqref{it:NonderogatoryImpliesRankR} follows immediately from Item \eqref{it:RankIFFBasis} together with Lemma \ref{lemma:BorderRankGeo}. To prove Item \eqref{it:BRankLessRankImpliesDefective}, suppose $\cT'$ has border $\K$-rank $R$ and assume towards a contradiction that $\am(\scrL)=\gm(\scrL)$ for all JGE values $\scrL$ of $\cT$. Proposition \ref{proposition:alggreatergeo} shows that $p_\cT$ factors as a product of linear terms, so in this case $\cT$ must have a basis of JGE vectors which using Item \eqref{it:RankIFFBasis} shows that $\cT'$ has rank $R$, a contradiction. It follows that $\cT$ has some JGE value with algebraic multiplicity strictly greater than geometric multiplicity. The proof is completed by Lemma \ref{lemma:BorderRankGeo} which shows that every JGE value of $\cT$ has geometric multiplicity at least one.
\end{proof}

\section{Perturbation bounds for joint generalized eigenvalues}
\label{sec:perturbation}

We now examine how joint generalized eigenvalues are affected by a perturbation of the tensor of interest. Our main goal at this point is to provide a way to determine if border $\K$-rank $R$ tensors near a given slice mix invertible $\K$-rank $R$ tensor have distinct JGE values. That is, we want to show that nearby border $\K$-rank $R$ tensors are nondefective hence have rank equal to border rank. To achieve this goal, we provide a Bauer--Fike like bound for JGE eigenvalues which may be converted to a spectral matching distance bound in the case $K=2$. 

As in the previous section, it will be helpful to reduce to the case where $\cT$ has size $R \times R \times K$ and full multilinear rank.
To make this reduction we consider a ``tensor Procrustes problem" which serves as an analogue of the classical matrix Procrustes problem \cite{S66}. 

In particular we will show that given two tensors $\cW',\hat{\cW'} \in \K^{I_1 \times I_2 \times I_3}$, both having multilinear rank bounded above by $(R,R,K)$, there exist orthogonal compressions $\cW,\hat{\cW} \in \K^{R \times R \times K}$ of $\cW'$ and $\hat{\cW'}$ such that
\beq
\label{eq:ProcrustesEq}
\|\cW - \hat{\cW} \|_\mathrm{F} \leq \| \cW' - \hat{\cW'} \|_\mathrm{F}.
\eeq
Though it is straightforward how to compute these orthogonal compressions in the special case that the mode-$j$ fibers of $\cW'$ and $\hat{\cW'}$ span the same space for each $j=1,2,3$, one should not expect an arbitrary pair of low multilinear rank tensors to have this special property.

In Section \ref{sec:TensorProcrustes} we explain how to compute orthogonal compressions of $\cW'$ and $\hat{\cW'}$ which satisfy equation \eqref{eq:ProcrustesEq}. Additionally, we explain our perspective on orthogonal compressions which differs slightly from the perspective most commonly taken in applications. While the tensor Procrustes problem plays an important role in our theoretical results, the solution may be unsurprising to readers familiar with the matrix Procrustes problem. Such readers can advance to Section \ref{sec:existenceCorProof}.
 
\subsection{ Orthogonal compressions of pairs of tensors: The tensor Procrustes problem}
\label{sec:TensorProcrustes}

As previously discussed, the situation most common to application is that one has access to a measurement $\cM'=\cT'+\cN'$ of a rank $R$ tensor $\cT'$ which is corrupted by measurement error $\cN'$. In order to approximate the factors underlying the tensor $\cT'$, the practitioner will compute a rank $R$ approximation of $\cM'$. Typically the rank of $\cT'$ is significantly less than its dimensions, so to reduce the computational complexity of computing a rank $R$ approximation of $\cM'$ one first computes a low  multilinear rank approximation $\cW'$ of $\cM'$. One then computes a $R$ rank approximation of the low multilinear rank approximation $\cW'$. 

While this approach is natural to use in applied settings, it can have adverse effects when searching for theoretical guarantees. Letting $\hcT'$ denote a best border rank $R$ approximation of $\cM'$, to simplify to the $R \times R \times K$ setting, we need orthogonal compressions of both $\cT'$ and $\hcT'$. If one independently computes orthogonal compressions $\cT$ and $\hcT$ of $\cT'$ and $\hcT'$, then it is possible that the distance between $\cT$ and $\hcT$ is greater than the distance between $\cT'$ and $\hcT'$. This difficulty must be treated with particular care if the span of the mode-$j$ subspaces of $\cT'$ is not the same as the span of the mode-$j$ subspace of $\hcT$. 

Roughly speaking this difficulty is resolved by first applying an orthogonal transformation to $\hcT$ which both does not increase the distance between $\cT$ and $\hcT$ and makes it so that the mode-$j$ subspaces of the orthogonal transformation of $\hcT$ are the same as those of $\cT$. From this point it is straightforward how to compute the appropriate orthogonal compressions. As the (border) rank of $\hcT'$ is invariant under orthogonal transformations, for our purposes there is no loss in studying an orthogonal transformation of $\hcT'$ rather than $\hcT'$ itself.\footnote{In practice one rarely has access to both $\cT'$ and $\hcT'$, but this is immaterial to our present discussion.}

\begin{theorem}
\label{theorem:TensorProcrustes} Let $\cW', \hat{\cW'} \in \K^{I_1 \times \cdots \times I_\ell}$ be tensors having multilinear rank $(R_1, \cdots, R_\ell)$ and $(\hR_1, \cdots, \hR_\ell)$, respectively. Let $\bV: \K^{R_j} \to \K^{I_j}$ be column-wise orthonormal with range equal to the subspace spanned by the mode-$j$ fibers of $\cW'$. If $\hR_j \leq R_j$, then there exists a column-wise orthonormal matrix $\hbV: \R^{R_j} \to \R^{I_j}$ with range containing the subspace spanned by the mode-$j$ fibers of $\hcW$ such that
\[
\|\cW' \cdot_j \bV^\mathrm{H} - \hat{\cW'} \cdot_j \hbV^\mathrm{H} \|_\mathrm{F} \leq \|\cW'-\hat{\cW'}\|_\mathrm{F}.
\]
As an immediate consequence, letting $R^*_i = \max \{R_i,\hR_i\}$ for each $i=1,\dots,\ell$ there exist orthogonal compressions $\cW,\hcW \in \K^{R^*_1 \times \cdots \times R^*_\ell} $ of $\cW'$ and $\hat{\cW'}$, respectively, such that
\[
\|\cW-\hcW\|_\mathrm{F} \leq \|\cW' -\hat{\cW'}\|_\mathrm{F}. 
\]
\end{theorem}

As will become clear, we may  use Theorem \ref{theorem:TensorProcrustes} to without loss of generality restrict to the case of $R \times R \times K$ tensors. Before proving the theorem we give a technical lemma. 

\begin{lemma}
\label{lemma:ProcrustesIsometries}
Let $\bM,\bH \in \K^{I_1 \times I_2}$ be matrices of rank $R_\bM$ and $R_\bN$, respectively, where $R_\bM \geq R_\bN$. There exists a unitary matrix $ \bU \in \K^{I_1 \times I_1}$ such that $\ran \bM \supseteq \ran \bU \bH$ and so that 
\[
\|\bM-\bU \bN\|_\mathrm{F} \leq \|\bM -\bN \|_\mathrm{F}.
\]
As a consequence, given any column-wise orthonormal matrix $\bU_\bM \in \K^{I_1 \times R_\bM}$ mapping onto $\ran \bM$, there is a column-wise orthonormal matrix $\bU_\bN  \in \K^{I_1 \times R_\bN}$ such that $\ran \bU_\bN \supseteq \ran \bN$ and such that
\[
\| \bU_\bM^\mathrm{H} \bM - \bU_\bN^\mathrm{H} \bN \|_\mathrm{F} \leq \|\bM-\bN \|_\mathrm{F}.
\]
\end{lemma}
\begin{proof}
The proof of Lemma \ref{lemma:ProcrustesIsometries} is a routine argument using the solution of the orthogonal Procrustes problem \cite{S66}. Details are given in the supplementary materials. 
\end{proof}

We now give the proof of Theorem \ref{theorem:TensorProcrustes}.

\begin{proof}
The main strategy of the proof is to apply Lemma \ref{lemma:ProcrustesIsometries} to each mode of the tensors $\cW'$ and $\hat{\cW'}$. To this end, fix $j$ and let $\cW'_{(j)}$ and $\hat{\cW'}_{(j)}$ be mode-$j$ unfoldings of $\cW'$ and $\hat{\cW'}$, respectively, with the same ordering of the fibers.  Recall that $\cW_{(j)}$  has rank $R_j$ and $\hat{\cW'}_{(j)}$ has rank $\hR_j$, and WLOG assume $R_j \geq \hR_j$. 

Using Lemma \ref{lemma:ProcrustesIsometries}, given any column-wise orthonormal matrix $\bV \in \R^{I_j \times R_j}$ with range equal to $\ran \cW'_{(j)}$, there exists a column-wise orthonormal matrix $\hbV \in \K^{I_j \times R_j}$ such that the range of $\hbV$ contains $\ran \hat{\cW'}_{(j)}$ and such that
\[
\| \bV^\mathrm{H} \cW'_{(j)} - \hbV^\mathrm{H} \hat{\cW'}_{(j)} \|_\mathrm{F} \leq \|\cW'_{(j)}-\hat{\cW'}_{(j)} \|_\mathrm{F}.
\]
Since the ordering of the mode-$j$ fibers is the same in $\cW_{(j)}$ and $\hat{\cW'}_{(j)}$ we have
\[
\| \bV^\mathrm{H} \cW'_{(j)} - \hbV^\mathrm{H} \hat{\cW'}_{(j)}\|_\mathrm{F} = \|\cW' \cdot_j \bV^\mathrm{H} - \hat{\cW'} \cdot_j \hbV^\mathrm{H}\|_\mathrm{F}  \quad \quad \mathrm{and} \quad \quad \|\cW'_{(j)}-\hat{\cW'}_{(j)}\|_\mathrm{F}=\|\cW'-\hat{\cW'} \|_\mathrm{F}
\]
from which the first conclusion follows.

The second part of the result follows from a standard argument using induction on the mode of the fibers. 
\end{proof}

It is worth noting that the Procrustes compressions $\cW$ and $\hcW$ one arrives at using the methodology in the proof are not necessarily optimal. That is, there may exist cores $\cW^*$ and $\hat{\cW^*}$ such that
\[
\|\cW^*-\hat{\cW^*}\|_\mathrm{F} < \|\cW - \hcW \|_\mathrm{F} \leq \|\cW'- \hat{\cW'}\|_\mathrm{F}.
\]
One may compute locally optimal cores by using an ALS approach for instance where one continues to iterate over the modes of the tensors and solves the matrix Procrustes problem for each corresponding unfolding. Methods for computing globally optimal cores are outside the scope of this article.

A similar issue occurs when computing a best low multilinear rank approximation of a given tensor $\cM' \in \K^{I_1 \times I_2 \times I_3}$. Here one wants to minimize $\|\cM'-\cW'\|$ subject to a multilinear rank constraint on $\cW'$. A common approach is to compute a MLSVD of $\cM'$ and truncate at the specified multilinear rank. While this approach is in a sense optimal ``per mode", since the MLSVD treats each mode independently the result is rarely globally optimal when all modes are considered together.

\subsection{Proof of Proposition \ref{proposition:existenceCor}}
\label{sec:existenceCorProof}

We are now in position to prove Proposition \ref{proposition:existenceCor}.

\begin{proof}
The assumption that every border $\K$-rank $R$ tensor in an open $\epsilon$ ball around $\cT$ is simple in fact implies that this $\epsilon$ ball cannot contain a tensor with border $\K$-rank strictly less than $R$. We temporarily assume this implication holds and give a detailed proof at the end.

Let $\cM' \in \K^{I_1 \times I_2 \times I_3}$ be any tensor such that $\|\cM'-\cT'\| < \epsilon/2$, and let $\hcT'$ be a best border $\K$-rank $R$ approximation of $\cM'$. Note that $\cT'$ has multilinear rank $(R,R,K)$ by assumption, so inequality \eqref{eq:BRankVsMLRank} shows that $\cT'$ has border $\K$-rank $R$ in addition to having $\K$-rank $R$. It follows that 
\[
\|\hcT' - \cM' \| \leq \| \cT' - \cM' \|,  \quad \quad \mathrm{hence} \quad \quad \| \cT' - \hcT' \| < \epsilon,
\]
as otherwise $\cT'$ would be a strictly better border $\K$-rank $R$ approximation of $\cM'$ than $\hcT'$.  To show that $\cM'$ has a best $\K$-rank $R$ approximation, it is sufficient to show that $\hcT'$ has $\K$-rank equal to $R$.

Since $\hcT'$ has border $\K$-rank $R$, the rank of each mode-$i$ unfolding of $\hcT'$ is less than or equal to $R$. Therefore,
using Theorem \ref{theorem:TensorProcrustes}, we may simultaneously orthogonally compress $\cT'$ and $\hcT'$ to tensors of size $R \times R \times K^*$ for some $K^* \leq R$ without increasing the Frobenius distance between the tensors. Therefore we may WLOG we may assume $\cT'=\cT$ and $\hcT'=\hcT$ both have size $R \times R \times K^*$. After this restriction, our assumptions imply that $\hcT$ is simple. An application of Theorem \ref{theorem:rankvsmult} \eqref{it:NonderogatoryImpliesRankR} then shows that $\hcT$ has rank $R$. 

We next prove that every rank $R$ tensor in the open ball of Frobenius radius $\epsilon$ centered at $\cT$ has a unique CPD. By assumption every border rank $R$ tensor in this ball is simple. As previously mentioned, this assumption implies that the $\epsilon$ ball does not contain a tensor with border rank less than $R$, hence every rank $R$ tensor in the ball is simple. Using Theorem \ref{theorem:rankvsmult} \eqref{it:RankIFFBasis} together with Lemma \ref{lemma:DistinctENsImpliesEigenbasis}, it follows that every tensor inside this ball satisfies Kruskal's condition \cite{K77}, hence every tensor in this ball has a unique CPD.

We now show that the set of tensors contained in this open ball which do not have a unique best $\K$-rank $R$ approximation has Lebesgue measure $0$. For tensors with entries in $\C$ (equipped with Lebesgue measure over the complexes) this is an immediate consequence of \cite{QML19}. For tensors with entries in $\R$ (equipped with Lebesgue measure over the reals) this follows from  \cite[Corollary 18]{FO14} together with the first part of this result.

It remains to show that every tensor in an $\epsilon$ ball centered at $\cT'$ has border $\K$-rank at least $R$. To this end, suppose towards a contradiction that there is some tensor $\cS'$ such that $\|\cT'-\cS'\|<\epsilon$ and such that $\cS'$ has border $\K$-rank $N < R$. Also fix $\delta_1,\delta_2 > 0$ such that $\|\cT'-\cS'\|+\delta_1+\delta_2 < \epsilon$. Since $\cS'$ has border $\K$-rank $N$, there must exist some tensor $\cW'=[\![\bA',\bB',\bC']\!]$ such that $\|\cW'-\cS'\|<\delta_1$ and such that $\K$-$\rank(\cW')=\K$-$\underline{\rank}(\cW') = N$. Furthermore, a routine argument shows that $\cW$ can be chosen so that its factors $\bA'$ and $\bB'$ have full column rank. 

One can then extend the factors $\bA',\bB',\bC'$ to matrices 
\[
\tilde{\bA}'=[\bA' \ \ \bE_{\bA}'] \qquad \tilde{\bB}'=[\bB' \ \ \bE_{\bB}'] \qquad \tilde{\bC}'=[\bC' \ \  \bE_{\bC}']
\] 
each having $R$ columns where $\bE_{\bA}',\bE_{\bB}'$ and $\bE_{\bC}'$ are chosen so that $\tilde{\bA}'$ and $\tbB'$ have full column rank, $\tilde{\bC}'$ has a repeated column, and $\|\cW' - [\![\tbA',\tbB',\tbC']\!]\| < \delta_2$. Using Theorem \ref{theorem:rankvsmult} shows that the tensor $[\![\tbA',\tbB',\tbC']\!]$ has both $\K$-rank and $\K$-border rank equal to $R$. Furthermore, by construction we have
\[
\|\cT'-[\![\tbA',\tbB',\tbC']\!]\| < \|\cT'-\cS'\|+\delta_1+\delta_2 < \epsilon.
\]
However, Theorem \ref{theorem:rankvsmult} also shows that this tensor is not simple, since $\tbC'$ has repeated columns. This contradicts our assumption that every border rank $R$ tensor in an $\epsilon$ ball around $\cT'$ is simple.
\end{proof}

\subsection{Metrics for joint generalized eigenvalues}

We now discuss in more detail the measures of distance we use for joint generalized eigenvalues. Recall the chordal metric between  two one-dimensional subspaces $\scrL_1,\scrL_2 \subseteq \K^K$ is the sine of the angle between $\scrL_1$ and $\scrL_2$. Letting $\blam_i$ be a unit vector such that $\spann(\blam_i)=\scrL_i$ for $i=1,2$, it is not difficult to show that the chordal metric may be computed as
\beq
\label{eq:ChordalEqs}
\chi(\scrL_1,\scrL_2)=\|\bP_{(\scrL_1)^\perp} \blam_2 \|_2 = \max_{\substack{\bgam \in (\scrL_1)^\perp \\ \|\bgam \|_2=1}} |\langle \blam_2,\bgam \rangle | { = \sqrt{1- |\langle \blam_1,\blam_2 \rangle|^2.}}
\eeq
Here $(\scrL_1)^\perp \subset \K^K$ is the orthogonal complement of $\scrL_1 \subset \K^K$ and $\bP_{(\scrL_1)^\perp}$ denotes the projection onto $(\scrL_1)^\perp$. Note that $ |\langle \blam_1,\blam_2 \rangle|$ is the cosine of the angle between the subspaces $\scrL_1$ and $\scrL_2$.

Also recall the matching distance defined in equation \eqref{eq:MatchingDistanceDef} which gives a metric between the spectra of two tensors. In addition to the matching distance we will need a weaker measure of separation between the spectra of two tensors. To this end, for $i=1,2$, let $\cT_i \in \K^{R \times R \times K}$ be a tensor having $R$ JGE values counting algebraic multiplicity, and let $\{\scrL_{i,k}\}_{k=1}^R$ be the spectrum of $\cT_i$. The \df{spectral variation} from $\cT_1$ to $\cT_2$, denoted $\sv[\cT_1,\cT_2]$, is defined by
\[
\sv[\cT_1,\cT_2] = \max_i \min_j \chi(\scrL_{1,j}, \scrL_{2,i}).
\]

The key difference between spectral variation and matching distance is that, in matching distance, each JGE value of $\cT_2$ must be paired with a unique JGE value of $\cT_1$, while in spectral variation the JGE values of $\cT_2$ need not be paired with unique JGE values of $\cT_1$. Due to the lack of unique pairings in spectral variation, the spectral variation from $\cT_1$ to $\cT_2$ need not equal the spectral variation from $\cT_2$ to $\cT_1$. Furthermore, the spectral variation from $\cT_1$ to $\cT_2$ may be equal to zero while the spectra of $\cT_1$ and $\cT_2$ are not equal. As a consequence, while matching distance defines a metric on sets of JGE values, spectral variation does not. 

\begin{exa}
\label{exa:SVdefective}
Let $\cT_1,\cT_2 \in \K^{R \times R \times 2}$. Suppose $\cT_1$ has JGE values
\[
\scrL_{1,1}=\spann (\be_1) \mathand \scrL_{1,2}= \spann (\be_{ 2}) 
\]
and that $\cT_2$ has JGE values
\[
\scrL_{2,1} = \scrL_{2,2} =\spann (\be_1).
\]
Here $\be_1$ and $\be_2$ are the standard basis vectors in $\K^2$. Then one has 
\[
0 = \sv[\cT_1,\cT_2] \neq \sv[\cT_2, \cT_1] = 1. 
\]
\end{exa}

Of course, matching distance bounds are more desirable to obtain than spectral variation bounds; however, they are also typically much more difficult to obtain. A standard strategy to obtain a matching distance bound is to first find a spectral variation bound which grows linearly with the magnitude of a perturbation. 

\subsection{Bauer--Fike Theorem for JGE values}
We now present our main perturbation theoretic bound for the spectral variation between two tensors of $\K$-rank $R$. The bound we present is in essence determined by the spectral norm of (a transformation of) the error tensor.

\begin{theorem}
\label{theorem:TensBauerFike}
Let $\cT$ and $\cW$ be $R \times R \times K$ tensors with entries in $\K$. Assume that $\cT$ has $\K$-rank $R$ with CPD $\cpd$ where the columns of $\bC$ are normalized to have unit norm. Let $\{\scrL_k\}_{k=1}^K$ be the spectrum of $\cT$ and set $\cE = \cT -\cW$. If $\scrW$ is a JGE value of $\cW$ with geometric multiplicity at least one, then 
\beq
\label{eq:JGEvaluebound}
\min_k \chi (\scrL_k, \scrW) { \leq \sqrt{R} \Big\|\cE \cdot_1 \bA^{-1} \cdot_2 \bB^{-1} \cdot_3 \bP_{\scrW^\perp} \Big\|_{\mathrm{sp}} } \leq   \frac{\sqrt{R} \spnorm{{ \cE}}}{\sigma_{\min} (\bA) \sigma_{\min} (\bB)}. 
\eeq
Here $\sigma_{\min} (\bA)$ and $\sigma_{\min} (\bB)$ denote the smallest\footnote{Lemma \ref{lemma:invertslice} shows that $\bA$ and $\bB$ are invertible, hence these singular values are nonzero.} singular values of $\bA$ and $\bB$, respectively and $\bP_{\scrW^\perp}$ denotes the projection onto the orthogonal complement of $\scrW$. As a consequence, if $\cW$ is a slice mix invertible border $\K$-rank $R$ tensor, then
\beq
\label{eq:JGEBauerFike}  
\sv [\cT, \cW] {  \leq  \sqrt{R} \Big\| \cE \cdot_1 \bA^{-1} \cdot_2 \bB^{-1}  \Big\|_{\mathrm{sp}} } \leq \frac{\sqrt{R} \spnorm{{ \cE}}}{\sigma_{\min} (\bA) \sigma_{\min} (\bB)}.
\eeq
In the following special cases the coefficient $\sqrt{R}$ may be dropped from equations \eqref{eq:JGEvaluebound} and \eqref{eq:JGEBauerFike}$\mathrm{:}$ 
\begin{enumerate}
\item $K=2$.
\item $\cW$ is also slice mix invertible and has $\K$-rank $R$  with CPD $[\![\bA_\cW,\bB_\cW,\bC_\cW]\!]$ where $\bA_{\cW}=\bA$ or $\bB_\cW=\bB$.
\end{enumerate}
That is, in these special cases one has
\beq
\label{eq:JGEBauerFikeSpecial}  
\sv [\cT, \cW] {  \leq   \Big\| \cE \cdot_1 \bA^{-1} \cdot_2 \bB^{-1}  \Big\|_{\mathrm{sp}} } \leq \frac{ \spnorm{{ \cE}}}{\sigma_{\min} (\bA) \sigma_{\min} (\bB)}.
\eeq
\end{theorem}

In the general case, the proof of Theorem \ref{theorem:TensBauerFike} follows a similar strategy to the proof of \cite[Chapter VI, Theorem 2.7]{SS90} with modifications to address issues arising when $K>2$. In particular, \cite[Chapter VI, Theorem 2.7]{SS90} shows
\[
\min_k \chi (\scrL_k, \scrW)  \leq   \frac{\tnorm{{ \cE_{(1)}}}}{\sigma_{\min} (\bA) \sigma_{\min} (\bB)}
\]
in the $K=2$ case. As the proof in the general case is similar to Stewart and Sun's argument, we give it in the supplementary materials. However, if we additionally assume that $\cW$ has $\K$-rank $R$, then a new, simplified proof which is more intuitive can be given. This we now describe.

\begin{proof}

Assume that $\cW \in \K^{R \times R \times R}$ has rank $R$ and is slice mix invertible and let $[\![ \tbA, \tbB, \tbC ]\!]$ be a CPD of $\cW \cdot_1 \bA^{-1} \cdot_2 \bB^{-1}$ where the columns of $\tbC$ have unit norm.\footnote{Up to permutation and scaling one has $\tbA =\bA^{-1} \bA_{\cW}$ and $\tbB = \bB^{-1} \bB_{\cW}$ and $\tbC=\bC_\cW$.} Additionally set $\cE := \cW-\cT$.  As a consequence of Theorem \ref{theorem:rankvsmult} \eqref{it:RankIFFBasis}, the spans of the columns of $\bC$ and $\tbC$ give the JGE values of $\cT$ and $\cW$, respectively. For each column $\tbc_j$ of $\tbC$ we will show that the spectral norm of $\cE \cdot_1 \bA^{-1} \cdot_2 \bB^{-1} \cdot_3 \bP_\tbcjp$ gives an upper bound for
\[
\min_{r} \chi (\tbc_j,\bc_r) = \min_r \| \bP_{\tbcjp} \bc_r \|_2
\]
from which the result follows. Here $\bP_\tbcjp$ is the orthogonal projection on the orthogonal complement of $\spann(\tbc_j)$. 

Expanding out the spectral norm of the transformed error tensor, we have
\beq
\label{eq:ErrorSpectralProjection}
\spnorm{\cE \cdot_1 \bA^{-1} \cdot_2 \bB^{-1} \cdot_3 \bP_\tbcjp} = \max_{\|\bu\|_2=\|\bv\|_2=\|\bz\|_2=1} \bigg | \bigg(\sum_{r \neq j} \tba_r \otimes \tbb_r \otimes \left( \bP_\tbcjp \tbc_r \right) - \sum_{r=1}^R \be_r \otimes \be_r \otimes \left(\bP_\tbcjp \bc_r\right) \bigg) \cdot_1 \bu \cdot_2 \bv \cdot_3 \bz \bigg |.
\eeq

Note that  since $\bP_\tbcjp \tbc_j = \b0$, the tensor $\tcT \cdot_1 \bA^{-1} \cdot_2 \bB^{-1} \cdot_3 \bP_\tbcjp$ which appears on the left hand side of the difference in equation \eqref{eq:ErrorSpectralProjection} has rank at most $R-1$. It follows that there is some vector $\by \in \K^R$ which is orthogonal to the subspace spanned by the mode$-2$ fibers of this tensor. In particular, let $\by \in \K^R$ be a unit vector such that $\langle \by,\tbb_r \rangle =0$ for all $r \neq j$. With this choice of $\by$ we have
\[
\sum_{r \neq j}\langle \by,\tbb_r \rangle \tba_r  \otimes \left( \bP_\tbcjp \tbc_r \right) = 0
\]
from which it follows that
\beq
\begin{array}{rclcl}
\spnorm{\cE \cdot_1 \bA^{-1} \cdot_2 \bB^{-1} \by^{\mathrm{H}} \cdot_3 \bP_\tbcjp} & = & \max_{\|\bu\|_2=\|\bz\|_2=1} \left |\left(\sum_{r=1}^R \langle \by,\be_r \rangle \be_r  \otimes \left(\bP_\tbcjp \bc_r\right) \right) \cdot_1 \bu \cdot_3 \bz \right| \\
 & = & \spnorm{ \sum_{r=1}^R \langle \by,\be_r \rangle \be_r  \otimes \big(\bP_\tbcjp \bc_r\big) } \\
 & \geq &  \max_{r} \| \langle \by,\be_r \rangle \bP_\tbcjp \bc_r \|_2.
\end{array}
\eeq

Recalling that $\by$ is a unit vector, there must exist at least one index $r_*$ such that $ | \langle \by,\be_r \rangle | \geq 1/\sqrt{R}$, from which we obtain
\[
\spnorm{\cE \cdot_1 \bA^{-1} \cdot_2 \bB^{-1} \by^{\mathrm{H}} \cdot_3 \bP_\tbcjp} \geq\| \langle \by,\be_{r_*} \rangle \bP_\tbcjp \bc_{r_*} \|_2 \geq  \frac{\chi(\tbc_j,\bc_{r_*})}{\sqrt{R}} \geq \frac{\min_r \chi (\tbc_j,\bc_{r})}{\sqrt{R}}.
\]
We conclude 
\[
\sqrt{R} \tnorm{\bA^{-1}} \tnorm{\bB^{-1}} \spnorm{\cE}\geq \sqrt{R} \spnorm{\cE \cdot_1 \bA^{-1} \cdot_2 \bB^{-1} \by^{\mathrm{H}} \cdot_3 \bP_\tbcjp} \geq \min_r \chi (\tbc_j,\bc_{r}).
\]
The result in this special case follows from combining the above inequality with Theorem \ref{theorem:rankvsmult} \eqref{it:RankIFFBasis}. If additionally one has $\bB_\cW = \bB$, then one can take $\tbB = \bI$. The fact that the coefficient $\sqrt{R}$ may be dropped in this case easily follows. The case $\bA_\cW =\bA$ is similar.
\end{proof}

\subsubsection{Scaling of $\bA$ and $\bB$}

The quality of the bound obtained in Theorem \ref{theorem:TensBauerFike} depends on the scaling of the factor matrices $\bA$ and $\bB$ in the CPD $\cT=[\![\bA,\bB,\bC]\!]$. More precisely, for any diagonal matrix $\bD$, one has that $\cT=[\![\bA\bD,\bB\bD^{-1},\bC]\!]$ is a CPD of $\cT$, and the quantity of $\sigma_{\min}(\bA \bD) \sigma_{\min} (\bB \bD^{-1})$ is dependent on the choice of $\bD$. To obtain an optimal bound, one wants to minimize the right hand side of equation \eqref{eq:JGEBauerFike}, hence one wants to maximize $\sigma_{\min}(\bA \bD) \sigma_{\min} (\bB \bD^{-1})$ over invertible diagonal matrices $\bD$. 

One strategy to compute the optimal $\bD$ is to use an ALS approach in which one sets all but one  of the diagonal entries of $\bD$ to be equal to $1$, then maximizes over the remaining diagonal entry of $\bD$. By repeating this approach while varying which entry of $\bD$ is not fixed, one is able to obtain a (local) optimum for equation \eqref{eq:JGEBauerFike}.

As an alternative strategy, one may simply normalize $\bA$ and $\bB$ so that the $r$th column of $\bA$ and $\bB$ have the same norm for each $r=1,\dots,R$. Although this method cannot be expected to be optimal, it appears to perform well for randomly generated $\bA$ and $\bB$, and the associated computational cost is minimal.

\subsubsection{On the coefficient $\sqrt{R}$ }

Knowing that there are various special cases in which the coefficient $\sqrt{R}$ may be dropped from equation \eqref{eq:JGEBauerFike}, it is natural to wonder if there are examples where this coefficient is necessary, or if it is simply a relic of the proof. The following example shows that some constant is indeed necessary. That is, equation \eqref{eq:JGEBauerFike} does not hold for arbitrary pairs of slice mix invertible $\K$-rank $R$ tensors if the coefficient $\sqrt{R}$ is removed.

\begin{exa}
Let $\cT$ have CPD $\cpd$ where
\[
\bA = \bB = \bI_3 \mathand \bC = \begin{pmatrix} 1 & 0 & \frac{\sqrt{2}}{2} \\ 0 & 1 & \frac{\sqrt{2}}{2} \\ 0 & 0 & 0 \end{pmatrix}
\]
and let $\cW$ have CPD $[\![\tbA,\tbB,\tbC]\!]$ where
\[
\tbA =\tbB  = \begin{pmatrix} \frac{1}{10\sqrt{3}} & \frac{1}{\sqrt{3}} & 0 \\ \frac{1}{10\sqrt{3}} & -\frac{1}{2\sqrt{3}} & -\frac{1}{\sqrt{2}} \\ \frac{1}{10\sqrt{3}} & -\frac{1}{2\sqrt{3}} & \frac{1}{\sqrt{2}} 
\end{pmatrix} \mathand \tbC = \begin{pmatrix} 0 & 1 & 0 \\ 0 & 0 & 1 \\ 1 & 0 & 0 \end{pmatrix}
\]
Then $\cT$ and $\cW$ are both slice mix invertible $\R$-rank $3$ tensors and the spectral variation $\sv[\cT,\cW]$ is equal to $1$, as seen from the fact that the first column of $\tbC$ is orthogonal to all three columns of $\bC$. However, setting $\cE=\cW-\cT$ one has
\[
\Big\| \cE\cdot_1 \bA^{-1} \cdot_2 \bB^{-1} \Big\|_{\mathrm{sp}} < 1
\]
illustrating that equation \eqref{eq:JGEBauerFike} does not hold for this pair\footnote{Computation of the spectral norm is NP-hard \cite{HL13}; however, the mode three unfolding $\cE\cdot_1 \bA^{-1} \cdot_2 \bB^{-1}$ has spectral norm $\approx 0.983$ and this norm is as an upper bound for the tensor spectral norm. Using Tensorlab \cite{VDSBL16} suggests $\Big\| \cE\cdot_1 \bA^{-1} \cdot_2 \bB^{-1} \Big\|_{\mathrm{sp}} \approx 0.882$.}.
\end{exa}

\subsection{Matching distance bounds and the $K=2$ case}

While the spectral variation bound from Theorem \ref{theorem:TensBauerFike} provides insight as to how a CPD is affected by a perturbation, as illustrated by Example \ref{exa:SVdefective}, a spectral variation bound cannot be used to determine that a perturbation of a given tensor has distinct JGE values. To accomplish this we need to use a matching distance bound such as Theorem \ref{theorem:TensMDBound}.

As briefly mentioned before, a spectral variation bound can often be converted to a matching distance bound using the following idea:  Supposing temporarily that $\cW$ is slice mix invertible and has border $\K$-rank $R$ and letting $\{\scrL_r\}_{r=1}^R$ be the  spectrum of $\cT$, the bound given in Theorem \ref{theorem:TensBauerFike} gives a collection of neighborhoods $\{\mathfrak{N}_\epsilon (\scrL_r)\}_{r=1}^R$ in which the JGE values of $\cW$ must lie. Additionally, it is not difficult to show that JGE values are continuous in a perturbation along the set of slice mix invertible border rank $R$ tensors.\footnote{For example, one may use the fact that the characteristic polynomial of a slice mix invertible border $\K$-rank $R$ tensor factors as a product of $R$ linear terms to show this.}

More precisely, if $\mathfrak{T}:[0,1]\to \K^{R \times R \times K} $ is a continuous function such that $\mathfrak{T} (0) = \cT$ and $\mathfrak{T}(t)$ has border $\K$-rank $R$ and is slice mix invertible for all $t \in [0,1]$, then there exists a continuous function $\mathfrak{L}:[0,1] \to [\text{Gr}(1,\K^K)]^R$ such that $\mathfrak{L}(t)$ is equal to the  spectrum of $\mathfrak{T}(t)$ for all $t \in [0,1]$. Here $\text{Gr}(1,\K^K)$ denotes the set of one-dimensional subspaces of $\K^K$ and $[\text{Gr}(1,\K^K)]^R$ denotes a collection of $R$ elements of $\text{Gr}(1,\K^K)$ with repetition allowed.\footnote{The collection of one-dimensional subspaces of $\K^K$ is a metric space under the chordal metric. In general, the collection of $d$-dimensional subpsaces of $\K^K$ is called a Grassmanian.}

The above discussion implies that if $\cW$ is slice mix invertible and has border $\K$-rank $R$ and if there is a path of slice mix invertible border $\K$-rank $R$ tensors from $\cT$ to $\cW$ lying in the ball of radius $\spnorm{\cT-\cW}$ (i.e. if $\mathfrak{T}$ can be chosen so $\mathfrak{T}(0)=\cT$ and $\mathfrak{T}(1) = \cW$ and $\spnorm{\cT-\mathfrak{T}(t)} \leq  \spnorm{\cT-\cW}$ for all $t \in [0,1]$) and if the $\mathfrak{N}_\epsilon (\scrL_i) \cap \mathfrak{N}_\epsilon (\scrL_j)=\emptyset$ for each $i \neq j$, then each $\mathfrak{N}_\epsilon (\scrL_r)$ must contain exactly one JGE value of $\cW$. As a consequence, if these assumptions hold, then the bound given by \ref{theorem:TensBauerFike} is in fact a matching distance bound and may be used to guarantee the existence of distinct JGE values for $\cW$. See Figure \ref{fig:EigPertFig} for an illustration of the functions $\mathfrak{T}$ and $\mathfrak{L}$ and of the neighborhoods $\NepsL{r}$.

\begin{figure}
\label{fig:EigPertFig}
\hspace*{-1cm}
	\scalebox{.8}{
\begin{tikzpicture}

\path[name path=border1] (-3.75,.25) to[out=-10,in=150] (-1.5,-0.5);
\path[name path=border2] (12,1) to[out=150,in=-10] (5.5,3.2);
\path[name path=border3] (1,1) to[out=-10,in=150] (5.5,3.2);
\path[name path=border4] (0.25,1.25) to[out=150,in=-10] (-2.25,1.6);
\path[name path=redline] (-6.5,-1) -- (5.5,3.65);
\path[name intersections={of=border1 and redline,by={a}}];
\path[name intersections={of=border4 and redline,by={b}}];

\shade[left color=gray!10,right color=gray!80] 
  (-4.25,-.25) to[out=-30,in=130] (-.75,-1.25) -- (1,1.75) to[out=130,in=-30] (-2.5,2.75) -- cycle;

\path[name path=border5] (5.75,.5) to[out=-0,in=150] (7.5,-0.5);
\path[name path=border6] (6.5,.5) to[out=-10,in=-50] (10.25,1/2);
\path[name path=border7] (9.75,1.55) to[out=00,in=-10] (7.25,1.9);
\path[name path=border8] (6.5,1.55) to[out=-100,in=-10] (9.5,1.5);
\path[name path=lambda1line] (3.5,-1) -- (11,4);
\path[name path=lambda2line] (4.50,-.92) -- (14.5,3.65);
\path[name intersections={of=border5 and lambda1line,by={c}}];
\path[name intersections={of=border7 and lambda1line,by={d}}];
\path[name intersections={of=border6 and lambda2line,by={e}}];
\path[name intersections={of=border8 and lambda2line,by={f}}];

  \draw[violet,line width=1.5pt,shorten >= 3pt,shorten <= 0pt] 
  (e) .. controls (9,-1) and (8,1) ..
  coordinate[pos=.2] (eux1) 
  coordinate[pos=0.45] (eux2) 
  coordinate[pos=.65] (eux3) (f); 
  
\draw[blue,line width=1.5pt,shorten >= 3pt,shorten <= 3pt] 
  (c) .. controls (7,1) and (5,2) ..
  coordinate[pos=0.4] (dux1) 
  coordinate[pos=0.6] (dux2) 
  coordinate[pos=0.8] (dux3) (d);

\shade[left color=gray!10,right color=gray!80] 
 (3.75,-.75) to[out=10,in=150] (8.75,-1.75) -- (10.5,1.55) to[out=150,in=10] (5.5,2.55) -- cycle;
  

\draw[red,line width=1.5pt,shorten >= 3pt,shorten <= 3pt] 
  (a) .. controls (0,-1) and (-4,2) ..
  coordinate[pos=0] (cux1) 
  coordinate[pos=0.5] (cux2) 
  coordinate[pos=1] (cux3) (b);

  
  \draw[color=violet!50,line width=.75pt] (eux1) circle (.924);
  
   \draw[color=blue!50,line width=.75pt] (dux1) circle (.9);

\draw[line width=1.1pt] (1.4,-0.8) to
  coordinate[pos=0] (aux1) 
  coordinate[pos=0.5] (aux2) 
  coordinate[pos=1] (aux3) (3.5,2.3);

\node [xshift=10pt,yshift=-5pt] at (aux1) {$0$};
\node [xshift=10pt,yshift=-5pt] at (aux2) {$\frac{1}{2}$};
\node [xshift=10pt,yshift=-5pt] at (aux3) {$1$};

\node[name=tele] [xshift=0pt,yshift=35pt] at (aux3) {$t\in [0,1]$};
\node[name=lele] [xshift=130 pt,yshift=35pt] at (aux3) {$\mathfrak{L}(t)$};
\node[name=tauele] [xshift=-130 pt,yshift=35pt] at (aux3) {$\mathfrak{T}(t)$};

\draw[->,dashed] (tele) edge[bend left] (lele);
\draw[->,dashed] (tele) edge[bend left] (tauele);
    
\node [xshift=65pt,yshift=15pt] at (tele) {$\mathfrak{L}$};
\node [xshift=-65pt,yshift=-15pt] at (tele) {$\mathfrak{T}$};

\node [xshift=-7pt,yshift=5pt,red] at (cux1) {$\mathcal{T}$};
\node [xshift=-15pt,yshift=5pt,red] at (cux2) {$\mathfrak{T}(\frac{1}{2})$};
\node [xshift=-2pt,yshift=10pt,red] at (cux3) {$\hat{\mathcal{T}}$};

\node [xshift=9pt,yshift=-4pt,blue] at (dux1) {$\mathscr{L}_1$};
\node [xshift=-30pt,yshift=13pt,blue] at (dux2) {$\mathfrak{L}(\frac{1}{2})$};
\node [xshift=11pt,yshift=9pt,blue] at (dux3) {$\hat{\mathscr{L}}_1$};

\node [xshift=-7pt,yshift=-7pt,violet] at (eux1) {$\mathscr{L}_2$};
\node [xshift=6pt,yshift=-8.5pt,violet] at (eux2) {$\mathfrak{L}(\frac{1}{2})$};
\node [xshift=6pt,yshift=11pt,violet] at (eux3) {$\hat{\mathscr{L}}_2$};

\node [name=neps1,xshift=-40pt,yshift=-45pt,blue] at (dux1) {$\mathfrak{N}_\epsilon(\mathscr{L}_1)$};
\node [name=neps2,xshift=48pt,yshift=40pt,violet] at (eux1) {$\mathfrak{N}_\epsilon(\mathscr{L}_2)$};


\draw[->,dashed,shorten >=0.95cm,blue!80] (neps1)to (dux1);
\draw[->,dashed,shorten >=0.95cm,violet!80] (neps2)to (eux1);

\node at (-2,-2) {Set of border $\mathbb{K}$-rank $R$ tensors};

\node at (6.8,-2) {$\mathrm{Gr}(1,\mathbb{K}^K)$};

  \draw[violet,line width=1.5pt] (eux1)to[bend left] (eux2);
  \draw[violet,line width=1.5pt] (eux2)to[bend right] (eux3);
  
  \draw[blue,line width=1.5pt] (dux1)to[bend left] (dux2);
  \draw[blue,line width=1.5pt] (dux2)to[bend left] (dux3);


  
  \draw[->,dashed,shorten >=0.15cm,blue] (aux1)to[bend left] (dux1);
  \draw[->,dashed,shorten >=0.15cm,blue] (aux2)to[bend left] (dux2);
  \draw[->,dashed,shorten >=0.15cm,blue] (aux3)to[bend left=20] (dux3);
  

\foreach \coor in {1,2,3}
  \draw[->, dashed,shorten >=0.15cm,red] (aux\coor)to[bend left] (cux\coor);

  
  \draw[->,dashed,shorten >=0.15cm,violet] (aux1)to[bend left=35] (eux1);
  \draw[->,dashed,shorten >=0.15cm,violet] (aux2)to[bend left=19] (eux2);
  \draw[->,dashed,shorten >=0.15cm,violet] (aux3)to[bend left=60] (eux3);

\foreach \coor in {1,2,3}
  \draw[fill=black] (aux\coor) circle (3pt);

\foreach \coor in {1,2,3}
  \draw[fill=red] (cux\coor) circle (3pt);
  
\foreach \coor in {1,2,3}
  \draw[fill=blue] (dux\coor) circle (3pt);
  
  \draw[fill=violet] (eux1) circle (3pt);
  \draw[fill=violet] (eux2) circle (3pt);
  \draw[fill=violet] (eux3) circle (3pt);

\end{tikzpicture}}
\caption{Illustration of continuous functions $\mathfrak{T}$ and $\mathfrak{L}$ defined on $[0,1]$. The function $\mathfrak{T}$ maps to a path of border $\K$-rank $R$ tensors from $\cT$ to $\hcT$. The function $\mathfrak{L}$ is defined so that $\mathfrak{L}(t)$ is equal to the joint generalized spectrum of $\mathfrak{T}(t)$ for all $t \in [0,1]$.  All JGE values of $\hat{\cT}$ are contained in the union of the regions $\NepsL{1}$ and $\NepsL{2}$. Furthermore the neighborhood $\NepsL{r}$ contains a JGE value of $\hat{\cT}$ for each $r=1,2$.}
\end{figure}

The primary issue with using such an argument for general $K$ is that, it is not known if there is a continuous function $\mathfrak{T}$ which satisfies the necessary assumptions. That is, it is not known if there is a path of slice mix invertible border rank $R$ tensors from $\cT$ to $\cW$ which lies in the ball of radius $\spnorm{\cT-\cW}$ centered at $\cT$. However, in the case that $K=2$, the condition needed to guarantee that the $\mathfrak{N}_\epsilon (\scrL_r)$ do not intersect in fact guarantees the existence of such a path.

As discussed in \cite[Chapter VI]{SS90}, if $\cT$ is a slice mix invertible tensor of size $R \times R \times 2$, then the (joint) generalized eigenvalue problem for $\cT$ is equivalent to a classical eigenvalue problem. E.g. if $\bT_1$ is invertible, then the JGE eigenvectors of $\cT$ are exactly equal to the classical eigenvectors of $\bT_1^{-1} \bT_2$. Using this equivalence, we for free obtain basic facts about multiplicities of JGE values for slice mix invertible $R \times R \times 2$ tensors. Namely, a slice mix invertible $R \times R \times 2$ tensor always has $R$ JGE values in $\C^2$ counting algebraic multiplicity, and each JGE value of such a tensor must have geometric multiplicity at least one. We will make frequent use of these facts in the upcoming results. 

\begin{proposition}
\label{proposition:RegularityBound}
Let $\cT \in \K^{R \times R \times 2}$ be a tensor of $\K$-rank $R$ with CPD $\cT=\cpd$, where $\bC$ has columns of unit norm and let $\{\scrL_r\}_{r=1}^R$ be the set of JGE values of $\cT$. Let $\cE \in \K^{R \times R \times 2}$ be any tensor which satisfies
\[
\spnorm{\cE} \leq \frac{\sigma_{\min} (\bA) \sigma_{\min} (\bB) \min_{i \neq j} \chi (\scrL_i, \scrL_j)}{2}.
\]
Then $\cT+\cE$ is slice mix invertible and has border $\C$-rank equal to $R$. 
\end{proposition}
\begin{proof}
To simplify notation, set
\[
\delta = \min_{i \neq j} \frac{\chi (\scrL_i,\scrL_j)}{2}.
\]
By assumption the columns of $\bC$ all have unit norm, therefore, by reindexing columns if necessary, Theorem \ref{theorem:rankvsmult} shows that for each $r=1,\dots,R$ the column $\bc_r$ satisfies $\scrL_r=\spann(\bc_r)$. To show $\cT+\cE$ is slice mix invertible, we first show that there is a unit vector $\bq\in \K^2$ such that $|\langle \bc_r, \bq \rangle| \geq \delta$ for all $r$. Using this we will show that the matrix $\bar{q}_1 (\bT_1+\bE_1)+ \bar{q}_2 (\bT_2+\bE_2)$ is invertible, hence $\cT+\cE$ is slice mix invertible.

To this end let $\scrZ \subset \K^2$ be a one-dimensional subspace of $\K^2$ such that $\chi(\scrL_1,\scrZ)=\delta$. Then using the triangle inequality we must have
\beq
\label{eq:PsiChordalGap}
\chi(\scrL_r,\scrZ) \geq \delta \qquad \mathrm{for \ all \ } r=1,\dots, R,
\eeq
as otherwise would imply that there is an index $\ell \neq 1$ such that
\[
\chi(\scrL_\ell,\scrZ)<\delta \qquad \mathrm{hence} \qquad \chi(\scrL_1,\scrL_\ell) \leq \chi(\scrL_1,\scrZ)+\chi(\scrL_\ell,\scrZ)< 2\delta= \min_{i \neq j} \chi (\scrL_i,\scrL_j),
\]
which is a contradiction.

Now set $\scrQ=\scrZ^\perp \subset \K^2$ and let $\mathbf{q} \in \K^2$ be a unit vector such that $\scrQ = \spann(\mathbf{q})$. Using equation \eqref{eq:PsiChordalGap}, for each $r$ we have
\[
\delta \leq \chi(\scrL_\ell,\scrZ)= \sin (\angle (\scrL_r,\scrZ))=\sin(\pi/2 -\angle (\scrL_r,\scrQ))= \cos (\angle (\scrL_r,\scrQ))=|\langle \bc_r, \mathbf{q} \rangle|
\]
where $\angle(\cdot,\cdot)$ denotes the angle between two one-dimensional subspaces of $\K^2$. 

Letting $D_k (\bC)  \in \K^{R \times R}$ be the diagonal matrix whose diagonal entries are given by the $k$th row of $\bC$ as usual, we have that $\bar{q}_1 D_1 (\bC) +\bar{q}_2 D_2 (\bC)$ is a diagonal matrix with $r$th diagonal entry equal to $\langle \bc_r, \mathbf{q} \rangle$. If follows that 
\beq
\label{eq:OmegaSingVal}
\sigma_{\min} (\bar{q}_1 D_1 (\bC) +\bar{q}_2 D_2 (\bC))= \min (\{|\langle \bc_r, \mathbf{q} \rangle|\}_{r=1}^R) \geq \delta.
\eeq
Finally observe that
\[
\begin{array}{rclcl}
\sigma_{\min} \big(\bar{q}_1 (\bT_1+\bE_1)+ \bar{q}_2 (\bT_2+\bE_2)\big) &\geq& \sigma_{\min} (\bar{q}_1 \bT_1 +\bar{q}_2 \bT_2) - \|\bar{q}_1 \bE_1 +\bar{q}_2 \bE_2\|_2 \\
 &\geq & \sigma_{\min} \Big(\bA \Big(\bar{q}_1 D_1 \big(\bC)+\bar{ q}_2 D_2 (\bC)\Big) \bB^\mathrm{T}\Big)-\spnorm{\cE} \\
 & \geq & \sigma_{\min} (\bA) \sigma_{\min} (\bB) \sigma_{\min}\big(\bar{q}_1 D_1 \big(\bC)+\bar{q}_2 D_2 (\bC)\big)- \spnorm{\cE}\\ 
\end{array}
\]

Using inequality \eqref{eq:OmegaSingVal} together with the assumption that $\spnorm{\cE}< \sigma_{\min} (\bA) \sigma_{\min} (\bB)  \delta$ shows that $\cT+\cE$ is slice mix invertible.

Having shown $\cT+\cE$ is slice mix invertible, it follows that $\cT+\cE$ has multilinear rank $(R,R,\cdot)$. From equation \eqref{eq:BRankVsMLRank} we conclude that the border rank of $\cT+\cE$ is at least $R$. Combining this with the well-known fact that tensors of size $R \times R \times 2$ have border rank less than or equal to $R$, e.g. see \cite[Theorem 5.5.1.1]{Lan12}, shows that $\cT+\cE$ has border rank $R$.
\end{proof}

\subsection{Proof of Theorem \ref{theorem:TensMDBound}}

We now give the proof of Theorem \ref{theorem:TensMDBound}.

\begin{proof}

Before proceeding we note that the proof of the $\K=\R$ case makes use of the result over $\C$, so we first consider the complex case. Now set
\beq
\epsilon=\frac{\spnorm{\cT-\cW}}{\sigma_{\min} (\bA) \sigma_{\min} (\bB)} < \frac{ \min_{i \neq j} \chi (\scrL_i, \scrL_j)}{2}
\eeq
and for each $r=1,\dots,R$ let $\NepsL{r}$ be the collection of one-dimensional subspaces in $\C^2$ defined by
\[
\NepsL{r} = \{ \scrZ \subset \C^2 | \ \dim(\scrZ)=1 \ \mathrm{and} \ \chi(\scrL_r,\scrZ) \leq \epsilon\}.
\]
That is, $\NepsL{r}$ is the ball of chordal radius $\epsilon$ centered at $\scrL_r$. Then our choice of $\epsilon$ together with the triangle inequality gives
\[
\NepsL{i} \cap \NepsL{j} = \emptyset \quad \quad \mathrm{for \ all \ } i \neq j.
\]
To prove inequality \eqref{eq:MatchingDistanceBound}, it is sufficient show that each $\NepsL{r}$ contains exactly one JGE value of $\cW$. 

To this end let $\cE=\cW-\cT$ and for $\alpha \in [0,1]$ set $\cT^\alpha=\cT+\alpha \cE$ so that $\cT=\cT^0$ and $\cW=\cT^1$. Note that Proposition \ref{proposition:RegularityBound} shows that $\cT^\alpha$ is slice mix invertible and has border $\C$-rank $R$ for each $\alpha \in [0,1]$. It follows from Theorem \ref{theorem:rankvsmult} that each $\cT^{\alpha}$ has $R$ JGE values in $\C^2$ counting algebraic multiplicity. 

Now define
\[
\NepsLal{r}= \{ \scrZ \subset \C^2 | \ \dim(\scrZ)=1 \ \mathrm{and} \ \chi(\scrL_r,\scrZ) \leq \alpha \epsilon\}.
\]
Then $\NepsLal{r} \subset \NepsL{r}$ for each $\alpha \in [0,1]$ and $r$ so $\NepsLal{i} \cap \NepsLal{j} =\emptyset$ whenever $i \neq j$. Furthermore, using Theorem \ref{theorem:TensBauerFike} we have 
\[
\sv[\cT,\cT^\alpha] \leq \frac{\spnorm{\alpha \cE}}{\sigma_{\min} (\bA) \sigma_{\min} (\bB)}  =\alpha \epsilon.
\]
It follows that the spectrum of $\cT^\alpha$ is contained in $\cup_{r=1}^R \NepsLal{r}$ for each $\alpha \in [0,1]$. 

In particular $\NepsLzero{r}=\{\scrL_r\}$ contains exactly one JGE value of $\cT^0=\cT$ for each $r=1,\dots, R$. Since $\cT^\alpha$ is slice mix invertible for all $\alpha \in [0,1]$, \cite[Chapter VI, Theorem 2.1]{SS90} shows that the JGE values of $\cT^\alpha$ are continuous in $\alpha$ for $\alpha \in [0,1]$. Therefore, because the $\NepsLal{r}$ remain disjoint as $\alpha$ increases from $0$ to $1$, the continuity of the JGE values of $\cT^\alpha$ in $\alpha$ implies that each $\NepsLal{r}$ must contain exactly one JGE value of $\cT^\alpha$ for each $\alpha \in [0,1]$ and for each $r=1,\dots,R$. In particular, $\NepsLone{r}=\NepsL{r}$ contains exactly one eigenvalue of $\cT^1=\cW$ for each $r$. By definition, each $\NepsL{r}$ is a ball of chordal radius $\epsilon$ centered $\scrL_r$, so we conclude
\[
\md[\cT,\cW] \leq \epsilon = \frac{\spnorm{\cT-\cW}}{\sigma_{\min} (\bA) \sigma_{\min} (\bB)},
\]
and that all the JGE value of $\cW$ have algebraic multiplicity one, as claimed.

It remains to show that $\cW$ has $\C$-rank $R$.  As previously discussed, all JGE values of a slice mix invertible $R \times R \times 2$ tensor must have geometric multiplicity at least $1$. It follows that $\cW$ is simple, so Theorem \ref{theorem:rankvsmult} \eqref{it:NonderogatoryImpliesRankR} shows that $\cW$ has $\C$-rank $R$. 

Now assume that $\K=\R$ and that $\cT$ and $\cW$ have real entries. Applying the result over the complex field shows that $\cW$ must have distinct JGE values over $\C$ so it is sufficient to show that the JGE values of $\cW$, hence the JGE vectors of $\cW$, must all be real. Here, when we say that $\cW$ has real JGE values, we mean that the JGE values of $\cT$ can be expressed as the spans of real valued vectors. Equivalently, the classical generalized eigenvalues of the pencil $(\bW_1,\bW_2)$ are real valued.

Suppose towards a contradiction that $\cW$ has a JGE value which is not real and set
\[
\alpha_*= \inf_{\alpha \in [0,1]} \cT+\alpha \cE {\mathrm{\ has \ a \ JGE \ value \ which \ is \ not \ real }}. 
\]
Using the continuity for JGE values in $\alpha$, it is straightforward to show that $\cT+\alpha_0 \cE$ has a JGE value with algebraic multiplicity at least two. However, $\cT+\alpha_0 \cE$ satisfies the assumptions of Theorem \ref{theorem:TensMDBound}, so this is a contradiction as we have already shown that all JGE values of such a tensor must have algebraic multiplicity equal to one. We conclude that the JGE values and JGE vectors of $\cW$ are {all real}. It follows from Theorem \ref{theorem:rankvsmult} \eqref{it:RankIFFBasis} that $\cW$ has $\R$-rank $R$. 
\end{proof}

\section{Deterministic bounds for existence of best rank $R$ decompositions}
\label{sec:Bounds}

We now turn to proving the deterministic bounds for existence of best $\K$-rank $R$ approximations presented in the introduction of the article. While we do not have a matching distance bound in the case of general $K$, a key observation is that if the tensor $\cT$ is defective in the sense of JGE values, then every subpencil of $\cT$ must have a JGE value of algebraic multiplicity greater than one.

\begin{proposition}
\label{proposition:DefectiveSubpencils}
Let $\cT \in \K^{R \times R \times K}$  be a slice mix invertible tensor. Also let $\bU \in \K^{K \times K}$ be any invertible matrix and set $\cS = \cT \cdot_3 \bU$. If $\cT$ has a JGE value of algebraic multiplicity $m$, then for any $i \neq j$ the subpencil
\[
\big(\bS_i,\bS_{{ j}} \big)
\]
must have a JGE value of algebraic multiplicity at least equal to $m$.
\end{proposition}

Before giving the proof we give a lemma which relates the joint generalized eigenvalues and eigenvectors of $\cT$ and $\cT \cdot_3 \bU$. 

\begin{lemma}
\label{lemma:EigenvalueModes}
Let $\cT \in \K^{R \times R \times K}$ be a slice mix invertible tensor and let $\spann(\blam)$ be a JGE value of $\cT$. Also let $\bU \in \K^{K \times K}$ be an invertible matrix and set $\cS = \cT \cdot_3 \bU$. Then $\cS$ is slice mix invertible  and $\spann(\bU \blam)$ is a JGE value of $\cS$ which satisfies $\am(\blam)= \am({\bU} \blam)$ where $\spann(\blam)$ and $\spann( {\bU} \blam)$ are treated as JGE values of $\cT$ and $\cS$, respectively. 

Furthermore, if $\spann(\blam)$ has a corresponding JGE vector $\bx$, then $(\spann({\bU} \blam),\bx)$ is a JGE pair of $\cS$. As a consequence one has $\gm(\blam) = \gm({\bU} \blam)$. 
\end{lemma}

\begin{proof}
Assume $\spann(\blam)$ is a JGE value of $\cT$ with algebraic multiplicity $m$. View $\bgam = (\gamma_1, \dots, \gamma_K)$ as a vector in $\K^K$, and consider the change of variables $\bgam \to \bU^{{ T}}\bgam$. It is a straightforward to verify that
\[
p_{\cT \cdot_3 \bU} (\bgam)= p_{\cT} (\bU^{{ T}} \bgam). 
\]
 Since $\bU$ is invertible, it follows that $p_{\cS} (\bgam) \not\equiv 0$, hence $\cS$ is slice mix invertible. Now set $\bU^{{ T}} \bgam = \eta$. By assumption $p_{\cT} (\eta)=p_{\cT \cdot_3 \bU} (\bgam)$ factors as 
\[
p_{\cT} (\eta)= \langle \blam, \bar{\eta} \rangle ^m g(\eta) = \langle { \bU} \blam, \bar{\bgam} \rangle ^m g(\bU^{{ T}} \bgam),
\]
so $\spann({ \bU \blam})$ is a JGE value of $\cT \cdot_3 \bU$ with algebraic multiplicity $m$.

The proof for geometric multiplicity is straightforward.
\end{proof}

We now give the proof of Proposition \ref{proposition:DefectiveSubpencils}.

\begin{proof}
By assumption $\cT$ is slice mix invertible and has a JGE value of algebraic multiplicity $m$, so Lemma \ref{lemma:EigenvalueModes} shows that $\cS$ is slice mix invertible and has a JGE value with algebraic multiplicity $m$. Say this JGE value is equal to $\spann(\blam)$. It follows that the characteristic polynomial of $\cS$ factors as
\[
p_{\cS}(\bgam)=\langle \blam, \bar{\bgam} \rangle^m g(\bgam)
\]
where $g(\bgam)$ is a nonzero polynomial. 

To complete the proof, assume $i < j$ and observe that the characteristic polynomial of the subpencil $(\bS_i,\bS_{{ j}})$ is given by
\[
\begin{array}{rclcl}
p_{(\bS_i,\bS_{{ j}})}((\gamma_{i},\gamma_{{ j}}))&=&p_{\cS} (0,\dots, 0, \gamma_i,0 \dots, 0, \gamma_{{ j}},0 \dots, 0)\\
&=&(\bar{\gamma}_i \lambda_i +\bar{\gamma}_{{ j}} \lambda_{{ j}})^m g(0,\dots, 0, \gamma_i,0 \dots, 0, \gamma_{{ j}},0 \dots, 0)
\end{array}
\]
hence $\spann((\lambda_i,\lambda_{{ j}}))$ is a generalized eigenvalue of the subpencil $(\bS_i,\bS_{{ j}})$ with multiplicity at least $m$. 
\end{proof}

\subsection{Proof of Theorem \ref{theorem:MultiplePencilBound}}
\label{sec:MultiplePencilBoundProof}
We are now in position to prove our multiple pencil based bound for the existence of a best $\K$-rank $R$ approximation.

\begin{proof} Let $\cM' \in \K^{I_1 \times I_2 \times I_3}$ be any tensor such that 
\[
\|\cT'-\cM'\|_\mathrm{F} {\  < \ } \epsilon/2
\]
and let $\hcT'$ be a best border $\K$-rank $R$ approximation of $\cM'$ in the Frobenius norm. Since $\cT'$ has rank $R$ by assumption, we must have
\[
\|\hcT'-\cM'\|_\mathrm{F} \leq \|\cT'-\cM'\|_\mathrm{F} \qquad \mathrm{hence} \qquad \|\hcT'-\cT'\|_\mathrm{F}  {\  < \ } \epsilon.
\]

Since $\hcT'$ has border $\K$-rank $R$, equation \eqref{eq:BRankVsMLRank} shows that the multilinear rank of $\hcT'$ must be $(R,R,\hat{K})$ for some $\hat{K} \leq R$. It is easy to check that a tensor has $\K$-rank $R$ if and only if an orthogonal compression of the tensor has $\K$-rank $R$, so using Theorem \ref{theorem:TensorProcrustes}, we may WLOG assume that $\cT'=\cT$ and $\hcT'=\hcT$ both have size $R \times R \times \max\{K,\hat{K}\}$. There is  little difference in the $K \geq \hat{K}$ and $K< \hat{K}$ cases, so for ease of notation we assume $K \geq \hat{K}$. 

Now let $\bU \in \K^{K \times K}$ be a unitary and set $\cS=\cT\cdot_3 \bU$ and $\hcS=\hcT\cdot_3 \bU$. Since modal unitary multiplications preserve the Frobenius norm of a tensor we have $\|\cS - { \hcS}\|_\mathrm{F}  = \| \cT -\hcT \|_\mathrm{F} < \epsilon$. Therefore, for some $i=1,\dots, \lfloor K/2 \rfloor$ we must have that
\[
 \|(\bS_{2i-1},\bS_{2i})-({ \hat{\bS}}_{2i-1}, { \hat{\bS}}_{2i})\|_\mathrm{F} < \epsilon_i,
\]
as otherwise would imply
\[
\|\cS -  \hcS\|_\mathrm{F}^2 \geq  \sum_{k=1}^{\lfloor K/2 \rfloor} \|(\bS_{2k-1}, \bS_{2k})-({\hat{\bS}}_{2k-1}, {\hat{\bS}}_{2k})\|^2_\mathrm{F} \geq   \sum_{k=1}^{\lfloor K/2 \rfloor} \epsilon_k^2 = \epsilon^2.
\]

In particular, from our choice of $\epsilon_i$ the subpencil $({\hat{\bS}}_{2i-1},{\hat{\bS}}_{2i})$ must be simple. Using Proposition \ref{proposition:DefectiveSubpencils} allows us to conclude that the tensor ${\hcT}$ is simple. The result then follows from Proposition \ref{proposition:existenceCor}.
\end{proof}

One may notice that there is no guarantee that the tensor $\cT\cdot_3 \bU$ has a slice mix invertible subpencil  for a particular choice of $\bU$, hence the bound given by Theorem \ref{theorem:MultiplePencilBound} may be equal to $0$ for a particular choice of $\bU$. This does not cause serious difficulty: As a consequence of \cite[Section 4.4 (ii)]{DD14}, if $\cT$ is slice mix invertible then for generic unitaries every subpencil of $\cT$ is slice mix invertible, hence generically the computed bound is nonzero. 

This does point toward a greater issue, however. Different choices of unitaries lead to different values for the bound computed by Theorem \ref{theorem:MultiplePencilBound}. Of course, the result holds for each of these values, so one desires to use a unitary which gives the largest possible value of $\epsilon$. Outside of a few special cases, is not known how to optimally pick the unitary $\bU$.  Determining how to optimally pick $\bU$ is a direction for future research.

A subproblem of picking an optimal unitary is picking the optimal ordering of slices for a fixed unitary $\bU$. E.g. if $\cT \in \K^{4 \times 4 \times 4}$ and $\bU \in K^{4 \times 4}$, then for $\cS=\cT\cdot_3 \bU \in \K^{4 \times 4 \times 4}$ one may arrive at different bounds by considering collections of subpencils of the form $\{(\bS_1, \bS_2), (\bS_3,\bS_4)\}$ or $\{(\bS_1,\bS_3), (\bS_2,\bS_4)\}$ or $\{(\bS_1,\bS_4), (\bS_2,\bS_3)\}$. 

Let $\bE_\bU$ be the matrix with $i,j$ entry equal to the bound given by Theorem \ref{theorem:TensMDBound} for the pencil $(\bS_i,\bS_j)$. Then choosing an optimal collection of subpencils (i.e. an optimal ordering for the slices of $\cS$) is equivalent to marking the entries of $\bE_\bU$ so that there is at most one mark in each row and column of $\bE_\bU$ and so that the norm of the vector of marked entries is maximal.

This problem is  a variant of the longest path problem which is, in a sense, the ``opposite" of the travelling salesman problem (where the total ``distance" is minimized). The longest path problem is known to be NP-hard \cite{UU04}; for our purposes and for small $K$, we may nevertheless consider trying all possible pencil collections/paths. 

\section{Numerical experiments}
\label{sec:figures}

In this section, numerical experiments are used to examine the quality of the main perturbation bounds obtained in the paper, i.e. the spectral variation bound obtained in Theorem \ref{theorem:TensBauerFike} and the existence bounds obtained in Theorem \ref{theorem:MultiplePencilBound} and Theorem \ref{theorem:MLSVDExistenceBound}. In addition we illustrate the effect of a tensor Procrustes calculation as in Theorem \ref{theorem:TensorProcrustes} on the error between a rank $R$ tensor and a rank $R$ approximation of a perturbation of the tensor. Throughout the section, we define the signal to noise ratio (SNR) of a signal tensor $\cT'$ to a noise tensor $\cN'$ to be $20 \log_{10} \left(\frac{\|\cT'\|_\mathrm{F}}{\|\cN'\|_\mathrm{F}}\right).$ The experiments discussed in this section all restrict to the real setting. That is, we restrict to considering real valued tensors and real rank.

In all of our experiments, rank $R$ tensors of size $I \times I \times I$ are randomly generated by generating factor matrices $\bA',\bB',\bC'$ of size $I \times R$ with entries independently sampled from a standard normal distribution. We then scale the factor matrices so that each rank $R$ tensor has unit Frobenius norm. The entries of the noise tensors $\cN'$ are all independently sampled from a standard normal distribution and each $\cN'$ is scaled so that the resulting SNR is as specified.

\subsection{Observed spectral variation vs. spectral variation bound}

We first examine the spectral variation bound obtained from the spectral norm\footnote{As previously mentioned, computation of the spectral norm is NP-hard. Using Tensorlab \cite{VDSBL16}, one may compute a (local) optimum for the best rank $1$ approximation of the transformed error. The Frobenius norm of a best rank $1$ approximation to a tensor is equal to the spectral norm of the tensor \cite{FMPS13}.} of (a transformation of) the error tensor in Theorem \ref{theorem:TensBauerFike} and compare it to the observed spectral variation occurring under a perturbation with a specified SNR. Three cases are treated: First, we consider the case of structured perturbations, i.e. perturbations only affecting the factor matrix $\bC$. In this special case we may compare to the improved bound presented in equation \eqref{eq:JGEBauerFikeSpecial}. Second, we consider the challenging case of generic perturbations on a rank $R$ tensor of size $R \times R \times R$. Finally we consider the situation most commonly occurring in applications of tensors with rank significantly smaller than dimension. That is, we examine spectral variation bounds for rank $R$ tensors of size $I \times I \times I$ with $I>\!\!>R$.

In each experiment we present  the mean observed spectral variation and either the mean computed spectral variation bound $\sqrt{R} \spnorm{\cE \cdot_1 \bA^{-1} \cdot_2 \bB^{-2}}$ or, when appropriate, the mean computed structured spectral variation bound $\spnorm{\cE \cdot_1 \bA^{-1} \cdot_2 \bB^{-2}}$.

\subsubsection{Structured Perturbations}

Figure \ref{figure:SVStructPertRcombo} considers the structured perturbation case on tensors of rank $R$ and size $R \times R \times R$ where $R = 4$ or $R = 10$. For various choices of SNR we randomly generate a rank $R$ tensor $\cT=\cpd$ and an error matrix $\bE$ with entries independently sampled from the standard normal distribution. The error matrix $\bE$ is then scaled so that the SNR of $\bC$ to $\bE$ is as specified. A total of 20 trials are performed for each SNR and for each SNR we report the mean observed spectral variation between $\cT$ and $\hcT:= [\![ \bA,\bB, \bC+\bE]\!]$ in addition to the mean computed value of $\|\cE \cdot_1 \bA^{-1}  \cdot_2 \bB^{-1}  \|_{\mathrm{sp}}$. Here $\cE := \hcT - \cT$. This experiment illustrates the first inequality in equation \eqref{eq:JGEBauerFikeSpecial}.

These figures illustrate that the error bound $\|\cE \cdot_1 \bA^{-1} \cdot_2 \bB^{-1}  \|_{\mathrm{sp}}$ is highly predictive of the observed spectral variation when considering structured perturbations. We remind the reader that the error bound $\|\cE \cdot_1 \bA^{-1}  \cdot_2 \bB^{-1}  \|_{\mathrm{sp}}$ is guaranteed to be an upper bound in this special setting. 

\begin{figure}[h!]
\hspace*{-0.5cm}
%
%
\definecolor{mycolor1}{rgb}{0.00000,0.44700,0.74100}%
\definecolor{mycolor2}{rgb}{0.85000,0.32500,0.09800}%
\definecolor{mycolor3}{rgb}{0.92900,0.69400,0.12500}%
\begin{tikzpicture}

\begin{axis}[%
width=5.833in,
height= 0in,
at={(0in,6.365in)},
scale only axis,
clip=false,
xmin=0,
xmax=1,
ymin=0,
ymax=1,
axis line style={draw=none},
ticks=none,
axis x line*=bottom,
axis y line*=left
]
\node[align=center]
at (axis cs:0.5,-0.05) {Structured perturbations: Tensor rank $R = 4$ or $10$; Dimensions = $R \times R \times R$};
\end{axis}

\begin{axis}[%
width=1.8 in,
height=1.1in,
at={(0.4 in, 4.8 in)},
scale only axis,
xmin=0,
xmax=21,
xtick={1,5,9,13,17,21},
xticklabels={{0},{20},{40},{60},{80},{100}},
xlabel style={font=\color{white!15!black}},
xlabel={SNR of $\mathbf{C}$ to $\mathbf{E}$ (dB)},
ymin=-5,
ymax=0,
axis background/.style={fill=white},
title style={font=\bfseries},
title={$R=4$: log$_{\mathbf{10}}$ of SV bounds},
axis x line*=bottom,
axis y line*=left,
legend style={legend cell align=left, align=left, draw=white!15!black,draw=none, {nodes={scale=0.5, transform shape}}, at={(0.05,0.05)},anchor= south west}
]
\addplot [color=mycolor1, line width=1.5pt, mark=triangle, mark options={solid, mycolor1}]
  table[row sep=crcr]{%
1	-0.116477010751614\\
2	-0.183477107126603\\
3	-0.373863231827866\\
4	-0.555157165385734\\
5	-0.718665041871438\\
6	-1.01046340540526\\
7	-1.32497527916855\\
8	-1.53328976359217\\
9	-1.76652084043969\\
10	-1.98851618410718\\
11	-2.26312938004833\\
12	-2.52801146219859\\
13	-2.81550005476841\\
14	-3.05511823532838\\
15	-3.30828703238783\\
16	-3.55246108850388\\
17	-3.83444480193955\\
18	-3.94522081860421\\
19	-4.28503607860187\\
20	-4.40821436585904\\
21	-4.82868047982469\\
};
\addlegendentry{$\mathrm{sv} [\mathcal{T},\tilde{\mathcal{T}}]$}

\addplot [color=mycolor2, line width=1.5pt, mark=square, mark options={solid, mycolor2}]
  table[row sep=crcr]{%
1	0\\
2	-0.0509443195827028\\
3	-0.250527302688664\\
4	-0.442319365575169\\
5	-0.642262346421354\\
6	-0.934975626486451\\
7	-1.27445422142936\\
8	-1.46462961155203\\
9	-1.69258853609179\\
10	-1.92573675170748\\
11	-2.2226982270912\\
12	-2.46657153585766\\
13	-2.74247165098441\\
14	-2.9906656420541\\
15	-3.25952016640671\\
16	-3.48459503881818\\
17	-3.77623257735657\\
18	-3.90899343042828\\
19	-4.2486673503611\\
20	-4.32814194354316\\
21	-4.77329377006102\\
};
\addlegendentry{$\|\mathcal{E} \cdot_1 \mathbf{A}^{-1} \cdot_2 \mathbf{B}^{-1}\|_{\mathrm{sp}}$}

\end{axis}

\begin{axis}[%
width=1.8 in,
height=1.1in,
at={(3.3 in, 4.8 in)},
scale only axis,
xmin=0,
xmax=21,
xtick={1,5,9,13,17,21},
xticklabels={{0},{20},{40},{60},{80},{100}},
xlabel style={font=\color{white!15!black}},
xlabel={SNR of $\mathbf{C}$ to $\mathbf{E}$ (dB)},
ymin=-5,
ymax=0,
axis background/.style={fill=white},
title style={font=\bfseries},
title={$R=10$: log$_{\mathbf{10}}$ of SV bounds},
axis x line*=bottom,
axis y line*=left,
legend style={legend cell align=left, align=left, draw=white!15!black,draw=none, {nodes={scale=0.5, transform shape}}, at={(0.05,0.05)},anchor= south west}
]
\addplot [color=mycolor1, line width=1.5pt, mark=triangle, mark options={solid, mycolor1}]
  table[row sep=crcr]{%
1	-0.0706187973329936\\
2	-0.149503286251377\\
3	-0.344573201282323\\
4	-0.554617990787828\\
5	-0.79086553850672\\
6	-1.04923761590957\\
7	-1.31228849260815\\
8	-1.5234027307848\\
9	-1.78988938453193\\
10	-1.99379711297247\\
11	-2.29324470276904\\
12	-2.51376200609233\\
13	-2.79202768112432\\
14	-3.04097914292411\\
15	-3.30807773847195\\
16	-3.5499444026367\\
17	-3.78357420891092\\
18	-4.03711494828497\\
19	-4.30044547683669\\
20	-4.55061912467796\\
21	-4.780231715732\\
};
\addlegendentry{$\mathrm{sv} [\mathcal{T},\tilde{\mathcal{T}}]$}

\addplot [color=mycolor2, line width=1.5pt, mark=square, mark options={solid, mycolor2}]
  table[row sep=crcr]{%
1	0\\
2	-0.0484520492738962\\
3	-0.2888970724568\\
4	-0.53415637744559\\
5	-0.761248264324663\\
6	-1.02200488152386\\
7	-1.29172581758059\\
8	-1.50743483982311\\
9	-1.77089520181175\\
10	-1.96872203871473\\
11	-2.27233961074361\\
12	-2.49958470265831\\
13	-2.779626437551\\
14	-3.02637584093322\\
15	-3.28840328572747\\
16	-3.53015793210016\\
17	-3.76273519273047\\
18	-4.01374797361305\\
19	-4.26002899894166\\
20	-4.53612514717032\\
21	-4.75882397181556\\
};
\addlegendentry{$\|\mathcal{E} \cdot_1 \mathbf{A}^{-1} \cdot_2 \mathbf{B}^{-1}\|_{\mathrm{sp}}$}

\end{axis}
\end{tikzpicture}%
	\caption{Observed spectral variation and computed bound between tensors of the form $\cpd$ and $[\![\bA,\bB,\bC+\bE]\!]$ with rank $R=4$ or $10$. The structured bound $\|\cE \cdot_1 \bA^{-1}  \cdot_2 \bB^{-1}  \|_{\mathrm{sp}}$ is observed to be highly predictive of the observed spectral variation.}
	\label{figure:SVStructPertRcombo}
\end{figure}

\subsubsection{Generic Perturbations: Rank equal to dimensions}

Figure \ref{figure:SVGenPertRcombo} illustrates the case of generic perturbations on tensors of rank $R$ and size $R \times R \times R$ where $R = 4$ or $R = 10$. In each trial we randomly generate a rank $R$ tensor $\cT$ and noise $\cN$ at the specified SNR, then compute a rank $R$ approximation $\hcT$ of $\cT+\cN$ using Tensorlab \cite{VDSBL16}. A total of 50 trials are performed for each SNR, and for each SNR we report the mean observed spectral variation between $\cT$ and $\hcT$ as well as the mean computed value for $\sqrt{R} \|\cE \cdot_1 \bA^{-1}  \cdot_2 \bB^{-1}  \|_{\mathrm{sp}}$ where $\cE := \hcT - \cT$. This experiment illustrates the first inequality in equation \eqref{eq:JGEBauerFike}. 

This case is particularly challenging since the perturbation is applied directly to the core of the tensor, which is not typically the case in applications. Furthermore, our bound serves as a worst case scenario bound which is unlikely to be achieved under generic perturbations.

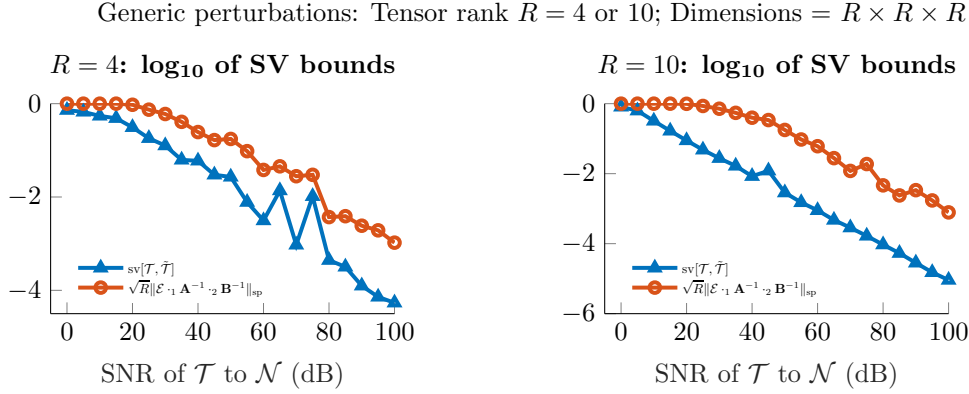
\begin{figure}[h!]
\hspace*{-0.5cm}
%
%
\definecolor{mycolor1}{rgb}{0.00000,0.44700,0.74100}%
\definecolor{mycolor2}{rgb}{0.85000,0.32500,0.09800}%
\begin{tikzpicture}

\begin{axis}[%
width=5.833in,
height= 0in,
at={(0in,6.365in)},
scale only axis,
clip=false,
xmin=0,
xmax=1,
ymin=0,
ymax=1,
axis line style={draw=none},
ticks=none,
axis x line*=bottom,
axis y line*=left
]
\node[align=center]
at (axis cs:0.5,-0.05) {Generic perturbations: Tensor rank $R = 4$ or $10$; Dimensions = $R \times R\times R$};
\end{axis}

\begin{axis}[%
width=1.8in,
height=1.1in,
at={(0.4 in, 4.8 in)},
scale only axis,
xmin=0,
xmax=21,
xtick={1,5,9,13,17,21},
xticklabels={{0},{20},{40},{60},{80},{100}},
xlabel style={font=\color{white!15!black}},
xlabel={SNR of $\mathcal{T}$ to $\mathcal{N}$ (dB)},
ymin=-4.5,
ymax=0,
axis background/.style={fill=white},
title style={font=\bfseries},
title={$R=4$: log$_{\mathbf{10}}$ of SV bounds},
axis x line*=bottom,
axis y line*=left,
legend style={legend cell align=left, align=left, draw=white!15!black,draw=none, {nodes={scale=0.5, transform shape}}, at={(0.05,0.05)},anchor= south west}
]
\addplot [color=mycolor1, line width=1.5pt, mark=triangle, mark options={solid, mycolor1}]
  table[row sep=crcr]{%
1	-0.13911154479153\\
2	-0.176635363074034\\
3	-0.262215874893949\\
4	-0.312848012692741\\
5	-0.507172196002718\\
6	-0.74201410436537\\
7	-0.898467638601102\\
8	-1.20490128170139\\
9	-1.2246568238469\\
10	-1.52462328804354\\
11	-1.5697394309662\\
12	-2.11039562038559\\
13	-2.50459451697712\\
14	-1.86197638443572\\
15	-3.02601212370162\\
16	-1.98571924889219\\
17	-3.35425888750719\\
18	-3.49729002335189\\
19	-3.90295349582193\\
20	-4.14114726154413\\
21	-4.26713099143944\\
};
\addlegendentry{$\mathrm{sv} [\mathcal{T},\tilde{\mathcal{T}}]$}

\addplot [color=mycolor2, line width=1.5pt, mark=o, mark options={solid, mycolor2}]
  table[row sep=crcr]{%
1	0\\
2	0\\
3	-0.000102962933209554\\
4	0\\
5	-0.0234702476998228\\
6	-0.128405241414541\\
7	-0.221458480529286\\
8	-0.389363087328419\\
9	-0.608493557581324\\
10	-0.780414769011799\\
11	-0.755855902835878\\
12	-1.01752212968345\\
13	-1.42010860767131\\
14	-1.34267419117452\\
15	-1.55456814097523\\
16	-1.52481037882938\\
17	-2.43048790397163\\
18	-2.41123535384863\\
19	-2.61532484043395\\
20	-2.71425652999255\\
21	-2.97832144448641\\
};
\addlegendentry{$\sqrt{R}\|\mathcal{E} \cdot_1 \mathbf{A}^{-1} \cdot_2 \mathbf{B}^{-1}\|_{\mathrm{sp}}$}

\end{axis}

\begin{axis}[%
width=1.8in,
height=1.1in,
at={(3.3 in, 4.8 in)},
scale only axis,
xmin=0,
xmax=21,
xtick={1,5,9,13,17,21},
xticklabels={{0},{20},{40},{60},{80},{100}},
xlabel style={font=\color{white!15!black}},
xlabel={SNR of $\mathcal{T}$ to $\mathcal{N}$ (dB)},
ymin=-6,
ymax=0,
axis background/.style={fill=white},
title style={font=\bfseries},
title={$R=10$: log$_{\mathbf{10}}$ of SV bounds},
axis x line*=bottom,
axis y line*=left,
legend style={legend cell align=left, align=left, draw=white!15!black,draw=none, {nodes={scale=0.5, transform shape}}, at={(0.05,0.05)},anchor= south west}
]
\addplot [color=mycolor1, line width=1.5pt, mark=triangle, mark options={solid, mycolor1}]
  table[row sep=crcr]{%
1	-0.0913219652080639\\
2	-0.202118585946198\\
3	-0.493106438918827\\
4	-0.777836966060567\\
5	-1.05205945902218\\
6	-1.31898309761447\\
7	-1.55965980320979\\
8	-1.77934386671399\\
9	-2.07390818075237\\
10	-1.91696769602434\\
11	-2.54261973419746\\
12	-2.82196032132937\\
13	-3.04911300720866\\
14	-3.32476364044404\\
15	-3.54287694735269\\
16	-3.77911711464477\\
17	-4.02797808428368\\
18	-4.27207173015177\\
19	-4.54810544825681\\
20	-4.82074076107129\\
21	-5.0397126577405\\
};
\addlegendentry{$\mathrm{sv} [\mathcal{T},\tilde{\mathcal{T}}]$}

\addplot [color=mycolor2, line width=1.5pt, mark=o, mark options={solid, mycolor2}]
  table[row sep=crcr]{%
1	0\\
2	0\\
3	0\\
4	0\\
5	-0.00415134392618465\\
6	-0.0657223152235864\\
7	-0.144545500081752\\
8	-0.262462632781808\\
9	-0.400859226680908\\
10	-0.468502029250907\\
11	-0.750942741892839\\
12	-1.0236992453962\\
13	-1.21564889917029\\
14	-1.55625794619393\\
15	-1.92138191782927\\
16	-1.72493863043979\\
17	-2.33414020837222\\
18	-2.62210045109134\\
19	-2.47119410740203\\
20	-2.76472156027544\\
21	-3.10627826764917\\
};
\addlegendentry{$\sqrt{R}\|\mathcal{E} \cdot_1 \mathbf{A}^{-1} \cdot_2 \mathbf{B}^{-1}\|_{\mathrm{sp}}$}

\end{axis}

\end{tikzpicture}%
	\caption{Observed spectral variation and computed bound for a rank $R$ tensor $\cT$ and a rank $R$ approximation to $\cT+\cN$ where $R=4$ or $10$. The bound $\sqrt{R}\|\cE \cdot_1 \bA^{-1}  \cdot_2 \bB^{-1}  \|_{\mathrm{sp}}$ typically exceeds the observed spectral variation by about one order of magnitude when $R=4$ and by about two orders of magnitude when $R=10$. For perspective we note that $\log_{10} (\sqrt{4}) \approx 0.3$ while $\log_{10} (\sqrt{10}) = 0.5.$}
	\label{figure:SVGenPertRcombo}
\end{figure}

\subsubsection{Generic Perturbations: Rank less than dimensions}

Figures \ref{figure:svscdat4combo} and \ref{figure:svscdat10combo} consider the case which most commonly occurs in applications where the rank of the tensor of interest is significantly less than dimensions of the tensor. Here we work with tensors $ \cT'$ of rank $R$ and dimensions $I \times I \times I$ where either $R = 4$ and $I = 10,20,$ or $100$ or $R = 10$ and $I = 20$ or $100$. 

Let $\hcT'$ denote a rank $R$ approximation of $\cT'+\cN'$. To directly apply the spectral variation bound obtained from the spectral norm in Theorem \ref{theorem:TensBauerFike} to a pair of rank $R$ tensors, the factor matrices of at least one tensor should be invertible. In particular the tensors should have dimensions equal to their rank. Thus it is most appropriate to apply this bound to core tensors $\cT$ and $\hcT$ of size $R \times R \times R$ satisfying $\|\cT - \hcT \|_\mathrm{F} \leq \|\cT' - \hcT'\|_\mathrm{F}$ as described in the tensor Procrustes problem of Theorem \ref{theorem:TensorProcrustes}. Write $\cT = [\![\bA,\bB,\bC]\!]$. For each SNR we report the mean observed spectral variation between the core tensors $\cT$ and $\hcT$ as well as the mean computed value for $\sqrt{R} \|\cE \cdot_1 \bA^{-1} \cdot_2 \bB^{-1} \|_{\mathrm{sp}}$ where $\cE := \hcT - \cT$. For each choice of $I$ and $R$ the mean is computed over $20$ trials.

In this setting we find that the error bound $\sqrt{R} \|\cE \cdot_1 \bA^{-1} \cdot_2 \bB^{-1} \|_{\mathrm{sp}}$ from the first inequality in equation \eqref{eq:JGEBauerFike} is highly predictive of the observed spectral variation, with the error bound typically within a half order of magnitude of the observed spectral variation. Furthermore, we find that the observed spectral variation and spectral variation bound become remarkably small as $I$ grows, even for low SNR. It is well known that computation of the core of an observed tensor can significantly increase the effective SNR. Our results further suggest that a significant portion of the error remaining after computing a rank $R$ approximation of an observed tensor $\cT'+\cN'$ is in the form of rotational error (i.e. error accounted for by unitary rotations of the mode-$j$ subspaces of the tensor) which disappears after a Tensor Procrustes computation. 

While this rotational error is still meaningful from the perspective of quality of the estimated factors, rotational error cannot lead to nonexistence of a best rank $R$ approximation and cannot negatively impact the condition number of computing a CPD (see \cite{V17} for details on this condition number), as these issues are invariant under unitary changes of basis. The observation that, in the setting of independent and identically distributed (i.i.d) noise, a significant portion error observed between the original tensor $\cT'$ and a rank $R$ approximation to the observed tensor $\cT'+\cN'$ is merely rotational error further explains the success of CPD in many large scale settings. In particular, this observation illustrates that CPD computation in practice is likely to be more stable than one might naively expect. See Section \ref{sec:TensorProcrustesExperiments} for experiments in this direction.

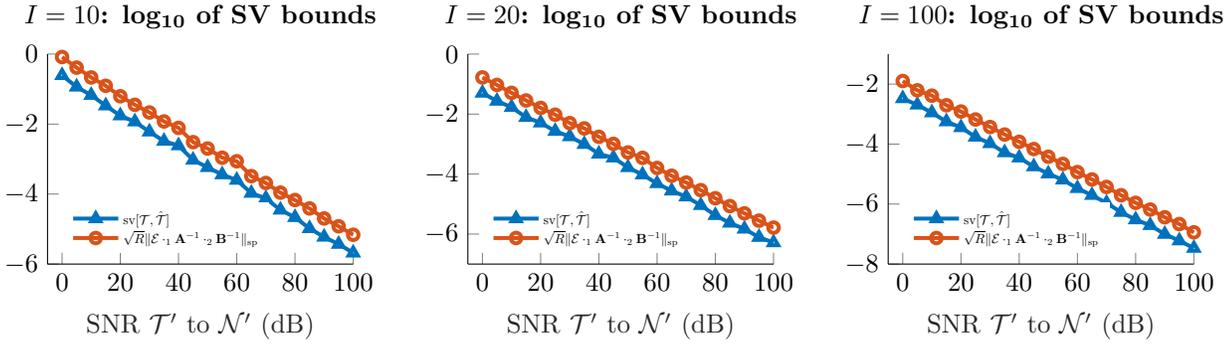
\begin{figure}[h!]
\hspace*{-1cm}
%
%
\definecolor{mycolor1}{rgb}{0.00000,0.44700,0.74100}%
\definecolor{mycolor2}{rgb}{0.85000,0.32500,0.09800}%
\begin{tikzpicture}
\pgfplotsset{scaled y ticks=false}

\begin{axis}[%
width=5.833in,
height= 0in,
at={(0in,8.465in)},
scale only axis,
clip=false,
xmin=0,
xmax=1,
ymin=0,
ymax=1,
axis line style={draw=none},
ticks=none,
axis x line*=bottom,
axis y line*=left
]
\node[align=center]
at (axis cs:0.5,-0.05) {Generic perturbations with Procrustes: Tensor rank $R = 4$; Dimensions = $I \times I \times I$};
\end{axis}

\begin{axis}[%
width=1.6in,
height=1.1in,
at={(0 in, 6.9 in)},
scale only axis,
xmin=0,
xmax=21,
xtick={1,5,9,13,17,21},
xticklabels={{0},{20},{40},{60},{80},{100}},
xlabel style={font=\color{white!15!black}},
xlabel={SNR $\mathcal{T}'$ to $\mathcal{N}'$ (dB)},
ymin=-6,
ymax=0,
axis background/.style={fill=white},
title style={font=\bfseries},
title={$I = 10$: log$_{\mathbf{10}}$ of SV bounds},
axis x line*=bottom,
axis y line*=left,
legend style={legend cell align=left, align=left, draw=white!15!black,draw=none, {nodes={scale=0.5, transform shape}}, at={(0.05,0.05)},anchor= south west}
]
\addplot [color=mycolor1, line width=1.5pt, mark=triangle, mark options={solid, mycolor1}]
  table[row sep=crcr]{%
1	-0.609866050541206\\
2	-0.938296521713281\\
3	-1.17491658554145\\
4	-1.47628621031307\\
5	-1.76053826344315\\
6	-1.93352375742741\\
7	-2.22090761799094\\
8	-2.48703353541862\\
9	-2.61551886336209\\
10	-3.02618605288201\\
11	-3.23759688405011\\
12	-3.44438559798338\\
13	-3.60104094585336\\
14	-3.97269991576793\\
15	-4.12466207506072\\
16	-4.45346277971139\\
17	-4.66623135907378\\
18	-4.98246293636933\\
19	-5.2265081026971\\
20	-5.43104206572295\\
21	-5.68044256267452\\
};
\addlegendentry{$\mathrm{sv} [\mathcal{T},\hat{\mathcal{T}}]$}

\addplot [color=mycolor2, line width=1.5pt, mark=o, mark options={solid, mycolor2}]
  table[row sep=crcr]{%
1	-0.0845588532685347\\
2	-0.38687713944725\\
3	-0.667847789612793\\
4	-0.90595493470423\\
5	-1.2051928235322\\
6	-1.44957103990002\\
7	-1.66875857860957\\
8	-1.92501171435374\\
9	-2.10938408634806\\
10	-2.51644174965368\\
11	-2.69872064718395\\
12	-2.96137103185676\\
13	-3.06202113826678\\
14	-3.48717966447529\\
15	-3.68325752741784\\
16	-3.95943789613344\\
17	-4.17095569745989\\
18	-4.41036630717989\\
19	-4.69679410387869\\
20	-4.92060755415732\\
21	-5.16329663691004\\
};
\addlegendentry{$\sqrt{R}\|\mathcal{E} \cdot_1 \mathbf{A}^{-1} \cdot_2 \mathbf{B}^{-1}\|_{\mathrm{sp}}$}

\end{axis}

\begin{axis}[%
width=1.6in,
height=1.1in,
at={(2.2 in, 6.9 in)},
scale only axis,
xmin=0,
xmax=21,
xtick={1,5,9,13,17,21},
xticklabels={{0},{20},{40},{60},{80},{100}},
xlabel style={font=\color{white!15!black}},
xlabel={SNR $\mathcal{T}'$ to $\mathcal{N}'$ (dB)},
ymin=-7,
ymax=0,
axis background/.style={fill=white},
title style={font=\bfseries},
title={$I = 20$: log$_{\mathbf{10}}$ of SV bounds},
axis x line*=bottom,
axis y line*=left,
legend style={legend cell align=left, align=left, draw=white!15!black,draw=none, {nodes={scale=0.5, transform shape}}, at={(0.05,0.05)},anchor= south west}
]
\addplot [color=mycolor1, line width=1.5pt, mark=triangle, mark options={solid, mycolor1}]
  table[row sep=crcr]{%
1	-1.29813485909508\\
2	-1.56563147940272\\
3	-1.77125477223279\\
4	-2.10654470587454\\
5	-2.29093949959981\\
6	-2.57050505266392\\
7	-2.74918187262437\\
8	-3.00572587749753\\
9	-3.32805618197192\\
10	-3.45742276454658\\
11	-3.77853694111161\\
12	-4.02119991964739\\
13	-4.31312375808952\\
14	-4.5582137658081\\
15	-4.75605992934915\\
16	-5.04236986158798\\
17	-5.37680010979613\\
18	-5.62996018700338\\
19	-5.82958024846941\\
20	-6.10570942316942\\
21	-6.28366133716299\\
};
\addlegendentry{$\mathrm{sv} [\mathcal{T},\hat{\mathcal{T}}]$}

\addplot [color=mycolor2, line width=1.5pt, mark=o, mark options={solid, mycolor2}]
  table[row sep=crcr]{%
1	-0.778565134995623\\
2	-1.02861360917976\\
3	-1.29218969481938\\
4	-1.54605211550231\\
5	-1.7941221093465\\
6	-2.02023053883162\\
7	-2.30119658962647\\
8	-2.47767190884531\\
9	-2.75819024111379\\
10	-2.99380453684864\\
11	-3.27909323924689\\
12	-3.45903366305149\\
13	-3.79503393034761\\
14	-4.05546243521577\\
15	-4.28231116966214\\
16	-4.53578178016131\\
17	-4.80531470457729\\
18	-5.06723221779637\\
19	-5.30861182500867\\
20	-5.55215375641449\\
21	-5.7844696571786\\
};
\addlegendentry{$\sqrt{R}\|\mathcal{E} \cdot_1 \mathbf{A}^{-1} \cdot_2 \mathbf{B}^{-1}\|_{\mathrm{sp}}$}

\end{axis}

\begin{axis}[%
width=1.6in,
height=1.1in,
at={(4.4 in, 6.9 in)},
scale only axis,
xmin=0,
xmax=21,
xtick={1,5,9,13,17,21},
xticklabels={{0},{20},{40},{60},{80},{100}},
xlabel style={font=\color{white!15!black}},
xlabel={SNR $\mathcal{T}'$ to $\mathcal{N}'$ (dB)},
ymin=-8,
ymax=-1,
axis background/.style={fill=white},
title style={font=\bfseries},
title={$I = 100$: log$_{\mathbf{10}}$ of SV bounds},
axis x line*=bottom,
axis y line*=left,
legend style={legend cell align=left, align=left, draw=white!15!black,draw=none, {nodes={scale=0.5, transform shape}}, at={(0.05,0.05)},anchor= south west}
]
\addplot [color=mycolor1, line width=1.5pt, mark=triangle, mark options={solid, mycolor1}]
  table[row sep=crcr]{%
1	-2.48080320461759\\
2	-2.69336306433136\\
3	-2.95261013645824\\
4	-3.25210364919466\\
5	-3.44683430818432\\
6	-3.76323975446438\\
7	-3.98004970818159\\
8	-4.28347558841497\\
9	-4.45741559487913\\
10	-4.75179837558785\\
11	-4.98820325135706\\
12	-5.19220373815934\\
13	-5.47577512912929\\
14	-5.70519596018707\\
15	-5.96436946699105\\
16	-6.27255469898054\\
17	-6.51936967585345\\
18	-6.71023320731331\\
19	-7.00400685097349\\
20	-7.2205759443153\\
21	-7.47012624532477\\
};
\addlegendentry{$\mathrm{sv} [\mathcal{T},\hat{\mathcal{T}}]$}

\addplot [color=mycolor2, line width=1.5pt, mark=o, mark options={solid, mycolor2}]
  table[row sep=crcr]{%
1	-1.8980907777466\\
2	-2.20777651961374\\
3	-2.3854257822816\\
4	-2.70517164224111\\
5	-2.91465580903266\\
6	-3.17985734099745\\
7	-3.43071702350417\\
8	-3.68163520977294\\
9	-3.92348227108086\\
10	-4.16539932197368\\
11	-4.42232206546516\\
12	-4.66167017757718\\
13	-4.93388719402756\\
14	-5.17933001610073\\
15	-5.42860378857838\\
16	-5.70574638273285\\
17	-5.95940492928619\\
18	-6.18239704972509\\
19	-6.4397323659254\\
20	-6.66644248517966\\
21	-6.94245217162979\\
};
\addlegendentry{$\sqrt{R}\|\mathcal{E} \cdot_1 \mathbf{A}^{-1} \cdot_2 \mathbf{B}^{-1}\|_{\mathrm{sp}}$}

\end{axis}

\end{tikzpicture}%
		\caption{Observed spectral variation and computed bound for cores $\cT$ and $\hcT$ where $\cT'$ is a rank $4$ tensor of size $I \times I \times I$ with $I = 10,20$ or $100$ and $\hcT'$ is a rank $4$ approximation to $\cT'+\cN'$. The bound $\sqrt{R} \|\cE \cdot_1 \bA^{-1} \cdot_2 \bB^{-1} \|_{\mathrm{sp}}$ provides a strong estimate for the observed spectral variation. The spectral variation bound and observed spectral variation both decrease as $I$ grows.}
	\label{figure:svscdat4combo}
\end{figure} 

\begin{figure}[h!]
\hspace*{-0.5cm}
%
%
\definecolor{mycolor1}{rgb}{0.00000,0.44700,0.74100}%
\definecolor{mycolor2}{rgb}{0.85000,0.32500,0.09800}%
\begin{tikzpicture}
\pgfplotsset{scaled y ticks=false}

\begin{axis}[%
width=5.833in,
height= 0in,
at={(0in,6.365in)},
scale only axis,
clip=false,
xmin=0,
xmax=1,
ymin=0,
ymax=1,
axis line style={draw=none},
ticks=none,
axis x line*=bottom,
axis y line*=left
]
\node[align=center]
at (axis cs:0.5,-0.05) {Generic perturbations with Procrustes: Tensor rank $R = 10$; Dimensions = $I \times I \times I$};
\end{axis}

\begin{axis}[%
width=1.8in,
height=1.1in,
at={(0.4 in, 4.8 in)},
scale only axis,
xmin=0,
xmax=21,
xtick={1,5,9,13,17,21},
xticklabels={{0},{20},{40},{60},{80},{100}},
xlabel style={font=\color{white!15!black}},
xlabel={SNR $\mathcal{T}'$ to $\mathcal{N}'$ (dB)},
ymin=-6,
ymax=0,
axis background/.style={fill=white},
title style={font=\bfseries},
title={$I=20$: log$_{\mathbf{10}}$ of SV bounds},
axis x line*=bottom,
axis y line*=left,
legend style={fill opacity=0.0, text opacity = 1, legend cell align=left, align=left, draw=white!15!black,draw=none, {nodes={scale=0.5, transform shape}}, at={(0.05,0.05)},anchor= south west}
]
\addplot [color=mycolor1, line width=1.5pt, mark=triangle, mark options={solid, mycolor1}]
  table[row sep=crcr]{%
1	-0.710726518528614\\
2	-0.986661702350301\\
3	-1.29530409574156\\
4	-1.47033490990021\\
5	-1.75331574677563\\
6	-2.00931289348045\\
7	-2.24419930501922\\
8	-2.51510194553013\\
9	-2.76099294063877\\
10	-3.00493391452947\\
11	-3.26973520699618\\
12	-3.49516380259087\\
13	-3.72186835404251\\
14	-4.00805808248557\\
15	-4.25746918595745\\
16	-4.50558476730629\\
17	-4.74648937719969\\
18	-5.03115708433267\\
19	-5.30627802776872\\
20	-5.48709127751261\\
21	-5.7668092974944\\
};
\addlegendentry{$\mathrm{sv} [\mathcal{T},\hat{\mathcal{T}}]$}

\addplot [color=mycolor2, line width=1.5pt, mark=o, mark options={solid, mycolor2}]
  table[row sep=crcr]{%
1	-0.00859799324247358\\
2	-0.274204606828168\\
3	-0.558703878116558\\
4	-0.780380907384668\\
5	-1.07153924718876\\
6	-1.34588038916869\\
7	-1.56816507310959\\
8	-1.79857297534315\\
9	-2.09999283483034\\
10	-2.3443099488039\\
11	-2.58880122120803\\
12	-2.80446493032903\\
13	-3.07680599162938\\
14	-3.3109739775119\\
15	-3.57723727637126\\
16	-3.80980996057746\\
17	-4.08181707796653\\
18	-4.3651496775389\\
19	-4.57384359752204\\
20	-4.84023690962247\\
21	-5.05793528011387\\
};
\addlegendentry{$\sqrt{R}\|\mathcal{E} \cdot_1 \mathbf{A}^{-1} \cdot_2 \mathbf{B}^{-1}\|_{\mathrm{sp}}$}

\end{axis}

\begin{axis}[%
width=1.8in,
height=1.1in,
at={(3.3 in, 4.8 in)},
scale only axis,
xmin=0,
xmax=21,
xtick={1,5,9,13,17,21},
xticklabels={{0},{20},{40},{60},{80},{100}},
xlabel style={font=\color{white!15!black}},
xlabel={SNR $\mathcal{T}'$ to $\mathcal{N}'$ (dB)},
ymin=-7,
ymax=-1,
axis background/.style={fill=white},
title style={font=\bfseries},
title={$I = 100$: log$_{\mathbf{10}}$ of SV bounds},
axis x line*=bottom,
axis y line*=left,
legend style={fill opacity=0.0, text opacity = 1, legend cell align=left, align=left, draw=white!15!black,draw=none, {nodes={scale=0.5, transform shape}}, at={(0.05,0.05)},anchor= south west}
]
\addplot [color=mycolor1, line width=1.5pt, mark=triangle, mark options={solid, mycolor1}]
  table[row sep=crcr]{%
1	-2.0090718328634\\
2	-2.24129480528697\\
3	-2.49503561341188\\
4	-2.74000272599874\\
5	-2.9770251637847\\
6	-3.24321868182697\\
7	-3.48934608705483\\
8	-3.7456478467392\\
9	-4.00793764348633\\
10	-4.24425014663218\\
11	-4.47791839637077\\
12	-4.74148126190687\\
13	-4.98748155586311\\
14	-5.27137599830859\\
15	-5.49437987776206\\
16	-5.70685272830257\\
17	-5.96171876166618\\
18	-6.25625645429616\\
19	-6.51898467995274\\
20	-6.7415377014053\\
21	-6.99327242824882\\
};
\addlegendentry{$\mathrm{sv} [\mathcal{T},\hat{\mathcal{T}}]$}

\addplot [color=mycolor2, line width=1.5pt, mark=o, mark options={solid, mycolor2}]
  table[row sep=crcr]{%
1	-1.33242918457415\\
2	-1.6060583122114\\
3	-1.83330086707208\\
4	-2.07639295107494\\
5	-2.33469510519336\\
6	-2.5844822478819\\
7	-2.8400477438731\\
8	-3.10726139056738\\
9	-3.34337069409201\\
10	-3.60605419441287\\
11	-3.84949625803371\\
12	-4.11449525584032\\
13	-4.34202882102641\\
14	-4.59253901480263\\
15	-4.87163498469086\\
16	-5.08479739557542\\
17	-5.33602655573732\\
18	-5.59095705607594\\
19	-5.84677001423183\\
20	-6.10703254496413\\
21	-6.35267553478532\\
};
\addlegendentry{$\sqrt{R}\|\mathcal{E} \cdot_1 \mathbf{A}^{-1} \cdot_2 \mathbf{B}^{-1}\|_{\mathrm{sp}}$}

\end{axis}

\end{tikzpicture}%
		\caption{Observed spectral variation and computed bound for cores $\cT$ and $\hcT$ where $\cT'$ is a rank $10$ tensor of size $I \times I \times I$ with $I = 20$ or $100$ and $\hcT'$ is a rank $10$ approximation to $\cT'+\cN'$. The bound $\sqrt{R} \|\cE \cdot_1 \bA^{-1} \cdot_2 \bB^{-1} \|_{\mathrm{sp}}$ provides a strong estimate for the observed spectral variation. The spectral variation bound and observed spectral variation both decrease as $I$ grows.}
	\label{figure:svscdat10combo}
\end{figure}
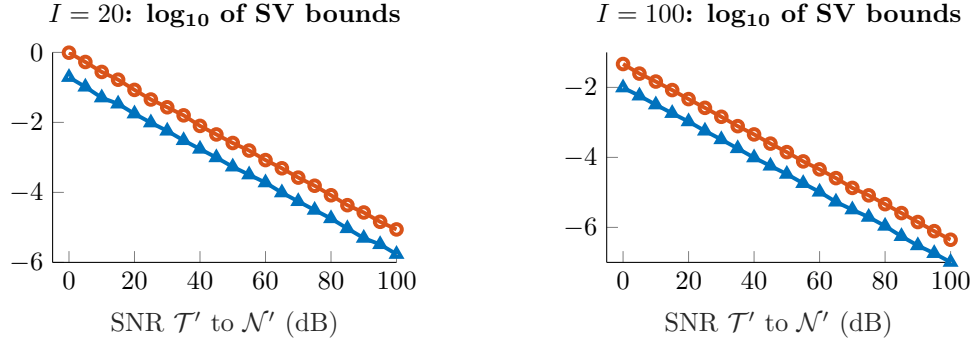 

\subsection{Bounds for existence of best rank $R$ approximations}

We now present the mean of the bound $\epsilon$ obtained using Theorem \ref{theorem:MultiplePencilBound} for the radius of a neighborhood around a randomly generated rank $R$ tensor of size $R \times R \times R$ of unit Frobenius norm in which best rank $R$ approximations are guaranteed to exist. For each choice of $R$, the bound $\epsilon$ is averaged over $20$ trials to compute the mean bound $\epsavg$. The mean bound $\epsavg$ is plotted in dB as $-20 \log_{10} (\epsavg)$. Thus, on average for the tensors in our experiment, a best rank $R$ approximation is guaranteed to exist if one has SNR greater than the reported value. The choices of $R$ considered are $R=2$ through $10$. 

As previously discussed, the bound obtained from Theorem \ref{theorem:MultiplePencilBound} depends on the choice of unitary $\bU$ chosen when setting $\cS = \cT \cdot_3 \bU$. While finding an optimal unitary is computationally challenging, as evidenced by its relationship to the longest path problem, one may simply try many randomly generated unitaries. For each tensor $\cT$ we try a total of $1,000$ random unitary matrices  and record the largest bound found after $1,10,100,$ and $1000$ unitaries tried. As a further step, one may find the optimal ordering of tensor slices for each unitary. Bounds computed without reordering slices are presented in the left image of Figure \ref{figure:AvgChordalBoundPlot} while bounds computed with finding the optimal slice ordering are presented in the right image of the same figure. The collection of tensors and unitaries tried in each case are the same.

\begin{figure}[h!]
\hspace*{-0.7cm}
%
%
\definecolor{mycolor1}{rgb}{0.00000,0.44700,0.74100}%
\definecolor{mycolor2}{rgb}{0.85000,0.32500,0.09800}%
\definecolor{mycolor3}{rgb}{0.92900,0.69400,0.12500}%
\definecolor{mycolor4}{rgb}{0.49400,0.18400,0.55600}%
\begin{tikzpicture}

\begin{axis}[%
width= 1.8 in,
height=1.4 in,
at={(0.4 in, 3.0 in)},
scale only axis,
xmin=1,
xmax=9,
xtick={1,2,3,4,5,6,7,8,9},
xticklabels={{2},{3},{4},{5},{6},{7},{8},{9},{10}},
xlabel style={font=\color{white!15!black}},
xlabel={Rank of tensor},
ymin=30,
ymax=100,
ytick={25,45,65,85},
yticklabels={{25},{45},{65},{85}},
ylabel style={font=\color{white!15!black}},
ylabel={Mean bound plotted in dB},
axis background/.style={fill=white},
title style={font=\bfseries},
title={Without reordering slices},
axis x line*=bottom,
axis y line*=left,
legend style={at={(0.97,0.03)}, anchor=south east, legend cell align=left, align=left, fill=none, draw=none,{nodes={scale=0.55, transform shape}}}
]
\addplot [color=mycolor1, line width=1.5pt, mark=triangle, mark options={solid, mycolor1}]
  table[row sep=crcr]{%
1	34.492556293437\\
2	48.0477296704179\\
3	61.5406387631577\\
4	73.8923813813984\\
5	79.6168511812558\\
6	89.6786139226781\\
7	92.5115659495182\\
8	93.1915249511392\\
9	98.3210703483774\\
};
\addlegendentry{Unitaries tried $= 1$}

\addplot [color=mycolor2, dashed, line width=1.5pt, mark=o, mark options={solid, mycolor2}]
  table[row sep=crcr]{%
1	34.492556293437\\
2	45.7365348815784\\
3	54.3886522731169\\
4	69.6047262375282\\
5	75.2342290568377\\
6	82.139346577929\\
7	85.3506422461756\\
8	85.1430839050099\\
9	93.2896729329201\\
};
\addlegendentry{Unitaries tried $= 10$}

\addplot [color=mycolor3, dashdotted, line width=1.5pt, mark=square, mark options={solid, mycolor3}]
  table[row sep=crcr]{%
1	34.492556293437\\
2	44.8233901413724\\
3	52.7021286457762\\
4	65.8450963333297\\
5	72.6923178920563\\
6	80.4915733282048\\
7	83.7630365172973\\
8	83.3547352395678\\
9	90.4375471701516\\
};
\addlegendentry{Unitaries tried $= 100$}

\addplot [color=mycolor4, dotted, line width=1.5pt, mark=diamond, mark options={solid, mycolor4}]
  table[row sep=crcr]{%
1	34.492556293437\\
2	44.6309372536956\\
3	51.814459388616\\
4	64.4743440374867\\
5	70.514242646608\\
6	78.7414144295891\\
7	82.0072824901731\\
8	81.6174227850374\\
9	88.1712399046043\\
};
\addlegendentry{Unitaries tried $= 1000$}

\end{axis}

\begin{axis}[%
width=5.833in,
height= 0in,
at={(0in,4.875in)},
scale only axis,
clip=false,
xmin=0,
xmax=1,
ymin=0,
ymax=1,
axis line style={draw=none},
ticks=none,
axis x line*=bottom,
axis y line*=left
]
\node[align=center]
at (axis cs:0.5,-0.05) {Bound for existence of best rank $R$ approximation of $R \times R \times R$ tensor};
\end{axis}

\begin{axis}[%
width=1.8 in,
height=1.4 in,
at={(3.5 in, 3.0 in)},
scale only axis,
xmin=1,
xmax=9,
xtick={1,2,3,4,5,6,7,8,9},
xticklabels={{2},{3},{4},{5},{6},{7},{8},{9},{10}},
xlabel style={font=\color{white!15!black}},
xlabel={Rank of tensor},
ymin=30,
ymax=100,
ytick={25,45,65,85},
yticklabels={{25},{45},{65},{85}},
ylabel style={font=\color{white!15!black}},
ylabel={Mean bound plotted in dB},
title style={font=\bfseries},
title={With optimal slice ordering},
axis background/.style={fill=white},
axis x line*=bottom,
axis y line*=left,
legend style={at={(0.97,0.03)}, anchor=south east, legend cell align=left, align=left, fill=none, draw=none,{nodes={scale=0.55, transform shape}}}
]
\addplot [color=mycolor1, line width=1.5pt, mark=triangle, mark options={solid, mycolor1}]
  table[row sep=crcr]{%
1	34.492556293437\\
2	46.7849995523894\\
3	57.3339932452104\\
4	69.5890757231639\\
5	76.0294928592019\\
6	84.0741940022655\\
7	85.7612611877936\\
8	85.4497805470257\\
9	90.6250600683251\\
};
\addlegendentry{Unitaries tried $= 1$}

\addplot [color=mycolor2, dashed, line width=1.5pt, mark=o, mark options={solid, mycolor2}]
  table[row sep=crcr]{%
1	34.492556293437\\
2	45.0324173949087\\
3	53.4661499247206\\
4	66.9010444108548\\
5	72.7393906262198\\
6	80.4160778124498\\
7	83.1343727112435\\
8	82.3377636438497\\
9	88.4179003568528\\
};
\addlegendentry{Unitaries tried $= 10$}

\addplot [color=mycolor3, dashdotted, line width=1.5pt, mark=square, mark options={solid, mycolor3}]
  table[row sep=crcr]{%
1	34.492556293437\\
2	44.6936177133348\\
3	52.4780988651596\\
4	65.2729848817516\\
5	70.6262973784948\\
6	78.9131952618225\\
7	81.3522056294803\\
8	80.3496043661345\\
9	87.2897356948306\\
};
\addlegendentry{Unitaries tried $= 100$}

\addplot [color=mycolor4, dotted, line width=1.5pt, mark=diamond, mark options={solid, mycolor4}]
  table[row sep=crcr]{%
1	34.492556293437\\
2	44.5682102592406\\
3	51.7085063270005\\
4	63.7859509105967\\
5	69.6366781927647\\
6	77.1878942172245\\
7	80.2940441788656\\
8	79.6542566386031\\
9	86.3094273262912\\
};
\addlegendentry{Unitaries tried $= 1000$}

\end{axis}

\end{tikzpicture}%
	\caption{Mean bound for radius of a ball centered at a slice mix invertible rank $R$ tensor of unit Frobenius norm in which best rank $R$ approximations are guaranteed to exist. The minimum SNR needed for existence of a best low rank approximation to be guaranteed increases as $R$ increases from $2$ to $10$. Increasing the number of unitaries tried can notably improve the bound; however, there are diminishing returns as the number of unitaries tried increases. }
	\label{figure:AvgChordalBoundPlot}
\end{figure}
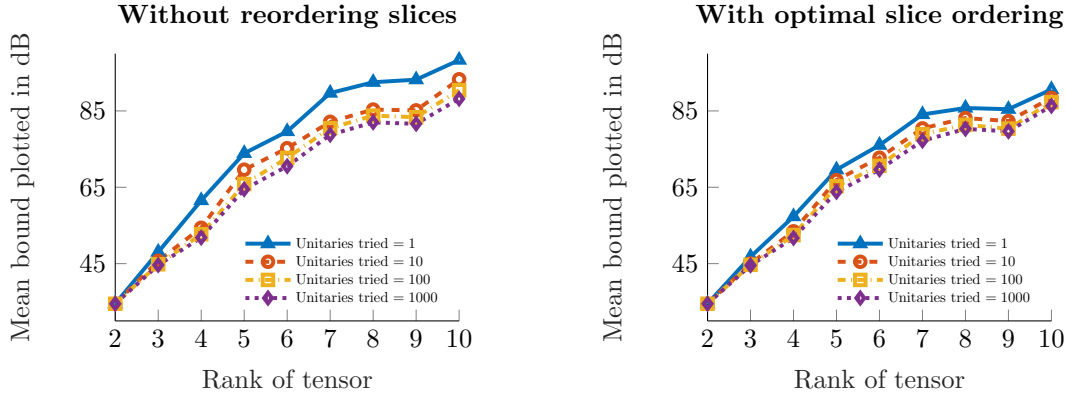 

\subsubsection{Bounds for existence in the case of rank less than dimensions}
\label{sec:existenceImagesBigDim}

We now examine the setting more common in applications where a tensor has rank significantly less than its dimensions. We first note that since the bound obtained from Theorem \ref{theorem:MultiplePencilBound} is a deterministic bound, the bound must take into account a worst case scenario and does not improve as tensor dimensions increase. A worst case scenario achievable for the core of a tensor is always achievable on the full tensor by applying a perturbation directly to the core of the tensor.

Although the deterministic bound given by Theorem \ref{theorem:MultiplePencilBound} does not improve as tensor dimensions grow, {\it the probability of a worst case scenario occurring dramatically decreases.} This is illustrated in Figure \ref{figure:ExistPro} with the following experiment. 

For various choices of $R$ and $I$ and SNR, we randomly generate a rank $R$ tensor $\cT'$ and a perturbation $\cN'$ of size $I \times I \times I$ such that the SNR of $\cT'$ to $\cN'$ is as specified. Then, as is commonly done in practice, we compute an MLSVD of $\cT'+\cN'$ and truncate the MLSVD to obtain a multilinear rank $R \times R \times R$ approximation of $\cT'+\cN'$. We let $\cW$ denote the truncated MLSVD of $\cT'+\cN'$. 

Theorem \ref{theorem:MLSVDExistenceBound} is used to compute a bound\footnote{For each $\cW$, a total of $1000$ randomly generated unitaries are tried when computing this radius. To limit computational complexity, we do not compute the optimal slice orderings. It is expected that either increasing the number of unitaries tried or computing the optimal slice ordering for each unitary would cause the curves in Figure \ref{figure:ExistPro} to shift slightly to the left.} $\epsilon$ such that if there exists a rank $R$ tensor within $\epsilon$ Frobenius distance of $\cW$, then $\cW$ is guaranteed to have a best rank $R$ approximation. We then compute a rank $R$ approximation $\hcT$ of $\cW$ and check if $\|\cW-\hcT\|_\mathrm{F} < \epsilon$, in which case we are able to guarantee that $\cW$ has a best rank $R$ approximation. For each choice of $I$, $R$, and SNR, a total of $10$ trials are performed and the proportion of truncated MLSVDs guaranteed to have a best rank $R$ approximation is reported.

As expected, we observe that the SNR needed for $\cW$ to be guaranteed to have a best rank $R$ approximation greatly decreases as $I$ increases. Additionally, for each choice of $I$ and $R$, we observe a sharp transition from being unable to guarantee the existence of a best rank $R$ approximation of $\cW$ to expecting to be able to guarantee the existence of a best rank $R$ approximation of $\cW$.

\begin{figure}[h!]
\hspace*{-1cm}
%
%
\definecolor{mycolor1}{rgb}{0.00000,0.44700,0.74100}%
\definecolor{mycolor2}{rgb}{0.85000,0.32500,0.09800}%
\definecolor{mycolor3}{rgb}{0.92900,0.69400,0.12500}%
\definecolor{mycolor4}{rgb}{0.49400,0.18400,0.55600}%
\begin{tikzpicture}

\begin{axis}[%
width=5.833in,
height= 0in,
at={(0in,4.875in)},
scale only axis,
clip=false,
xmin=0,
xmax=1,
ymin=0,
ymax=1,
axis line style={draw=none},
ticks=none,
axis x line*=bottom,
axis y line*=left
]
\node[align=center]
at (axis cs:0.5,-0.05) { Proportion of $I \times I \times I$ tensors $\cT+\cN$ with truncated MLSVD \\
 guaranteed to have a best rank $R$ approximation};
\end{axis}

\begin{axis}[%
width= 2.1 in,
height=1.3 in,
at={(0.0 in, 3.0 in)},
scale only axis,
xmin=-40,
xmax=30,
xlabel style={font=\color{white!15!black}},
xlabel={SNR of $\cT$ to $\cN$},
ymin=0,
ymax=1,
ylabel style={font=\color{white!15!black}},
ylabel={Proportion},
axis background/.style={fill=white},
title style={font=\bfseries},
title={Tensor rank $R=4$},
axis x line*=bottom,
axis y line*=left,
legend style={at={(1.05,0.03)}, anchor=south east, legend cell align=left, align=left, fill=none, draw=none,{nodes={scale=0.55, transform shape}}}
]
\addplot [color=mycolor1, line width=1.5pt, mark=triangle, mark options={solid, mycolor1}]
  table[row sep=crcr]{%
6	0\\
8	0\\
10	0.1\\
12	0.1\\
14	0.9\\
16	0.9\\
18	0.9\\
20	1\\
22	1\\
};
\addlegendentry{I = 10}

\addplot [color=mycolor2, dashed, line width=1.5pt, mark=o, mark options={solid, mycolor2}]
  table[row sep=crcr]{%
-8	0\\
-6	0\\
-4	0.2\\
-2	0.3\\
0	0.9\\
2	0.9\\
4	1\\
6	1\\
};
\addlegendentry{I = 20}

\addplot [color=mycolor3, dashdotted, line width=1.5pt, mark=square, mark options={solid, mycolor3}]
  table[row sep=crcr]{%
-26	0\\
-24	0\\
-22	0.3\\
-20	0.8\\
-18	1\\
-16	1\\
};
\addlegendentry{I = 100}

\addplot [color=mycolor4, dotted, line width=1.5pt, mark=diamond, mark options={solid, mycolor4}]
  table[row sep=crcr]{%
-32	0\\
-30	0\\
-28	0\\
-26	1\\
-24	1\\
-22	1\\
};
\addlegendentry{I = 200}

\end{axis}

\begin{axis}[%
width=2.1 in,
height=1.3 in,
at={(3.1 in, 3.0 in)},
scale only axis,
xmin=-20,
xmax=60,
xlabel style={font=\color{white!15!black}},
xlabel={SNR of $\cT$ to $\cN$},
ymin=0,
ymax=1,
ylabel style={font=\color{white!15!black}},
ylabel={Proportion},
axis background/.style={fill=white},
title style={font=\bfseries},
title={Tensor rank $R=10$},
axis x line*=bottom,
axis y line*=left,
legend style={at={(0.75,0.06)}, anchor=south east, legend cell align=left, align=left, fill=none, draw=none,{nodes={scale=0.55, transform shape}}}
]
\addplot [color=mycolor1, line width=1.5pt, mark=triangle, mark options={solid, mycolor1}]
  table[row sep=crcr]{%
38	0\\
40	0\\
42	0.1\\
44	0.2\\
46	0.7\\
48	1\\
50	1\\
};
\addlegendentry{I = 20}

\addplot [color=mycolor2, dashed, line width=1.5pt, mark=o, mark options={solid, mycolor2}]
  table[row sep=crcr]{%
10	0\\
12	0\\
14	0.8\\
16	1\\
18	1\\
};
\addlegendentry{I = 100}

\addplot [color=mycolor3, dashdotted, line width=1.5pt, mark=square, mark options={solid, mycolor3}]
  table[row sep=crcr]{%
0	0\\
2	0.2\\
4	0.8\\
6	1\\
8	1\\
};
\addlegendentry{I = 200}

\addplot [color=mycolor4, dotted, line width=1.5pt, mark=diamond, mark options={solid, mycolor4}]
  table[row sep=crcr]{%
-10	0\\
-8	0.2\\
-6	0.7\\
-4	1\\
-2	1\\
};
\addlegendentry{I = 400}

\end{axis}

\end{tikzpicture}%
	\caption{Proportion of $I \times I \times I$ tensors $\cT'+\cN'$ for which the truncated MLSVD $\cW$ of size $R \times R \times R$ is guaranteed to have a best rank $R$ approximation. The necessary SNR decreases greatly as tensor dimensions grow. For each choice of $I$ and $R$, a sharp transition point is observed. When the SNR exceeds the transition point, the truncated MLSVD is expected to have a best rank $R$ approximation.}
	\label{figure:ExistPro}
\end{figure}
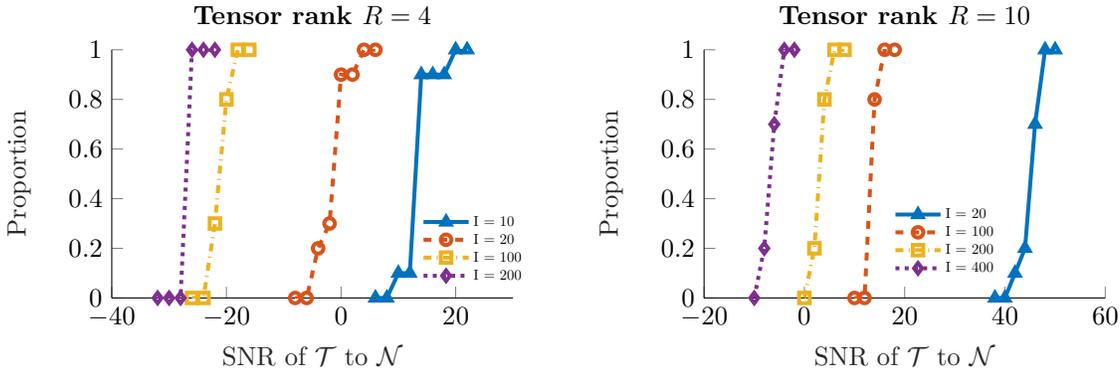

\subsection{0 dB SNR i.i.d noise is not 0 dB SNR noise on the core}
\label{sec:TensorProcrustesExperiments}

In this section we perform an experiment which is intended to help answer the following question: ``Given a tensor $\cT'$ and a perturbation $\cN'$, what is the expected distance between a pair of Procrustes cores computed for $\cT'$ and a best border rank $R$ approximation to $\cT'+\cN'$?" As we now argue, this is a critical measure of distance when considering the existence of a best rank $R$ approximation of $\cT'+\cN'$.

Recall the main idea in this article for determining if a tensor $\cT'+\cN'$ has a best rank $R$ approximation: The tensor $\cT'+\cN'$ is guaranteed to have a best border rank $R$ approximation $\hcT'$, so we need only verify that $\hcT'$ has rank $R$. To determine if $\hcT'$ has rank $R$, we check if $\hcT'$ lies in a neighborhood around $\cT'$ whose radius is determined by Theorem \ref{theorem:MultiplePencilBound}.

Moreover, since tensor (border) rank is invariant under modal unitary multiplication, it is sufficient to check that if one computes Procrustes cores $\cT$ and $\hcT$ for the pair $\cT',\hcT'$, then $\hcT$ lies in the neighborhood around $\cT$ given by Theorem \ref{theorem:MultiplePencilBound}. The key point is that \textit{the existence of a best rank $R$ approximation of $\cT'+\cN'$ can be determined by checking the distance between $\cT$ and $\hcT$.}

The experiment shown in Figure \ref{figure:TPEcombo} is performed as follows.  For various choices of $I$ and $R$ we randomly generate rank $R$ tensors $\cT'$ of size $I \times I \times I$ with unit Frobenius norm. For each tensor we generate a noise tensor $\cN'$ and scale $\cN'$ so that the SNR of $\cT'$ to $\cN'$ is as specified. After generating $\cT'$ and $\cN'$, we compute a rank $R$ approximation\footnote{For the purpose of this experiment, it is not necessary to compute a best rank $R$ approximation.} $\hcT'$ of $\cT'+\cN'$ and  the distance $\|\cT'-\hcT'\|_\mathrm{F}$ between the tensors $ \hcT'$ and $\cT'$. We then compute Procrustes cores $\cT$ and $\hcT$ of $\cT'$ and $\hcT'$, respectively, for which it is guaranteed that $\|\cT -\hcT\|_\mathrm{F} \leq \|\cT' - \hcT' \|_\mathrm{F}$ and record the distance $\|\cT -\hcT\|_\mathrm{F}$ between the cores. For each choice of $I$ and $R$, a total of $10$ trials are performed and the mean distance between the relevant pairs at each step is reported.

Figure \ref{figure:TPEcombo} illustrates that the distance between the core tensors $\cT$ and $\hcT$ is significantly smaller than the distance between $\cT'$ and the perturbed tensor $\cT'+\cN'$. We observe that the distance between $\cT'$ and $\hcT'$ is significantly less than that between $\cT'$ and $\cT'+\cN'$, and that the distance between $\cT$ and $\hcT$ is even smaller still. Furthermore, consistent with Figure \ref{figure:ExistPro}, this experiment shows that when $R <\!\!< I$, issues of nonexistence of a best rank $R$ approximation are unlikely to occur (given our experimental setup of i.i.d factors and i.i.d noise) unless the SNR between $\cT'$ and $\cN'$ is significantly lower than suggested by Figure \ref{figure:AvgChordalBoundPlot}.

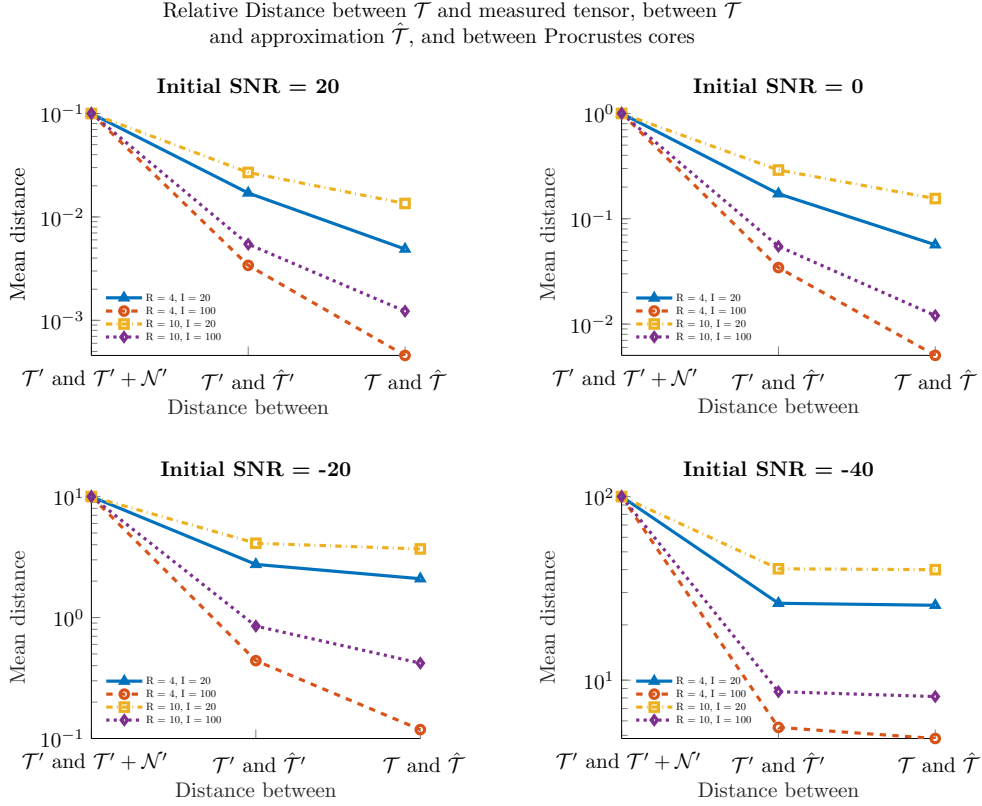
\begin{figure}[h!]
	\scalebox{.8}{
%
%
\definecolor{mycolor1}{rgb}{0.00000,0.44700,0.74100}%
\definecolor{mycolor2}{rgb}{0.85000,0.32500,0.09800}%
\definecolor{mycolor3}{rgb}{0.92900,0.69400,0.12500}%
\definecolor{mycolor4}{rgb}{0.49400,0.18400,0.55600}%
\begin{tikzpicture}

\begin{axis}[%
width=2.052in,
height=1.584in,
at={(4.027in,3.209in)},
scale only axis,
xmin=1,
xmax=3,
xtick={1,2,3,4,5},
xticklabels={{{ {$\cT'$ and $\cT'+\cN'$}}},{{{$\cT'$ and $\hcT'$}}},{{{$\cT$ and $\hcT$} }}},
xticklabel style={yshift=-3 pt},
xlabel style={font=\color{white!15!black},yshift=-2 pt},
xlabel={{ Distance between}},
ymode=log,
ymin=0.00504748639517362,
ymax=1,
yminorticks=true,
ylabel style={font=\color{white!15!black}},
ylabel={{{Mean distance}}},
axis background/.style={fill=white},
title style={font=\bfseries},
title={Initial SNR = 0},
axis x line*=bottom,
axis y line*=left,
legend style={at={(0.02,0.03)}, anchor=south west, legend cell align=left, align=left, fill=none, draw=none,{nodes={scale=0.5, transform shape}} }
]
\addplot [color=mycolor1, line width=1.5pt, mark=triangle, mark options={solid, mycolor1}]
  table[row sep=crcr]{%
1	1\\
2	0.172775369396928\\
3	0.0566256326463498\\
};
\addlegendentry{R = 4, I = 20}

\addplot [color=mycolor2, dashed, line width=1.5pt, mark=o, mark options={solid, mycolor2}]
  table[row sep=crcr]{%
1	1\\
2	0.0344065188640292\\
3	0.00504748639517362\\
};
\addlegendentry{R = 4, I = 100}

\addplot [color=mycolor3, dashdotted, line width=1.5pt, mark=square, mark options={solid, mycolor3}]
  table[row sep=crcr]{%
1	1\\
2	0.290294370271169\\
3	0.155609014214643\\
};
\addlegendentry{R = 10, I = 20}

\addplot [color=mycolor4, dotted, line width=1.5pt, mark=diamond, mark options={solid, mycolor4}]
  table[row sep=crcr]{%
1	1\\
2	0.0547672377561384\\
3	0.0120553348056223\\
};
\addlegendentry{R = 10, I = 100}

\end{axis}

\begin{axis}[%
width=2.052in,
height=1.584in,
at={(0.558in,3.209in)},
scale only axis,
xmin=1,
xmax=3,
xtick={1,2,3,4,5},
xticklabels={{{ {$\cT'$ and $\cT'+\cN'$}}},{{{$\cT'$ and $\hcT'$}}},{{{$\cT$ and $\hcT$} }}},
xticklabel style={yshift=-3 pt},
xlabel style={font=\color{white!15!black},yshift=-2 pt},
xlabel={{ Distance between}},
ymode=log,
ymin=0.000457686147334849,
ymax=0.1,
yminorticks=true,
ylabel={{{Mean distance}}},
axis background/.style={fill=white},
title style={font=\bfseries},
title={Initial SNR = 20},
axis x line*=bottom,
axis y line*=left,
legend style={at={(0.02,0.03)}, anchor=south west, legend cell align=left, align=left, fill=none, draw=none, {nodes={scale=0.5, transform shape}} }
]
\addplot [color=mycolor1, line width=1.5pt, mark=triangle, mark options={solid, mycolor1}]
  table[row sep=crcr]{%
1	0.1\\
2	0.0170360418127112\\
3	0.00488899919075866\\
};
\addlegendentry{R = 4, I = 20}

\addplot [color=mycolor2, dashed, line width=1.5pt, mark=o, mark options={solid, mycolor2}]
  table[row sep=crcr]{%
1	0.1\\
2	0.00340308000322454\\
3	0.000457686147334849\\
};
\addlegendentry{R = 4, I = 100}

\addplot [color=mycolor3, dashdotted, line width=1.5pt, mark=square, mark options={solid, mycolor3}]
  table[row sep=crcr]{%
1	0.1\\
2	0.0269256662609456\\
3	0.0134925840525876\\
};
\addlegendentry{R = 10, I = 20}

\addplot [color=mycolor4, dotted, line width=1.5pt, mark=diamond, mark options={solid, mycolor4}]
  table[row sep=crcr]{%
1	0.1\\
2	0.00544607188071547\\
3	0.00122662144691036\\
};
\addlegendentry{R = 10, I = 100}

\end{axis}

\begin{axis}[%
width=2.152in,
height=1.584in,
at={(0.558in,0.7in)},
scale only axis,
xmin=1,
xmax=3,
xtick={1,2,3,4,5},
xticklabels={{{ {$\cT'$ and $\cT'+\cN'$}}},{{{ $\cT'$ and $\hcT'$}}},{{{ $\cT$ and $\hcT$} }}},
xticklabel style={yshift=-3 pt},
xlabel style={font=\color{white!15!black},yshift=-2 pt},
xlabel={{ Distance between}},
ymode=log,
ymin=0.1,
ymax=10,
yminorticks=true,
ylabel style={font=\color{white!15!black}},
ylabel={{{Mean distance}}},
axis background/.style={fill=white},
title style={font=\bfseries},
title={Initial SNR = -20},
axis x line*=bottom,
axis y line*=left,
legend style={at={(0.02,0.03)}, anchor=south west, legend cell align=left, align=left, fill=none, draw=none, {nodes={scale=0.5, transform shape}} }
]
\addplot [color=mycolor1, line width=1.5pt, mark=triangle, mark options={solid, mycolor1}]
  table[row sep=crcr]{%
1	10\\
2	2.75024169380598\\
3	2.09440752821611\\
};
\addlegendentry{R = 4, I = 20}

\addplot [color=mycolor2, dashed, line width=1.5pt, mark=o, mark options={solid, mycolor2}]
  table[row sep=crcr]{%
1	10\\
2	0.440162697903075\\
3	0.118487772943173\\
};
\addlegendentry{R = 4, I = 100}

\addplot [color=mycolor3, dashdotted, line width=1.5pt, mark=square, mark options={solid, mycolor3}]
  table[row sep=crcr]{%
1	10\\
2	4.11617747317861\\
3	3.69176336897849\\
};
\addlegendentry{R = 10, I = 20}

\addplot [color=mycolor4, dotted, line width=1.5pt, mark=diamond, mark options={solid, mycolor4}]
  table[row sep=crcr]{%
1	10\\
2	0.850023079198715\\
3	0.419945883161676\\
};
\addlegendentry{R = 10, I = 100}

\end{axis}

\begin{axis}[%
width=2.052in,
height=1.584in,
at={(4.027in,0.7in)},
scale only axis,
xmin=1,
xmax=3,
xtick={1,2,3,4,5},
xticklabels={{{ { $\cT'$ and $\cT'+\cN'$}}},{{{$\cT'$ and $\hcT'$}}},{{{ $\cT$ and $\hcT$} }}},
xticklabel style={yshift=-3 pt},
xlabel style={font=\color{white!15!black},yshift=-2 pt},
xlabel={{ Distance between}},
ymode=log,
ymin=4.808862906224,
ymax=100,
yminorticks=true,
ylabel style={font=\color{white!15!black}},
ylabel={{{ Mean distance}}},
axis background/.style={fill=white},
title style={font=\bfseries},
title={Initial SNR = -40},
axis x line*=bottom,
axis y line*=left,
legend style={at={(0.02,0.03)}, anchor=south west, legend cell align=left, align=left, fill=none, draw=none, {nodes={scale=0.5, transform shape}} }
]
\addplot [color=mycolor1, line width=1.5pt, mark=triangle, mark options={solid, mycolor1}]
  table[row sep=crcr]{%
1	100\\
2	26.2275944175073\\
3	25.5793324727459\\
};
\addlegendentry{R = 4, I = 20}

\addplot [color=mycolor2, dashed, line width=1.5pt, mark=o, mark options={solid, mycolor2}]
  table[row sep=crcr]{%
1	100\\
2	5.52810035048934\\
3	4.808862906224\\
};
\addlegendentry{R = 4, I = 100}

\addplot [color=mycolor3, dashdotted, line width=1.5pt, mark=square, mark options={solid, mycolor3}]
  table[row sep=crcr]{%
1	100\\
2	40.4208784839827\\
3	39.9807467864287\\
};
\addlegendentry{R = 10, I = 20}

\addplot [color=mycolor4, dotted, line width=1.5pt, mark=diamond, mark options={solid, mycolor4}]
  table[row sep=crcr]{%
1	100\\
2	8.65118262534741\\
3	8.13881154778399\\
};
\addlegendentry{R = 10, I = 100}

\end{axis}

\begin{axis}[%
width=5.833in,
height= 0in,
at={(0in,5.375in)},
scale only axis,
clip=false,
xmin=0,
xmax=1,
ymin=0,
ymax=1,
axis line style={draw=none},
ticks=none,
axis x line*=bottom,
axis y line*=left
]
\node[align=center]
at (axis cs:0.5,-0.05) {Relative Distance between $\cT$ and measured tensor, between $\cT$ \\ and approximation $\hcT$, and between Procrustes cores};
\end{axis}
\end{tikzpicture}%
}
		\label{figure:TPEcombo}
	\caption{Distance between $\cT'$ and $\cT'+\cN'$, between $\cT'$ and $\hcT'$, and between $\cT$ and $\hcT$. The mean distance between $\cT$ and $\hcT$ is observed to be significantly smaller than that between $\cT'$ and $\cT'+\cN'$, and the difference between these distances increases as the ratio of $I$ to $R$ increases. The mean distance between $\cT$ and $\hcT$ is smaller than that between $\cT'$ and $\hcT'$; however, the order of magnitude of the gap depends on the initial SNR of $\cT'$ to $\cN'$ with smaller gaps when the initial SNR is lower.}
\end{figure}

\section{Conclusion}
We have given a deterministic bound on the radius of an open Frobenius norm ball centered at a simple rank $R$ tensor in which best rank $R$ approximations are guaranteed to exist and in which every rank $R$ tensor has a unique CPD. This illustrates the existence of a neighborhood of ``mathematical truth" around the signal portion of tensors occurring in application and gives a quantitative explanation for the numerical success of the CPD in practical settings. In addition we solved a tensor Procrustes problem which provides a method for computing orthogonal compressions for a pair of tensors.

The bound for existence of best rank $R$ approximations was computed by connecting the CPD to a JGE value decomposition for a tuple of matrices. Several basic results were established for JGE values, including an examination of algebraic and geometric multiplicities. We showed that a slice mix invertible tensor of size $R \times R \times K$ which has  border $\K$-rank $R$ and $\K$-rank strictly greater than $R$ must be defective in the sense of algebraic and geometric multiplicities for JGE values. Furthermore, we developed perturbation theoretic bounds which may be used to guarantee that tensors in a neighborhood of a given simple rank $R$ tensor are simple, thus giving a neighborhood in which every tensor has a best rank $R$ approximation.

\section{Acknowledgements} The authors thank the anonymous referees for their comments which helped improve clarity and exposition in the article.

\renewcommand\thesection{\ifcase\value{section}S\ or S.1\else S\fi}

\section{Supplementary Proofs}

The supplementary materials give the proofs of Theorem \ref{theorem:MLSVDExistenceBound}, Lemma \ref{lemma:ProcrustesIsometries}, and Theorem \ref{theorem:TensBauerFike} in the general case.

\subsection{Proof of Theorem \ref{theorem:MLSVDExistenceBound}}

As previously mentioned, the proof of Theorem \ref{theorem:MLSVDExistenceBound} is nearly identical to that of Theorem \ref{theorem:MultiplePencilBound}, with all notable differences occurring at the start of the proof. 

\begin{proof} Let $\cM',\cW',\cW$ and $R$ and $K$ be as in the statement of the theorem. Assume there exists some $\K$-rank $R$ tensor $\tcT'$ such that
\[
\|\cM-\tcT'\|_\mathrm{F} < \epsilon - \|\cM' - \cW'\|_\mathrm{F}
\]
and let $\hcT'$ be a best border $\K$-rank $R$ approximation of $\cM'$ in the Frobenius norm. Since $\tcT'$ has rank $R$ by assumption, we must have
\[
\|\cM'-\hcT' \|_\mathrm{F} \leq \|\cM'-\tcT'\|_\mathrm{F}  < \epsilon - \|\cM' - \cW'\|_\mathrm{F},
\]
hence
\[
\|\cW'-\hcT'\|_\mathrm{F} < \epsilon. 
\]

 Note that projecting the mode three fibers of $\hcT'$ to the subspace spanned by the mode three fibers of $\cM'$ neither increases the Frobenius distance between $\hcT'$ and $\cM'$ nor the rank of $\hcT'$, thus we can assume $R_3(\hcT') \leq R_3 (\cM')$. Recalling that $K=\min\{(\cM'),R\}$, since $\hcT'$ has border $\K$-rank $R$, the multilinear rank of $\hcT'$ is at most $(R,R,K)$. As consequence of Theorem \ref{theorem:TensorProcrustes}, there exists an orthogonal compression $\hcT \in \K^{R \times R \times K}$ of $\hcT'$ such that
\[
\|\cW-\hcT\|_\mathrm{F} \leq \|\cW'-\hcT'\|_\mathrm{F}  < \epsilon
\]
To show $\hcT'$ has $\K$-rank $R$ is sufficient to show that $\hcT$ has $\K$-rank $R$.

Now let $\bU \in \K^{K \times K}$ be a unitary and set $\cS=\cW\cdot_3 \bU$ and $\hcS=\hcT\cdot_3 \bU$. Since modal unitary multiplications preserve the Frobenius norm of a tensor we have $\|\cS- \hcS\|_\mathrm{F}  = \| \cW -\hcT \|_\mathrm{F} < \epsilon $. Therefore, for some $i=1,\dots, \lfloor K/2 \rfloor$ we must have that
\[
\|(\bS_{2i-1}, \bS_{2i})-({ \hat{\bS}}_{2i-1}, { \hat{\bS}}_{2i})\|_\mathrm{F} < \epsilon_i,
\]
as otherwise would imply
\[
\| \cS -  \hcS\|_\mathrm{F}^2 \geq  \sum_{k=1}^{\lfloor K/2 \rfloor} \|(\bS_{2k-1},\bS_{2k})-({\hat{\bS}}_{2k-1},{\hat{\bS}}_{2k})\|^2_\mathrm{F} \geq   \sum_{k=1}^{\lfloor K/2 \rfloor} \epsilon_k ^2=  \epsilon^2 .
\]

In particular, from our choice of $\epsilon_i$ the subpencil $({\hat{\bS}}_{2i-1},{\hat{\bS}}_{2i})$ must be  simple. Using Proposition \ref{proposition:DefectiveSubpencils} allows us to conclude that the tensor ${\hcT}$ is simple. It follows from Theorem \ref{theorem:rankvsmult} \eqref{it:NonderogatoryImpliesRankR} that $\hcT$ has rank $R$. 

The fact that $\hcT$ has a unique CPD is then a straightforward consequence of Kruskal's condition \cite{K77}. In addition, the fact that the set of tensors $\cM'$ which satisfy the assumptions of the theorem but do not have a unique best rank $R$ approximation has measure zero is an immediate consequence of \cite{QML19} when working over $\C$ and of \cite[Corollary 18]{FO14} when working over $\R$.
\end{proof}
	
\subsection{Proof of Lemma \ref{lemma:ProcrustesIsometries}}


\begin{proof}
Consider the orthogonal Procrustes problem: 
\beq
\label{eq:Procrustes}
\min_{\Psi \in O(I_1)} \|\bG- \Psi \bH\|_\mathrm{F}.
\eeq
where $O(I_1)$ denotes the group of $I_1 \times I_1$ unitary matrices. \cite{S66} shows that $\bU  \in O(I_1)$ is a minimizer of \eqref{eq:Procrustes} if and only if $\bU$ can be written $\bU=\bV \bZ^\mathrm{H}$ where $\bV ,\bZ \in O( I_1)$ are unitary matrices such that 
\beq
\label{eq:ProcrustesSVD}
\bG \bH^\mathrm{H}=\bV \mathbf{\Sigma} \bZ^\mathrm{H}
\eeq
gives a singular value decomposition for $\bG \bH^\mathrm{H}$ for an appropriate choice of $\mathbf{\Sigma}$.

It is straightforward to show that the unitary matrices $\bV$ and $\bZ$ in equation \eqref{eq:ProcrustesSVD} may be chosen so that
\[
\bV = (\bV_1 \ \bV_2)  \mathand \bZ = (\bZ_1 \ \bZ_2)
\]
where $\bV_1,{ \bZ_1} \in \K^{I_1 \times R_\bG}$ and $\bV_2,\bZ_2 \in \K^{I_1 \times (I_1-R_\bG)}$ are column-wise orthonormal matrices satisfying
\[
\ran \bV_1 = \ran \bG \mathand \ran \bZ_1 \supseteq  \ran\bH
\]
\[
\ran \bV_2  = (\ran \bG)^\perp \mathand \ran \bZ_2 \subseteq (\ran \bH)^\perp.
\]
Using this choice of $\bV$ and $\bZ$ we have $\bZ_2^\mathrm{H} \bH=\b0$, hence
\[
\bV \bZ^\mathrm{H} \bH= (\bV_1 \ \bV_2  ) (\bZ_1 \ \bZ_2)^\mathrm{H} \bH= \bV_1 \bZ_1^\mathrm{H} \bH.
\]
It follows that $\ran \bM \supseteq \ran \bV \bZ^\mathrm{H} \bN$ which completes the first part of the proof.

To prove the second claim let $\bU_\bM \in \R^{I_1 \times R_\bG}$ be any column-wise orthonormal matrix mapping onto $\ran \bG$, and set $\bU_\bN = \bZ \bV^\mathrm{H} \bU_\bM$ with $\bV$ and $\bZ$ as above. Then we have
\[
\|\bU_\bM^\mathrm{H} { \bM}- \bU_\bN^\mathrm{H} \bN \|_\mathrm{F} \leq \| \bU_\bM^\mathrm{H} \|_2 \| \bM-  \bV\bZ^\mathrm{H} \bN\|_\mathrm{F}  \leq \|\bM-\bN\|_\mathrm{F},
\]
To see that $\ran \bU_\bN \supseteq \ran \bN$, we use $\bV_2^\mathrm{H} \bM=\b0$ to find
\[
\bZ \bV^\mathrm{H} \bU_\bM = (\bZ_1 \ \bZ_2 ) (\bV_1 \ \bV_2)^\mathrm{H} \bU_\bM=\bZ_1 \bV_1^\mathrm{H} \bU_\bM.
\]
Since $\bU_\bM \in \K^{I_1 \times R_\bG}$ and  $\bV_1 \in \R^{I_1 \times R_\bG}$ are column-wise orthonormal matrices with the same range, it follows that $\bV_1^{\mathrm{H}} \bU_\bM$ is an invertible map on $\K^{R_\bM}$. Therefore
\[
\ran \bU_\bN= \ran \bZ_1 \bV_{1}^\mathrm{H} \bU_\bM = \ran  \bZ_1 \supseteq \ran \bN
\]
as claimed. \ \end{proof}

\subsection{Proof of Theorem \ref{theorem:TensBauerFike} in the general case}

We now give the proof of Theorem \ref{theorem:TensBauerFike} in the general setting. For clarity, the difference between the proof we now give and the proof in the main article is that we now only assume that $\cW$ has some JGE value with nonzero geometric multiplicity rather than assuming that $\cW$ has $\K$-rank $R$ and is slice mix invertible. This more general setting for example  be applies to slice mix invertible $R \times R \times K$ tensors which have border $\K$-rank $R$ but rank strictly greater than $R$.

\begin{proof}

We begin by proving the first inequality in equation \eqref{eq:JGEvaluebound}. To this end, let $\scrW \subset \K^K$ be a JGE value of $\cW$ and let $\bomega \in \scrW$ and $\bgam \in \scrW^\perp$ be vectors with unit two norm. Let $\bx \in \K^R$ be a JGE vector of $\cW$ corresponding to $\scrW$ and let $\by \in \K^R$ be a vector such that the tuple $(\cW,\bomega,\bx,\by)$ satisfies equation (2.8). Set $\bv = \frac{\bB^{\mathrm{T}}\bx}{\|\bB^{\mathrm{T}}\bx\|_2}$ so that $\|\bv\|_2=1$. Additionally, for $k=1,\dots,K$ let $\bE_k$ denote the $k$th frontal slice of $\cE$. 

With this setup we have
\[
\b0= \langle \bomega,\bgam \rangle \bA^{-1} \by = \sum_{k=1}^K \bar{\gamma}_k \bA^{-1} \bW_k \bB^{-\mathrm{T}} \bv = \sum_{k=1}^K \bar{\gamma}_k D_k (\bC) \bv + \sum_{k=1}^K \bar{\gamma}_k \bA^{-1} \bE_k \bB^{-\mathrm{T}} \bv. 
\]
It follows that 
\[
\sum_{k=1}^K \bar{\gamma}_k D_k (\bC) \bv = - \sum_{k=1}^K \bar{\gamma}_k \bA^{-1} \bE_k \bB^{-\mathrm{T}} \bv. 
\]
Furthermore we have
\begin{align}
\label{eq:JGEspNormBound} \max_{\substack{\bgam \in \scrW^\perp \\ \|\bgam \|_2=1}} \Big\|\sum_{k=1}^K \bar{\gamma}_k D_k (\bC) \bv\Big\|_2 
&\leq 
\max_{\substack{\|\bu\|_2=1 \\ \|\bz\|_2=1 }} \max_{\substack{\bgam \in \scrW^\perp \\ \|\bgam \|_2=1 }} \Big| \sum_{k=1}^K \bar{\gamma}_k  \bu^\mathrm{T}\bA^{-1} \bE_k \bB^{-\mathrm{T}} \bz \Big| 
\\ \nonumber
& =  \max_{\substack{\|\bw\|_2=1 \\ \|\bz\|_2=1  \\ \|\bgam\|_2=1 }}  \Big| \Big( \cE \cdot_1 \bA^{-1} \cdot_2 \bB^{-1} \cdot_3 \bP_{\scrW^\perp} \Big) \cdot_1 \bu^\mathrm{T} \cdot_2 \bz^\mathrm{T} \cdot_3 \bgam^\mathrm{T} \Big| 
\\ \nonumber
& = \Big\| \cE \cdot_1 \bA^{-1} \cdot_2 \bB^{-1} \cdot_3 \bP_{\scrW^\perp} \Big\|_{\mathrm{sp}}.
\end{align}
In the $K=2$ case, one may then follow the proof of \cite[Chapter VI, Theorem 2.7]{SS90} to obtain
\[
\min_k \chi (\scrL_k, \scrW) \leq \Big\| \cE \cdot_1 \bA^{-1} \cdot_2 \bB^{-1} \cdot_3 \bP_{\scrW^\perp} \Big\|_{\mathrm{sp}}.
\]

We next show that there exists some JGE value $\scrL$ of $\cT$ such that
\beq
\label{eq:JGEvalueIntermediateBound}
\frac{\chi(\scrL,\scrW)}{\sqrt{R}} \leq \max_{\substack{\bgam \in \scrW^\perp \\ \|\bgam \|_2=1}}  \bigg\|\sum_{k=1}^K \bbgam_k D_k (\bC) \bv \bigg\|_2,
\eeq
from which the first inequality in equation \eqref{eq:JGEvaluebound} will follow. For $r=1, \dots, R$, let $\bc_r= (c_{1 r},\dots, c_{K r})$ be the $r$th column of $\bC$. Then $c_{k r}$ is the $r$th diagonal entry of $D_k (\bC)$ for each $k$, and, as shown by Theorem \ref{theorem:rankvsmult}, the set $\{\spann (\bc_r)\}_{r=1}^R$ is the spectrum of $\cT$. By way of equation \eqref{eq:ChordalEqs}, to  prove that equation \eqref{eq:JGEvalueIntermediateBound} holds it is sufficient to show that there is an index $r_*$ such that
\beq
\label{eq:svBoundGoal}
 \max_{\substack{\bgam \in \scrW^\perp \\ \|\bgam \|_2=1}}\frac{ | \langle \bc_{r_*}, \bgam \rangle | }{\sqrt{R}} \leq \max_{\substack{\bgam \in \scrW^\perp \\ \|\bgam \|_2=1}} \bigg\|\sum_{k=1}^K \bbgam_k D_k (\bC) \bv \bigg\|_2 .
\eeq

To this end let $\be_r$ be the $r$th standard basis vector in $\K^R$ and observe
\beq
\label{eq:BoundOnBasisVec}
\sum_{k=1}^K \bar{\gamma}_k D_k (\bC) \be_r   =\sum_{k=1}^K \bar{\gamma}_k c_{kr} \be_r =   \left(\sum_{k=1}^K \bar{\gamma}_k c_{kr}\right) \be_r  = \langle  \bc_r, \bgam \rangle \be_r. 
\eeq
Decompose $\bv=\sum_{r=1}^R v_r \be_r$ and note that $\|\bv\|_2=1$ implies $\sum_{r=1}^R |v_r|^2=1$. Then for each fixed $\bgam \in \scrW^\perp$ with $\|\bgam\|_2=1$ we have
\[
\begin{array}{rcccc}

\|\sum_{k=1}^K \bar{\gamma}_k D_k (\bC) \bv\|_2 &=&  \left\| \sum_{k=1}^K \bar{\gamma}_k D_k (\bC) \left(\sum_{r=1}^R v_r \be_r \right) \right\|_2 & = & \left\| \sum_{r=1}^R v_r \left( \sum_{k=1}^K \bar{\gamma}_k D_k (\bC) \be_r\right) \right\|_2 \\
&=& \|\sum_{r=1}^R v_r \langle \bc_r, \bgam \rangle \be_r \|_2 &=& \sqrt{\sum_{r=1}^R |v_r|^2 | \langle  \bc_r, \bgam \rangle|^2}.
\end{array}
\]

Now let $r_*$ be an index such that $|v_{r_*}|\geq |v_r|$ for all $r=1,\dots, R$. Since $\|v\|_2=1$, we must have $|v_{r_*}|\geq \frac{1}{\sqrt{R}}$. Thus we have
\[
\bigg\|\sum_{k=1}^K \bar{\gamma}_k D_k (\bC) \bv\bigg\|_2=\sqrt{\sum_{r=1}^R |v_r|^2 | \langle  \bc_r, \bgam \rangle|^2}\geq |v_{r_*} | |\langle \bc_{r_*},\bgam \rangle | \geq \frac{|\langle \bc_{r_*}, \bgam \rangle |}{\sqrt{R}}.
\]
Noting that the above inequality holds for all $\bgam \in \scrW^\perp$ with $\|\bgam\|_2=1$ gives equation \eqref{eq:svBoundGoal} from which equation \eqref{eq:JGEvalueIntermediateBound} follows. The desired inequality then follows by combining inequalities \eqref{eq:JGEspNormBound} and \eqref{eq:JGEvalueIntermediateBound}.

To prove the second inequality in equation \eqref{eq:JGEvaluebound} observe that
\[
\spnorm{\cE \cdot_1 \bA^{-1} \cdot_2 \bB^{-1} \cdot_3 \bP_{\scrW^\perp}}
 \leq   \|\bA^{-1}\|_2 \|\bB^{-1}\|_2 \max_{\substack{\bgam \in \scrW^\perp \\ \|\bgam \|_2=1 }} \left\|\sum _{k=1}^K \bar{\gamma}_k \bE_k \right\|_2    \leq  \frac{ \spnorm{{ \cE}}}{\sigma_{\min} (\bA) \sigma_{\min} (\bB)}
\]
from which the result follows.

It remains to show that the coefficient $\sqrt{R}$ is not necessary in the special case where $\cW$ is slice mix invertible with $\K$-rank $R$ and CPD $[\![\bA_\cW,\bB_\cW,\bC_\cW]\!]$ where $\bA_\cW=\bA$ or $\bB_\cW=\bB$. It is straightforward to check that the spectral variation $\sv [\cT,\cW]$ is invariant under a simultaneous permutation of the first and second mode of $\cT$ and $\cW$, so it is sufficient to consider the case $\bB_\cW=\bB$. 

In this case, Theorem \ref{theorem:rankvsmult} shows that the JGE vectors of both $\cT$ and $\cW$ are given by the columns of $\bB^{-\mathrm{T}}$, hence the vector $\bv$ appearing the proof above is simply equal to $\be_r$ for some index $r \in 1,\dots,R$.  Using equation \eqref{eq:JGEspNormBound} together with equation \eqref{eq:BoundOnBasisVec} then shows that we have 
\[
\chi(\scrL_r,\scrW)=\max_{\substack{\bgam \in \scrW^\perp \\ \|\bgam \|_2=1}} |\langle \bc_r, \bgam\rangle| \leq  \Big\| \cE \cdot_1 \bA^{-1} \cdot_2 \bB^{-1} \cdot_3 \bP_{\scrW^\perp} \Big\|_{\mathrm{sp}}  \leq   \frac{ \spnorm{{ \cE}}}{\sigma_{\min} (\bA) \sigma_{\min} (\bB)},
\]
as claimed.
\end{proof}


\bibliographystyle{siamplain}
\bibliography{references}

\vfill

\end{document}